\documentclass[11pt, letterpaper, leqno]{article}
\usepackage{amsmath,amsthm,amscd,amssymb,amsbsy,cite,amstext,mathtools}
\usepackage{amssymb}
\usepackage{graphicx}
\usepackage{cite}
\usepackage{mathrsfs}
\usepackage{xcolor}

\usepackage{algorithm,algorithmic}
\usepackage{pgf}
\usepackage{color}
\usepackage{enumerate}
\usepackage{xcolor}
\usepackage{mathrsfs}
\usepackage[symbol]{footmisc}
\usepackage[colorlinks=true, pdfstartview=FitV, linkcolor=blue, citecolor=orange, urlcolor=blue]{hyperref}
\date{}
\usepackage{lineno}

\usepackage{amsfonts}
\usepackage[T1]{fontenc}
\usepackage{euler}

\theoremstyle{plain}
\newtheorem{theorem}{Theorem}[section]
\newtheorem{lemma}{Lemma}[section]

\newtheorem{claim}{Claim}[section]
\newtheorem{problem}{Problem}

\theoremstyle{remark}
\newtheorem{remark}{Remark}[section]
\newtheorem*{acknowledgment}{Acknowledgment}

\theoremstyle{definition}

\newtheorem{definition}{Definition}[section]

\numberwithin{equation}{section}
\newcommand{\eqindent}{\displayindent0pt\displaywidth\textwidth}


\newcounter{parentnumber}

\newcommand{\LA}[1]{\refstepcounter{equation}\text{(\theequation)}\label{#1}}

\numberwithin{equation}{section}

\newcommand{\supp}[1]{\mathrm{supp}\left( {#1} \right)}
\newcommand{\norm}[1]{\| {#1}\| }
\newcommand{\diam}{\mathrm{diam}\,}
\newcommand{\set}[1]{\left\{#1\right\}}
\newcommand{\void}{\varnothing}
\newcommand{\abs}[1]{\left|#1\right|}
\newcommand{\dist}[2]{\mathrm{dist}\left({#1}, {#2}\right)}
\newcommand{\eq}[1]{\eqref{#1}}

\newcommand{\jet}{\mathscr{J}}
\newcommand{\Q}{\mathcal{Q}}
\newcommand{\R}{\mathbb{R}}
\newcommand{\brac}[1]{\left(#1\right)}
\renewcommand{\d}{\partial}
\newcommand{\Rn}{\mathbb{R}^n}
\newcommand{\grad}{\nabla}
\newcommand{\qt}[1]{``\hspace{1sp}#1\,''}

\newcommand{\itau}{{[-\tau,\tau]}}
\newcommand{\E}{\mathcal{E}}
\newcommand{\ct}{C^2}
\newcommand{\ctet}{{C^2(E,\tau)}}
\newcommand{\ctrn}{{C^2(\Rn)}}
\newcommand{\ctrt}{C^2(\Rn,\tau)}
\renewcommand{\P}{\mathcal{P}}
\newcommand{\pos}{[0,\infty)}

\newcommand{\for}{\enskip\text{for}\enskip}
\renewcommand{\forall}{\enskip\text{for all }\enskip}
\newcommand{\dq}{\delta_Q}

\newcommand{\Ls}{CZ^\sharp}
\newcommand{\Lz}{CZ^0}
\newcommand{\Le}{{CZ^{\rm empty}}}
\newcommand{\Lsk}{CZ^{\sharp\sharp}}
\newcommand{\xqs}{{x_Q^\sharp}}

\newcommand{\ul}[1]{\underline{#1}}
\newcommand{\A}{\ul{\mathcal{A}}}
\newcommand{\depth}{\mathrm{depth}}
\renewcommand{\phi}{\varphi}
\newcommand{\vp}{\vec{P}}
\newcommand{\da}{\d^\alpha}

\newcommand{\sk}{\sigma^\sharp}
\newcommand{\G}{\Gamma_\tau}
\newcommand{\Gk}{\Gamma_\tau^\sharp}
\newcommand{\tma}{{2-\abs{\alpha}}}
\newcommand{\ksh}{{k^\sharp}}
\newcommand{\T}{\mathscr{T}}

\begin{document}

	\title{$C^2$ Interpolation with Range Restriction}
	\author{Charles Fefferman \and Fushuai Jiang \and Garving~K. Luli}

	\maketitle
	{\begin{center}	\today\end{center}}
	\begin{abstract}
		Given $ -\infty< \lambda < \Lambda < \infty $,  $ E \subset \mathbb{R}^n $ finite, and $ f : E \to [\lambda,\Lambda] $, how can we extend $ f $ to a $ C^m(\mathbb{R}^n) $ function $ F $ such that $ \lambda\leq F \leq \Lambda $ and $ \norm{F}_{C^m(\mathbb{R}^n)} $ is within a constant multiple of the least possible, with the constant depending only on $ m $ and $ n $? In this paper, we provide the solution to the problem for the case $ m = 2 $. Specifically, we construct a (parameter-dependent, nonlinear) $ C^2(\mathbb{R}^n) $ extension operator that preserves the range $[\lambda,\Lambda]$, and we provide an efficient algorithm to compute such an extension using $ O(N\log N) $ operations, where $ N = \#(E) $. 
	\end{abstract}
	
	\tableofcontents
	
	\section{Introduction}
	
	For integers $ m\geq 0,n \geq 1 $, we write $ C^m(\Rn) $ to denote the Banach space of $ m $-times continuously differentiable real-valued functions such that the following norm is finite
	\begin{equation*}
		\norm{F}_{C^m(\Rn)} := \sup\limits_{x \in \Rn}\max_{\abs{\alpha} \leq m}\abs{ \partial^\alpha F(x) }.
	\end{equation*}
	We use $ C(m,n) $, $ k(m,n) $, etc., to denote controlled constants that depend only on $ m $ and $ n $. If $ E $ is a finite set of $\Rn$, we write $ \#E $ to denote the number of elements in $ E $. 
	
	We consider the following interpolation problem with lower and upper bounds $\lambda, \Lambda,$ respectively.
	\begin{problem}\label{prob.tau-nonsymmetric}
		Let $ E \subset \Rn $ be a finite set. Let $ -\infty < \lambda< \Lambda <\infty $. Let $ f : E \to [\lambda,\Lambda] $. Compute a function $ F: \Rn \to [\lambda,\Lambda] $ such that 
		\begin{enumerate}[(A)]
			\item $ F = f $ on $ E $, and
			\item $ \norm{F}_{C^m(\Rn)} \leq C(m,n)\cdot \inf \set{\norm{\tilde{F}}_{C^m(\Rn)} : \tilde{F} = f
				\text{ on }E\text{, and }\lambda \leq\tilde{F}\leq \Lambda} $.
		\end{enumerate}
	\end{problem}
	
	By \qt{computing a function $F$} from $ (E,\lambda,\Lambda,f) $, we mean the following: After processing the input $ (E,\lambda,\Lambda,f) $, we are able to accept a query consisting of a point $ x \in \Rn $, and produce a list of numbers $ (f_\alpha(x) : \abs{\alpha} \leq m) $. The algorithm \qt{computes the function $ F $} if for each $ x \in \Rn $, we have $ \d^\alpha F(x) =f_\alpha(x) $ for $ \abs{\alpha} \leq m $. 
	
	We call the function $ F $ in Problem \ref{prob.tau-nonsymmetric} a $ C $-optimal interpolant of $ f $ (see Condition (B)).

	Problem \ref{prob.tau-nonsymmetric} is closely related to a common theme in data visualization, where one wants to present some given three-dimensional data as a surface or a contour map. Moreover, one may want to preserve some crucial inherent properties of the data, such as nonnegativity or convexity. This occurs when the data arises as some physical quantities and we want to preserve the physical meaning of the interpolant. For instance, nonnegative constraint is natural when the data represents a probability distribution or (absolute) temperature. More generally, it is sometimes desirable to impose both upper and lower bounds on the interpolants, commonly referred to as \qt{range-restricted interpolants}. See e.g., \cite{range1,range2,range3,range4,range5,range6}. These problems arise, for example, when the predicted trajectory must avoid collision with prescribed obstacles. We refer the readers to the aforementioned references for further background and related topics on range-restricted interpolation.

	By letting $ \tau := \frac{\Lambda - \lambda}{2} $ and replacing $ f $ by $ f - \frac{\Lambda-\lambda}{2} $, we see that Problem \ref{prob.tau-nonsymmetric} admits the following equivalent formulation.

	\begin{problem}\label{prob.tau}
		Let $ E \subset \Rn $ be a finite set. Let $ \tau > 0 $. Let $ f : E \to \itau $. Compute a function $ F: \Rn \to \itau $ such that 
		\begin{enumerate}[(A)]
			\item $ F = f $ on $ E $, and
			\item $ \norm{F}_{C^m(\Rn)} \leq C(m,n)\cdot \inf \set{\norm{\tilde{F}}_{C^m(\Rn)} : \tilde{F} = f
				\text{ on }E\text{, and }-\tau \leq\tilde{F}\leq \tau} $.
		\end{enumerate}
	\end{problem}

	Formally letting $ \tau \to \infty $, we recover the classical Whitney Interpolation Problem.
	
	\begin{problem}[Classical Whitney Interpolation Problem]\label{prob.Whitney}
		Let $ E \subset \Rn $ be a finite set. Let $ f : E \to \R $. Compute a function $ F: \Rn \to \R $ such that 
		\begin{enumerate}[(A)]
			\item $ F = f $ on $ E $, and
			\item $ \norm{F}_{C^m(\Rn)} \leq C(m,n)\cdot \inf \set{\norm{\tilde{F}}_{C^m(\Rn)} : \tilde{F} = f
				\text{ on }E} $.
		\end{enumerate}
	\end{problem}

	Problems \ref{prob.tau-nonsymmetric}--\ref{prob.Whitney} for the cases $m =0,1$ can be immediately solved by the classical Whitney's Extension Theorem (see Theorem \ref{thm.WT}(B) below). Some alternatives for the case $m = 0$ include Urysohn's Lemma and Kirszbraun's formula.

	The classical Whitney Interpolation Problem \ref{prob.Whitney} is well-understood thanks to the works of Brudnyi and Shvartsman \cite{BS98,BS01,BS94-Tr}, Fefferman and Klartag\cite{FK09-Data-1,FK09-Data-2,F09-Data-3,F05-Sh,F05-L}. In \cite{FK09-Data-1,FK09-Data-2}, the authors provide an efficient algorithm for solving the classical Whitney Interpolation Problem \ref{prob.Whitney}. Their algorithm pre-processes the set $ E $ using at most $ C(m,n)N\log N $ operations (on a van Neumann machine that can operate with exact numbers) and $ C(m,n)N $ storage with $ N = \#E $. Then, after reading $ f $, the algorithm is ready to answer queries. A query consists of a point $ x \in \Rn $, and the answer to a query is the $m$-th order Taylor polynomial of an interpolant $ F $ with the least norm up to a constant factor $ C(m,n) $. The number of operations to answer a query is $ C(m,n)\log N $. The complexity of the Fefferman-Klartag algorithm is most likely the best possible. 
	
	Problem \ref{prob.tau} (or equivalently Problem \ref{prob.tau-nonsymmetric}) and the classical Whitney Interpolation Problem \ref{prob.Whitney} are special instances of the following smooth selection problem. 
	
	\begin{problem}\label{prob.selection}
		Let $ E \subset \Rn $ be finite. For each $ x \in E $, let $ K(x) \subset \R^d $ be convex. Find a function $ \vec{F} = (F_1, \cdots, F_d) : \Rn \to \R^d $ such that $ F(x) \in K(x) $ and $ \norm{\vec{F}}_{C^m(\Rn,\R^d)} $ as small as possible, up to a constant factor depending only on $ m $, $ n $, and $ d $.
	\end{problem}
	
	If we specialize $ K(x) \subset \R $ (hence $ d = 1 $) in Problem \ref{prob.selection} to a singleton, we obtain the classical Whitney Interpolation Problem \ref{prob.Whitney}.
	
	Here we note a subtle but crucial difference between Problems \ref{prob.tau-nonsymmetric}, \ref{prob.tau} and Problem \ref{prob.selection} (with $ d = 1 $ and $ K(x) $ being a fixed compact interval for each $ x \in E $). The lower and upper bounds, $ [\lambda, \Lambda] $ or $ \itau $, for $ F $ are global in Problem \ref{prob.tau-nonsymmetric} or \ref{prob.tau}. On the other hand, these bounds are only imposed on the set $ E $ in Problem \ref{prob.selection}.

	Problem \ref{prob.selection} and the related "Finiteness Principles" (see e.g. Theorem \ref{thm.sfp} below) have been extensively studied by Y. Brudnyi and P. Shvartsman \cite{BS94-W,BS01}, C. Fefferman, A. Israel, and G.K. Luli\cite{FIL16}, C. Fefferman and P. Shvartsman \cite{FShv18}.

	In this paper, inspired by \cite{FK09-Data-1,FK09-Data-2,F09-Data-3}, building on the work of \cite{JL20,JL20-Ext,JL20-Alg}, we solve Problem \ref{prob.tau} for the case $ n \in \mathbb{N} $ and $ m = 2 $.

	To facilitate the discussion on algorithms, we introduce the following concepts.

	\begin{definition}\label{def.depth}
		\newcommand{\bn}{{N_0}}
		\newcommand{\rbn}{\R^\bn}
		Let $ \bn \geq 1 $ be an integer. Let $ B = \set{\xi_1, \cdots, \xi_\bn} $ be a basis of $ \rbn $. Let $ \Omega \subset \rbn $ be a subset. Let $ X $ be a set. Let $ \Xi: \Omega \to X $ be a map.
		\begin{itemize}
			\item We say $ \Xi $ \underline{has depth $D$} (with respect to the basis $ B $) if there exists a $ D $-dimensional subspace $ V = {\rm span}\brac{\xi_{i_1}, \cdots, \xi_{i_D}} $, $ \xi_{i_1}, \cdots, \xi_{i_D} \in B $, such that for all $ z_1, z_2 \in\Omega $ with $ \pi_V(z_1) = \pi_V(z_2) $, we have $ \Xi(z_1) = \Xi(z_2) $. Here, $ \pi_V : \rbn \to V $ is the natural projection. We call the set of indices $ \set{i_1, \cdots, i_D} $ the \underline{source} of $ \Xi $. 
			\item Suppose $ \Xi $ has depth $ D $. Let $  V = {\rm span}\brac{\xi_{i_1}, \cdots, \xi_{i_D}} $ be as above. By an \underline{efficient representation} of $ \Xi $, we mean a specification of the index set $ \set{i_1, \cdots, i_D} \subset \set{1, \cdots, \bn} $ and a map $\tilde{\Xi}:\Omega\cap V \to X$ agreeing with $\Xi$ on $\Omega \cap V$, such that given $v \in \Omega\cap V$, $\tilde{\Xi}(v)$ can be computed using at most $C_D$ operations. Here, $C_D$ is a constant depending only on $D$.
		\end{itemize}
	\end{definition}

	\begin{remark}\label{rem.depth}
		
		Suppose $\Xi:\R^{\bar{N}}\to \R$ is a linear functional. Recall from \cite{FK09-Data-2} that a \qt{compact representation} of a linear functional $\Xi:\R^{\bar{N}}\to \R$ consists of a list of indices $ \set{i_1, \cdots, i_D} \subset \set{1, \cdots, \bar{N}} $ and a list of coefficients $ \chi_{i_1}, \cdots, \chi_{i_D} $, so that the action of $\Xi$ is characterized by
		\begin{equation*}
			\Xi : (\xi_1, \cdots, \xi_{\bar{N}}) \mapsto \sum_{\Delta = 1}^{D}\chi_{i_\Delta} \cdot \xi_{i_\Delta}.
		\end{equation*}
		Therefore, given $v \in \mathrm{span}(\xi_{i_1},\cdots,\xi_{i_D})$, we can compute $\Xi(v)$ by the dot product of two vectors of length $D$, which requires $C_D$ operations. The present notion of \qt{efficient representation} is a natural generalization of the \qt{compact representation} in \cite{FK09-Data-2} adapted to the nonlinear nature of constrained interpolation (see also \cite{JL20,JL20-Ext,JL20-Alg}). 
		
	\end{remark}

	We write $ \ctet $ to denote the collection of functions $ f : E \to \itau $, which can be identified with $ \itau^{\#E} $. We define 
	\begin{equation*}
		\norm{f}_\ctet := \inf\set{\norm{F}_\ctrn : F = f \text{ on } E \text{ and } -\tau \leq F \leq \tau}.
	\end{equation*}

	Our first main theorem is the following. 
	
	\begin{theorem}\label{thm.bd-op}
		\renewcommand{\E}{{\mathcal{E}_\tau}}
		Let $ n $ be a positive integer. Let $ \tau > 0 $. Let $ E \subset \Rn $ be a finite set. There exist controlled constants $C(n), D(n)$, and a map $ \E :   \ctet \times \pos \to \ctrn  $ such that the following hold. 
		
		\begin{enumerate}[(A)]
			\item Let $ M \geq 0 $. Then for all $ f \in \ctet $ with $ \norm{f}_{\ctet} \leq M $, we have $ \E(f,M) \in \ctrt $, $ \E(f,M) = f $ on $ E $, and 
			$ \norm{\E(f,M)}_{\ct(\Rn)} \leq CM $.
			
			\item For each $ x \in \Rn $, there exists a set $ S(x) \subset E $, independent of $ \tau $, with $ \# S(x) \leq D $, such that for all $ M \geq 0 $ and $ f, g \in \ctet $ with $ f|_{S(x)} = g|_{S(x)} $, we have
			\begin{equation*}
				\d^\alpha \E(f,M)(x) = \d^\alpha \E(g,M)(x)
				\for
				\abs{\alpha} \leq 2\,.
			\end{equation*}
			
		\end{enumerate}
		
	\end{theorem}
	
	In fact, we can prove a stronger version of Theorem \ref{thm.bd-op}. Let $ \P^+ $ denote the vector space of polynomials on $ \Rn $ with degree no greater than two, and let $ \jet_x^+F $ denote the two-jet of $ F $ at $ x $.
	
	\begin{theorem}\label{thm.bd-alg}
		Let $ n $ be a positive integer. Let $ \tau > 0 $. Let $ E \subset \Rn $ be a finite set with $ \#(E) = N $. Then there exists a collection of maps $ \set{\Xi_{\tau,x} : \tau \in \pos,\,x \in \Rn} $, where 
		\begin{equation*}
			\Xi_{\tau,x} : \ctet \times \pos \to \P^+
		\end{equation*}
		for each $ x \in \Rn $, such that the following hold.
		\begin{enumerate}[(A)]
			\item There exists a controlled constant $ D(n) $ such that for each $ x \in \Rn $, the map $ \Xi_{\tau,x}(\cdot\,,\cdot) : \ctet\times \pos \to \P^+ $ is of depth $ D $. Here, we view $\ctet\subset \R^{\#E}$ with the standard basis. Moreover, the source of $ \Xi_{\tau,x} $ is independent of $ \tau $. 
			
			\item Suppose we are given $ (f,M) \in \ctet \times \pos $ with $ \norm{f}_{\ctet} \leq M $.
			Then there exists a function $ F \in \ctrt $ such that
			\begin{equation*}
				\begin{split}
					\jet^+_x F = \Xi_{\tau,x}(f,M) \text{ for all } x \in \Rn,\,
					\norm{F}_{\ctrn} \leq C\cdot M,
					\text{ and }
					F(x) = f(x)
					\for
					x \in E.
				\end{split}
			\end{equation*}
			Here, $ C $ depends only on $ n $.

			\item There is an algorithm that takes the given data set $E$, performs one-time work, and then responds to queries.
			
			A query consists of a pair $ (\tau,x) \in\pos\times \Rn $, and the response to the query is the depth-$ D $ map $ \Xi_{\tau,x} $, given in its efficient representation.
			
			The one-time work takes $ C_1N\log N $ operations and $ C_2N $ storage. The work to answer a query is $ C_3\log N $. Here, $ C_1, C_2, C_3 $ depend only on $ n $. 
		\end{enumerate}
	\end{theorem}

	We briefly explain the strategy for the proofs of Theorem \ref{thm.bd-op} and Theorem \ref{thm.bd-alg}, sacrificing accuracy for the ease of understanding. 
	
	We will prove Theorem \ref{thm.bd-op} and Theorem \ref{thm.bd-alg} by inducting on the dimension $ n $. The base case for the induction is given by
	
	\begin{theorem}\label{thm.bd-alg-1d}
		Theorems \ref{thm.bd-op} and \ref{thm.bd-alg} are true for $ n = 1 $. 
	\end{theorem}
	
	Assume the validity of Theorems \ref{thm.bd-op} and \ref{thm.bd-alg} for $ n -1 $. Let $ E \subset \Rn $ be a finite set. We perform a Calder\'on-Zygmund decomposition of $ \Rn $ into dyadic cubes $ \set{Q : Q \in \Lambda} $, such that near each $ Q $, $ E $ lies on a hypersurface with mean curvature bounded by $ C\dq^{-1} $, where $ \dq $ is the sidelength of $ Q $. As such, $ E $ can be locally straightened to lie within a hyperplane by a $ C^2 $-diffeomorphism, and the local interpolation problem is readily solvable by the induction hypothesis. We then construct the global interpolation map by patching together these local interpolation maps via a  partition of unity. To avoid large derivatives caused by the partition functions supported on small cubes, we introduce a collection of \qt{transition jets} that guarantee Whitney compatibility among neighboring cubes, and construct our local interpolants in accordance with these transition jets. 
	
	We have given an overly simplified account of our strategy. In practice, we have to take great care to preserve the range restriction $-\tau \leq F\leq \tau$, and control the derivative contribution from hypersurface with large mean curvature

	The proof for Theorem \ref{thm.bd-alg-1d} will be given in Section \ref{sect:1d}. The proofs for Theorem \ref{thm.bd-op} and Theorem \ref{thm.bd-alg} will be presented in Sections \ref{sect:induction}--\ref{sect:main-proof}.

	Using Theorem \ref{thm.bd-alg}, together with the ``Well Separated Pairs Decomposition'' technique from computational geometry, we obtain the following.
	
	\begin{theorem}\label{thm.sfp}
		Let $ E \subset \Rn $ be a finite set with $ \# E = N < \infty $. Then there exist controlled constants $ C_1, \cdots, C_5 $, depending only on $ n $, and a list of subsets $ S_1, \cdots, S_L \subset E $ satisfying the following.
		\begin{enumerate}[(A)]
			\item We can compute the list $ \set{S_\ell : \ell = 1, \cdots, L} $ from $ E $ using one-time work of at most $ C_1N\log N $ operations and using storage at most $ C_2N $.
			
			\item $ \#S_\ell \leq C_3 $ for each $ \ell = 1, \cdots, L $.
			\item $ L \leq C_4N $.
			\item Given any $ \tau > 0 $ and $ f : E \to\itau $, we have
			\begin{equation*}
				\max_{1 \leq \ell \leq L} \norm{f}_{\ct(S_\ell,\tau)} \leq \norm{f}_\ctet \leq C_5\max_{1 \leq \ell \leq L}\norm{f}_{\ct(S_\ell,\tau)}.
			\end{equation*}
		\end{enumerate}
	\end{theorem}

	The proof for Theorem \ref{thm.sfp} will be given in Section \ref{sect:sfp}.

	The above theoretical results allow us to produce efficient algorithms to solve Problem \ref{prob.tau} (or Problem \ref{prob.tau-nonsymmetric}) in the case $m=2$. In this paper, we content ourselves with an idealized computer with standard von Neumann
	architecture that is able to process exact real numbers. We refer the readers to \cite{FK09-Data-2} for discussion
	on finite-precision computing.

	Theorem \ref{thm.bd-alg} guarantees the existence of the following algorithm. 
	
	\begin{algorithm}[H]
		\caption{Algorithm for $ \ctrn $ Interpolation with range restriction}\label{alg.interpolant}
		\begin{itemize}
			\item[] \textbf{DATA:} $ E \subset \Rn $ finite with $ \#E = N $. $ \tau > 0 $. $ f: E \to \itau $. $ M \geq 0 $. 
			\item[] \textbf{ORACLE:} $ \norm{f}_{\ctet} \leq M $. 
			\item[] \textbf{RESULT}: A query function that accepts a point $ x \in \Rn $ and produces a list of numbers $ ( f_\alpha(x) : \abs{\alpha} \leq 2) $ that guarantees the following: There exists a function $ F \in \ctrt $ with $ \norm{F}_{\ctrn} \leq CM$ and $ F|_E = f $, such that $ \d^\alpha F(x) = f_\alpha(x) $ for $ \abs{\alpha} \leq 2 $. The function $ F $ is independent of the query point $ x $, and is uniquely determined by $ (E,\tau,f,M) $.
			
			\item[] \textbf{COMPLEXITY:} 
			
			\begin{itemize}
				\item Preprocessing $ (E,\tau) $: at most $ CN\log N $ operations and $ CN $ storage.
				
				\item Processing $f$: $CN$ operations and $CN$ storage.
				
				\item Answering a query: at most $ C\log N $ operations.
			\end{itemize}
			
		\end{itemize}

	\end{algorithm}

	%

	Theorem \ref{thm.sfp} guarantees the existence of the following algorithm for approximating $ \norm{f}_\ctet $. 
	
	\begin{algorithm}[H]
		\caption{Algorithm for approximate $ \ctrn $ norm with range restriction}\label{alg.norm}
		
		\begin{itemize}
			\item[] \textbf{DATA:} $ E \subset \Rn $ finite with $ \#E = N $. $ \tau > 0 $.
			\item[] \textbf{QUERY:} $f: E \to 
			\itau $.
			\item[] \textbf{RESULT}: The order of magnitude of $ \norm{f}_{\ctet} $. More precisely, the algorithm outputs a number $ M \geq 0 $ such that both of the following hold.
			\begin{itemize}
				\item We guarantee the existence of a function $ F \in \ctrt $ such that $ F|_E = f $ and $ \norm{F}_{\ctrn} \leq CM $.
				\item We guarantee there exists no $ F \in \ctrt $ with norm at most $ C^{-1}M $ satisfying $ F|_E = f $.
			\end{itemize}
			
			\item[] \textbf{COMPLEXITY:} 
			\begin{itemize}
				\item Preprocessing $ E $: at most $ CN\log N $ operations and $ CN $ storage.
				\item answer a query: at most $ CN $ operations.
			\end{itemize} 
		\end{itemize}
		
	\end{algorithm}

	Finally, we mention that the techniques developed in this paper can readily be adapted to treat $C^2(\Rn)$ nonnegative interpolation; for comparison, see \cite{JL20,JL20-Ext,JL20-Alg}.

	This paper is a part of a literature on extension, interpolation, and selection of functions, going back to H. Whitney's seminal works\cite{W34-1,W34-2,W34-3}, and including fundamental contributions by G. Glaeser\cite{G58}, Y. Brudnyi and P. Shvartsman \cite{BS85,BS94-W,BS94-Tr,BS97,BS98,BS01,Shv82,Shv84,Shv86,Shv87,Shv90,Shv01,Shv02,Shv04,Shv08}, and E.~Bierstone, P.~Milman, and W.~Paw{\l}ucki\cite{BMP03,BMP06,BM07}, as well as our own papers\cite{F05-Sh,F05-L,F05-J,F06,F07-L,F09-Data-3,FK09-Data-1,FK09-Data-2,FL04,JL20,JL20-Alg,JL20-Ext}. See e.g. \cite{F09-Int} for the history of the problem, as well as N. Zobin \cite{Z98,Z99} for a related problem.

	\begin{acknowledgment}
	    We are indebted to Jes\'us A. De Loera, Kevin O'Neill, Naoki Saito for their valuable comments. We also thank all the participants in the 11th Whitney workshop for fruitful discussions, and Trinity College Dublin for hosting the workshop. 
	
		The first author is supported by the Air Force Office of Scientific Research (AFOSR), under award FA9550-18-1-0069, the National Science Foundation (NSF), under grant DMS-1700180, and the US-Israel Binational Science Foundation (BSF), under grant 2014055. The second author is supported by the Alice Leung Scholarship, the UC Davis Summer Graduate Student Researcher (GSR) Award and the Yuel-Jing Lin Fund. The third author is supported by NSF Grant DMS-1554733 and the UC Davis Chancellor's Fellowship.
	\end{acknowledgment}

	\section{Notations}

	We use $ c_*, C_*, C' $, etc., to denote constants depending only on $ n $, referred to as \qt{controlled constants}. They may be different quantities in different occurrences. We will label them to avoid confusion when necessary.
	
	Let $M$ and $M'$ be two nonnegative quantities determined by $E$, $f$, and $n$. We say that $M$ and $M'$ have the \underline{same order of magnitude} provided that there exists a controlled constant $C(n)$ such that $C^{-1}M\leq M' \leq CM$. In this case we write $M \approx M'$. To compute the order of magnitude of $M'$ is to compute a number $M$ such that $M\approx M'$.

	We assume that we are given an ordered orthogonal coordinate system on $ \Rn $, specified by an ordered list of unit vectors $ [e_1,\cdots,e_n] $. We use $ \abs{\,\cdot\,} $ to denote Euclidean distance. We use $ B(x,r) $ to denote the ball of radius $ r $ centered at $ x $. For $ X, Y \subset \Rn $, we write $ \dist{X}{Y} := \inf_{x \in X, y \in Y}\abs{x - y} $. 
	
	We use $ \alpha = (\alpha_1,\cdots, \alpha_n),\beta = (\beta_1, \cdots, \beta_n) \in \mathbb{N}_0^n $, etc., to denote multi-indices. We write $ \d^\alpha $ to denote $ \d_{e_1}^{\alpha_1}\cdots\d_{e_n}^{\alpha_n} $. We adopt the partial ordering $ \alpha \leq \beta $ if and only if $ \alpha_i \leq \beta_i $ for $ i = 1,\cdots,n $. 
	
	By a cube, we mean a set of the form $ Q = \prod_{i = 1}^n[a_i, a_i+\delta) $ for some $ a_1, \cdots, a_n \in \R $ and $ \delta > 0 $. If $ Q $ is a cube, we write $ \dq $ to denote its sidelength. For $ r > 0 $, we use $ r Q $ to denote the cube whose center is that of $ Q $ and whose sidelength is $ r\dq $. Given two cubes $ Q, Q' $, we write $ Q \leftrightarrow Q' $ if either $Q = Q'$, or if $ closure(Q) \cap closure(Q') \neq \void $ and $interior(Q) \cap interior(Q') = \void$.
	
	A dyadic cube is a cube of the form $ Q = \prod_{i = 1}^{n} [2^k \cdot p_i, 2^k\cdot (p_i + 1) ) $ for some $ p_1, \cdots, p_n, k \in \mathbb{Z} $. Each dyadic cube $ Q $ is contained in a unique dyadic cube with sidelength $ 2\dq $, denoted by $ Q^+ $.

	Let $ n \geq 1 $. Let $ X $ be the diffeomorphic image of a cube or all of $ \Rn $. We use $ C^2(X) $ to denote the vector space of twice continuously differentiable real-valued functions up to the closure of $ X $, whose derivatives up to order two are bounded. Let $ X_0 $ be the interior of $ X $. For $ F \in C^2(X) $, we define
	\begin{equation*}
		\norm{F}_{C^2(X)} := \sup_{x \in X_0} \max_{\abs{\alpha}\leq 2}\abs{\d^\alpha F(x)}.
	\end{equation*}
	We write $ C^2(X,\tau) $ to denote the collection of functions $ F \in C^2(X) $ such that $ -\tau \leq F \leq \tau $ on $ X $. 
	
	Let $ E \subset \Rn $ be finite. We define the following.
	\begin{equation*}
		\begin{split}
			C^2(E) := \set{f:E \to \R} \cong \R^{\#(E)}
			\enskip&\text{ and }\enskip
			\norm{f}_{C^2(E)}  := \inf\set{\norm{F}_{C^2(\Rn)} : F |_E = f};\\
			C^2(E,\tau) := \set{f:E \to \itau} \cong \itau^{\#(E)}
			\enskip&\text{ and }\enskip
			\norm{f}_{C^2(E,\tau)} := \inf\set{\norm{F}_{C^2(\Rn)} : F|_E = f
				\text{ and }-\tau \leq F \leq \tau}\,.
		\end{split}
	\end{equation*}
	
	We write $ \P $ and $ \P^+ $, respectively, to denote the vector spaces of polynomials with degree no greater than one and two.
	
	\newcommand{\ring}{\mathcal{R}}
	For $ x \in \R^n $ and a function $ F $ twice continuously differentiable at $ x $, we write $ \jet_x F$ and $\jet_x^+F $ to denote the one-jet and two-jet of $ F $ at $ x $, respectively, which we identify with the degree-one and degree-two Taylor polynomials, respectively:
	\begin{equation}
		\begin{split}
			\jet_x F(y) := \sum_{\abs{\alpha} \leq 1}\frac{\d^\alpha F(x)}{\alpha !}(y-x)^\alpha \quad\text{ and }\quad
			\jet_x^+ F(y) := \sum_{\abs{\alpha} \leq 2}\frac{\d^\alpha F(x)}{\alpha !}(y-x)^\alpha.
		\end{split}
		\label{eq.jet-def}
	\end{equation}
	We use $ \ring_x$, $\ring^+_x $ to denote the rings of one-jets, two-jets at $ x $, respectively. The multiplications on $ \ring_x $ and $ \ring_x^+ $ are defined in the following way:
	\begin{equation*}
		P\odot_x R :\equiv \jet_x(PR)
		\quad\text{ and }\quad
		P^+\odot_x^+R^+:\equiv \jet_x^+(P^+R^+),
	\end{equation*}
	for $ P,R \in \ring_x $ and $ P^+, R^+ \in \ring_x^+ $.

	\newcommand{\ws}{{W(S)}}
	\begin{definition}\label{def.WF}
		Let $ S \subset \Rn $ be a finite set. A \underline{Whitney field} on $ S $ is an array of polynomials $ \vec{P} = (P^x)_{x\in S} $ parameterized by points in $ S $, such that $ P^x \in \P $ for each $ x \in S $. The collection of Whitney fields will be denoted by $ W(S) $. It is a finite dimensional vector space equipped with a norm
		\begin{equation*}
			\norm{(P^x)_{x \in S}}_{W(S)} := \max_{x, y \in S, x \neq y, \abs{\alpha} \leq 1}\set{ \abs{\da P^x(x)}, \frac{\abs{\da (P^x - P^y)(x)}}{\abs{x-y}^\tma} }.
		\end{equation*}
	\end{definition}
	It is a vector space of dimension $ (\#S )\cdot \dim \P $. 
	
	Given a Whitney field $ \vec{P} = (P^x)_{x\in S} $, we sometimes use the notation
	\begin{equation*}
		(\vec{{P}},x) := P^x.
	\end{equation*}

	\section{Essential polynomials}
	\label{sect:polynomials}

	\subsection{Jets with range restriction}

	The following object captures the effect of the range restriction on jets.
	
	\newcommand{\K}{\mathcal{K}}
	
	\begin{definition}\label{def.Ktau}
		Let $ \tau > 0 $. Let $ x \in \Rn $ and $ M \geq 0 $. We define $ \K_\tau(x,M) $ to be the collection of polynomials $ P \in \P $ such that
		\begin{align}
			&\abs{P(x)} \leq \min\set{M,\tau}\text{, } \abs{\grad P} \leq M \text{, and }\label{Ktau-def-1}\\
			&\abs{\grad P} \leq M^{1/2}\cdot\min\set{\sqrt{\tau-P(x)}, \sqrt{\tau+P(x)}}\label{Ktau-def-2}
		\end{align}
	\end{definition}

	\begin{lemma}\label{lem.dist}
		Let $ P \in \P $ and $ x \in \Rn $ be such that $ -\tau \leq P(x) \leq \tau $. Let $ \mu:= \frac{\abs{\grad P}^2}{\brac{\min\set{\sqrt{\tau - P(x)},\sqrt{\tau + P(x)}}}^{2}}$. Then 
		\begin{equation}
			\dist{\set{P = P(x)}}{\set{P = \pm\tau}} \geq \mu^{-1/2}\cdot \min\set{\sqrt{\tau - P(x)},\sqrt{\tau+P(x)}}.
			\label{dist-1}
		\end{equation}
		In particular, given $ P \in \K_\tau(x,M) $ as in Definition \ref{def.Ktau}, we have 
		\begin{equation}
			\dist{\set{P = P(x)}}{\set{P = \pm\tau}} \geq M^{-1/2}\cdot \min\set{\sqrt{\tau - P(x)},\sqrt{\tau+P(x)}}.
			\label{dist-3}
		\end{equation}
	\end{lemma}

	\begin{proof}
		Suppose $ P(x) = \pm\tau $ or $ \grad P = 0 $. Then \eqref{dist-1} obviously holds.
		
		Suppose $ -\tau < P(x) < \tau $ and $ \grad P \neq 0 $. Since $ P $ is an affine function, the level sets of $ P $ are parallel hyperplanes. We have
		\begin{equation}
			\frac{\tau + P(x)}{\dist{\set{P = P(x)}}{\set{P = -\tau}}} = \abs{\grad(\tau+P)} = \abs{\grad P} = \abs{\grad (\tau - P)} = \frac{\tau - P(x)}{\dist{\set{P = P(x)}}{\set{P = \tau}}}.
			\label{dist-2}
		\end{equation}
		Combining the definition of $ \mu $ and \eq{dist-2}, we see that \eq{dist-1} holds. 
		
		Since $ \mu \leq M $ for any given $ P \in \K_\tau(x,M) $ as in Definition \ref{def.Ktau}, we see that \eqref{dist-3} follows.
	\end{proof}

	\begin{lemma}\label{lem.Ktau}
		There exists a controlled constant $ C(n) $ such that the following holds. Let $ x_0 \in \Rn $.
		\begin{enumerate}[(A)]
			\item Suppose there exists $ F \in \ctrt $ with $ \norm{F}_\ctrn \leq M $. Then $ \jet_{x_0} F \in \K_\tau(x_0,CM) $.
			\item There exists a map $ \T_*^{x_0}: \bigcup_{M\geq 0}\K_\tau(x_0,M) \to \ctrt $ such that the following holds. Suppose $ P \in \K_\tau(x_0,M) $. Then $ \T_*^{x_0}(P) $ satisfies $ \norm{\T_*^{x_0}(P)}_\ctrn \leq CM $ and $ \jet_{x_0}\T_*^{x_0}(P)  \equiv P $.
		\end{enumerate}
	\end{lemma}
	
	\begin{proof}
		We write $ C $, $ c $, etc., to denote constants that depend only on $ n $. 
		
		Without loss of generality, we may assume that $ x_0 = 0 \in \Rn $. 
		
		\subparagraph{(A).}
		
		Let $ F \in \ctrt $ with $ \norm{F}_{\ctrn} \leq M $. Let $ P :\equiv \jet_{0}F $. By Taylor's theorem, 
		\begin{equation}
			-CM\abs{x}^2 \leq F(x) - P(x) \leq CM\abs{x}^2
			\forall x \in \Rn.
			\label{Ktau-1}
		\end{equation}
		Since $ F \in \ctrt $ and $ \norm{F}_{\ctrn} \leq M $, we have
		\begin{equation}\label{Ktau-2}
			\abs{F(x)} \leq \min\set{M,\tau}\forall x \in \Rn.
		\end{equation}
		Combining \eqref{Ktau-1} and \eqref{Ktau-2}, we see that 
		\begin{align}
			\brac{\tau + P(0)} + \grad P \cdot x + CM\abs{x}^2 &\geq 0\forall x \in \Rn\label{Ktau-3}\text{, and }\\
			\brac{\tau -P(0)} - \grad P \cdot x + CM\abs{x}^2 
			&\geq 0\forall x \in \Rn\label{Ktau-4}.
		\end{align}
		By restricting to each line and computing the discriminant in both \eqref{Ktau-3} and \eqref{Ktau-4}, we see that $ \abs{\grad P} \leq CM^{1/2}\cdot\min\set{\sqrt{\tau - P(0)},\sqrt{\tau + P(0)}} $. Hence, $ P \in \K_\tau(0,CM) $. 
		
		\subparagraph{(B).}
		Let $ P \in \K_\tau(x_0,M) $ be given. We write $ \T_* $ instead of $ \T_*^0 $. 
		
		Suppose $ P(0) = \pm \tau $. Then property \eqref{Ktau-def-2} of $ \K_\tau $ in Definition \ref{def.Ktau} implies that $ P\equiv \pm\tau $. We may simply take $ \T_*(P):\equiv \pm\tau $. 
		
		Suppose $ -\tau < P(0) < \tau $. We define the following quantities.
		\begin{itemize}
			\item $ \mu:= \abs{\grad P}^2\cdot\brac{\min\set{
					\sqrt{\tau-P(0)},\sqrt{\tau+P(0)}
				}
			}^{-2} $. It is clear from \eqref{Ktau-def-2} that $ \mu \leq M $. 
			\item $ \delta:= \mu^{-1/2}\cdot\min\set{\sqrt{\tau - P(0)}, \sqrt{\tau + P(0)}} $.
		\end{itemize}
		We note that the definitions of $ \mu $ and $ \delta $ depend only on the polynomial $ P $; in particular, they are independent of $ M $.
		
		Let $ c_0 \in (0,1) $ be a small universal constant. For concreteness, we can take $ c_0 = \frac{1}{100} $. We consider the following two cases.
		\begin{enumerate}[\text{Case }I.]
			\item Either $ \mu \geq c_0\tau $ or $ \abs{P(0)} \geq c_0\tau $. 
			\item Both $ \mu < c_0\tau $ and $ \abs{P(0)} < c_0\tau $. 
		\end{enumerate}

		\textbf{Proof for Case I.} \quad Since $ M \geq \max\set{\mu,\abs{P(0)}} $, we see that
		\begin{equation}
			\tau \leq c_0^{-1}M.
			\label{Ktau-0}
		\end{equation}
		
		For convenience, we set
		\begin{equation*}
			u := \frac{\grad P}{\abs{\grad P}}.
		\end{equation*}
		
		Consider the auxiliary functions
		\begin{equation*}
			R^-(x) := P(x) + \frac{\mu}{4}\abs{x\cdot u}^2
			\quad\text{ and }\quad
			R^+(x):= P(x)-\frac{\mu}{4}\abs{x\cdot u}^2.
		\end{equation*} 
		Note that the graphs of $ R^- $ and $ R^+ $ are parabolic cylinders that are constant along each direction orthogonal to $ u $. By construction,
		\begin{equation}
			\jet_{0}R^- \equiv \jet_{0}R^+ \equiv P.
			\label{Ktau-9}
		\end{equation}
		We see from the definition of $ \mu $ that
		\begin{equation}
			R^-(x) \geq -\tau\text{, }R^+(x) \leq \tau\forall x \in \Rn.
			\label{Ktau-8}
		\end{equation}
		By computing the root along the $ u $-direction, we also have
		\begin{equation}\label{Ktau-10}
			\begin{split}
				R^-(x) \leq \tau &\text{ for } 0 \geq x\cdot u \geq -2(\sqrt{2}+1)\delta\text{, and }\\
				R^+(x)\geq -\tau &\text{ for } 0 \leq x\cdot u \leq 2(\sqrt{2}+1)\delta.
			\end{split}
		\end{equation}
		
		Consider the following regions in $ \Rn $.
		\begin{itemize}
			
			\item $ A_{0,-\tau} = \set{x\in \Rn: -\delta \leq x\cdot u \leq 0} $ and $ A_{0,\tau} = \set{x \in \Rn : 0 \leq x \cdot u \leq \delta} $. It follows from construction that $-\tau \leq P(x) \leq P(0)$ on $A_{0,-\tau}$ and $P(0) \leq P(x) \leq \tau$ on $A_{0,\tau}$.
			
			\item $ A_{1,-\tau} := \set{x \in \Rn: -2(\sqrt{2}+1)\delta \leq x\cdot u \leq -\delta} $ and $ A_\tau^{1}:= \set{x \in \Rn: \delta \leq x\cdot u \leq 2(\sqrt{2}+1)\delta } $.
			\item $ A_{2,-\tau}:= \set{x \in \Rn : x\cdot u \leq -2(\sqrt{2}+1)\delta} $ and $ A_\tau^{2}:= \set{x \in \Rn: x \cdot u \geq 2(\sqrt{2}+1)\delta} $.
		\end{itemize}

		We define a $ C^2 $ partition of unity $ \set{\theta^{[0]}, \theta_{-\tau}^{[1]}, \theta_{-\tau}^{[2]}, \theta_{\tau}^{[1]}, \theta_\tau^{[2]}} $ with the following properties.
		\begin{enumerate}[($ \theta $1)]
			\item $ \theta^{[0]} + \theta_{-\tau}^{[1]} + \theta_\tau^{[1]} + \theta_{-\tau}^{[2]} + \theta_\tau^{[2]} \equiv 1 $ on $ \Rn $. 
			\item $ 0 \leq \theta^{[0]} \leq 1 $, $ \supp{\theta^{[0]}} \subset A_{0,-\tau}\cup A_{0,\tau} $, $ \theta^{[0]} \equiv 1 $ near $ A_{0,-\tau} \cap A_{0,\tau} $, and $ \abs{\da\theta^{[0]}} \leq C\delta^{-\abs{\alpha}} $. 
			\item $ 0 \leq \theta_{-\tau}^{[1]} \leq 1 $, $ \supp{\theta_{-\tau}^{[1]}} \subset A_{0,-\tau}\cup A_{1,-\tau} $, $ \theta_{-\tau}^{[1]} \equiv 1 $ on $ A_{0,-\tau}\cap A_{1,-\tau} $, and
			$ \abs{\da\theta_{-\tau}^{[1]}} \leq C\delta^{-\abs{\alpha}} $.
			\item $ 0 \leq \theta_{\tau}^{[1]} \leq 1 $, $ \supp{\theta_{\tau}^{[1]}} \subset A_{0,\tau}\cup A_{1,\tau} $, $ \theta_\tau^{[1]} \equiv 1 $ on $ A_{0,\tau}\cap A_{1,\tau} $, and
			$ \abs{\da\theta_{\tau}^{[1]}} \leq C\delta^{-\abs{\alpha}} $.
			\item $ 0 \leq \theta_{-\tau}^{[2]} \leq 1 $, $ \supp{\theta_{-\tau}^{[2]}} \subset A_{1,-\tau}\cup A_{2,-\tau} $, $ \theta_{-\tau}^{[2]} \equiv 1 $ on $ A_{2,-\tau} $, and $ \abs{\da\theta_{-\tau}^{[2]}} \leq C\delta^{-\abs{\alpha}} $.
			\item $ 0 \leq \theta_{\tau}^{[2]} \leq 1 $, $ \supp{\theta_{\tau}^{[2]}} \subset A_{1,\tau}\cup A_{2,\tau} $, $ \theta_{\tau}^{[2]} \equiv 1 $ on $ A_{2,\tau} $, and $ \abs{\da\theta_{\tau}^{[2]}} \leq C\delta^{2-\abs{\alpha}} $.
		\end{enumerate}

		We define $ \T_*(P) \in \ctrn $ by
		\begin{equation*}
			\T_*(P):= \theta^{[0]} P + \theta_{-\tau}^{[1]}R^- -\tau\theta_{-\tau}^{[2]} +\theta_\tau^{[1]}R^+ + \tau\theta_{\tau}^{[2]}.
		\end{equation*}
		
		We see from ($ \theta $1--$ \theta $6), \eqref{Ktau-8}, and \eqref{Ktau-10} that $ -\tau \leq \T_*(P)(x) \leq \tau $ for all $ x \in \Rn $. Hence, $ \T_*(P) \in \ctrt $. Moreover, thanks to \eqref{Ktau-0}, we have
		\begin{equation*}
			\abs{\T_*(P)(x)} \leq CM\forall x \in \Rn.
		\end{equation*}
		
		We now estimate the derivatives of $ \T_*(P) $. First of all, for $ \alpha \neq 0 $, we have
		\begin{equation*}
			\da \T_*(P)(x) = \begin{cases}
				\sum_{\beta\leq \alpha}C_{\alpha,\beta}\cdot \d^\beta\theta^{[0]}(x)\cdot \d^{\alpha-\beta}(P - R^-)(x)&\for x \in A_{0,-\tau}\\
				\sum_{\beta\leq \alpha}C_{\alpha,\beta}\cdot \d^\beta\theta^{[0]}(x)\cdot \d^{\alpha-\beta}(P - R^+)(x)&\for x \in A_{0,\tau}\\
				\sum_{\beta\leq \alpha}C_{\alpha,\beta}\cdot \d^\beta\theta_{-\tau}^{[1]}(x)\cdot \d^{\alpha-\beta}(R^-(x) + \tau)&\for x \in A_{1,-\tau}\\
				\sum_{\beta\leq \alpha}C_{\alpha,\beta}\cdot \d^\beta\theta_{\tau}^{[1]}(x)\cdot \d^{\alpha-\beta}(R^+(x) - \tau)&\for x \in A_{1,-\tau}\\
				0&\for x \in A_{2,-\tau}\cup A_{2,\tau}
			\end{cases}\quad.
		\end{equation*}
		
		We analyze the first and third sums. The analysis for the second is similar to the first, and the fourth to the third.
		
		From the definitions of $ P $, $ \delta $, $ R^- $, and $ A_{0,-\tau} $, we see that 
		\begin{equation}
			\abs{\da(P - R^-)(x)} \leq C\mu\delta^{2-\abs{\alpha}} \leq CM\delta^{2-\abs{\alpha}}
			\for x \in A_{0,-\tau}\text{, }0 < \abs{\alpha} \leq 2.
			\label{Ktau-11}
		\end{equation}
		Combining \eqref{Ktau-11} and ($ \theta $2), we see that $ \abs{\da \T_*P(x)} \leq CM $ for $ x \in A_{0,-\tau} $, $ 0 < \abs{\alpha} \leq 2 $.  
		
		From the definitions of $ \delta $, $ R^- $, and $ A_{1,-\tau} $, we see that
		\begin{equation}\label{Ktau-12}
			\abs{\da(R^-(x)+\tau)} \leq C\mu\delta^{2-\abs{\alpha}} \leq C'M\delta^{2-\abs{\alpha}}
			\for x \in A_{1,-\tau}\text{, }0 < \abs{\alpha} \leq 2.
		\end{equation}
		Combining \eqref{Ktau-12} and ($ \theta $4), we see that $ \abs{\da \T_*P(x)} \leq CM $ for $ x \in A_{1,-\tau} $, $ 0 < \abs{\alpha} \leq 2 $.  
		
		Thus, we have shown that $ \norm{\T_*P}_{\ctrn} \leq CM $. This concludes our treatment of Case I.

		\textbf{Proof for Case II.}\quad Since both $ \mu < c_0\tau $ and $ \abs{P(0)} < c_0\tau $, there exists a universal constant $ c_1 $ such that
		\begin{equation*}
			\delta = \mu^{-1/2}\cdot\min\set{\sqrt{\tau-P(0)},\sqrt{\tau + P(0)}} \geq c_1.
		\end{equation*}
		We fix such $ c_1 $. Thanks to Lemma \ref{lem.dist}, we have
		\begin{equation}
			\dist{\set{P = P(0)}}{\set{P = \pm\tau}} \geq c_1.
			\label{Ktau-13}
		\end{equation}

		As before, we set $ u := \grad P /\abs{\grad P} $. Note that $ u $ is orthogonal to the level sets of $ P $. 
		Let $ \theta $ be a cutoff function such that $ \theta \equiv 1 $ near $ \set{P = P(0)} $, $ \supp{\theta} \subset \set{x \in \Rn: \abs{x\cdot u} \leq \frac{c_1}{2}} $, and $ \abs{\da\theta} \leq C $. We define
		\begin{equation*}
			\T_*P(x) := \theta(x)\cdot P(x).
		\end{equation*}
		
		Thanks to \eqref{Ktau-13}, we have $ \T_*P(x) \in \itau $ for all $ x \in \Rn $, so $ \T_*P \in \ctrt $. 
		
		By the fundamental theorem of calculus, we see that
		\begin{equation}
			\abs{P(x)} \leq \abs{P(0)} + \frac{c_1}{2}\abs{\grad P} \leq CM\forall x \in \supp{\theta}.
			\label{Ktau-14}
		\end{equation}
		
		From \eqref{Ktau-def-1} and \eqref{Ktau-14}, we have, for $ 0 < \abs{\alpha} \leq 2 $,
		\begin{equation}
			\abs{\da \T_*P(x)} \leq \sum_{0 < \beta \leq \alpha}\abs{C_{\alpha,\beta}\cdot\d^{\beta}\theta(x)\cdot \d^{\alpha - \beta}P(x)} \leq CM\forall x \in \supp{\theta}.
			\label{Ktau-15}
		\end{equation}
		
		Since $ \T_*P $ vanishes outside $ \supp{\theta} $, we can conclude from \eqref{Ktau-14} and \eqref{Ktau-15} that $ \norm{\T_*P}_{\ctrn} \leq CM $. This concludes the treatment of the second case and the proof of the lemma.

	\end{proof}

	\subsection{Whitney's Extension Theorem for finite set}

	\newcommand{\wst}{{W(S,\tau)}}

	Let $\tau > 0$. Recall $\K_\tau$ in Definition \ref{def.Ktau}. We use $ W(S,\tau) $ to denote the sub-collection of Whitney fields $ (P^x)_{x \in S} $ such that $ P^x \in \K_\tau(x,M) $ for some $ M \geq 0 $. We define
	\begin{equation}
		\norm{(P^x)_{x\in S}}_{\wst} := \norm{(P^x)_{x \in S}}_\ws + \inf \{M \geq  0 : P^x \in \K_\tau(x,M)\forall x \in S\}.
		\label{wst-def}
	\end{equation}

	Recall the following classical result.

	\begin{theorem}\label{thm.WT}
		There exists a controlled constant $ C(n) $ such that the following holds. Let $ S \subset \Rn $ be a finite set.
		\begin{enumerate}[(A)]
			\item \textbf{Taylor's Theorem.} Let $ F \in \ctrn $. Then $ \norm{(\jet_x F)_{x \in S}}_{W(S)} \leq C\norm{F}_\ctrn $.
			\item \textbf{Whitney's Extension Theorem.} There exists a linear map $ \T_w : W(S) \to \ctrn $ such that given any $ \vec{P} = (P^x)_{x \in S} \in W(S) $, we have $ \norm{\T_w(\vec{P})}_\ctrn \leq C\norm{\vec{P}}_{W(S)} $ and $ \jet_x \T_w(\vec{P}) \equiv P^x $ for each $ x \in S $. 
		\end{enumerate}
	\end{theorem}

	A proof of Theorem \ref{thm.WT} can be found in standard textbooks, see for instance \cite{St79}.

	\begin{theorem}\label{thm.WT-tau}
		There exists a controlled constant $ C(n) $ such that the following holds. Let $ S \subset \Rn $ be a finite set.
		\begin{enumerate}[(A)]
			\item \textbf{Taylor's Theorem for $ \ctrt $.} Let $ F \in \ctrt $. Then $ \norm{(\jet_x F)_{x \in S}}_{W(S,\tau)} \leq C\norm{F}_\ctrn $.
			\item \textbf{Whitney's Extension Theorem for $ \ctrt $.} There exists a map $ \T_{w,\tau} : W(S,\tau) \to \ctrt $ such that given any $ \vec{P} = (P^x)_{x \in S} \in W(S,\tau) $, we have $ \norm{\T_{w,\tau}(\vec{P})}_\ctrn \leq C\norm{\vec{P}}_{W(S,\tau)} $ and $ \jet_x \T_{w,\tau}(\vec{P}) \equiv P^x $ for each $ x \in S $. 
		\end{enumerate}
	\end{theorem}
	
	The proof of Theorem \ref{thm.WT-tau} is a straightforward combination of Lemma \ref{lem.Ktau} and the proof of Theorem \ref{thm.WT}.

	\subsection{Norms on small subsets}
	\label{sect:quad}

	Throughout this section, we fix a finite set 
	\begin{equation*}
		S \subset \Rn \quad ,\quad \#S\leq k_0,
	\end{equation*}
	where $ k_0 = k_0(n) $ is a controlled constant. 
	
	We define a norm $ \mathcal{L} $ on $ W(S) $ by
	\begin{equation}
		\begin{split}
			\mathcal{L} : W(S) &\to \pos \\
			(P^x)_{x \in S} &\mapsto \sum_{x \in S,\abs{\alpha} \leq 1}\abs{\da P(x)} + \sum_{x,y \in S, x \neq y, \abs{\alpha}\leq 1}\frac{\abs{\da(P^x - P^y)(x)}}{\abs{x-y}^\tma}\,.
		\end{split}
		\label{L-def}
	\end{equation}
	
	\begin{lemma}\label{lem.L}
		There exists a controlled constant $C$ such that given any $ \vp \in W(S) $, we have
		\begin{equation}
			C^{-1} \mathcal{L}(\vp) \leq \norm{\vp}_{W(S)} \leq C\mathcal{L}(\vp).
			\label{3.2.0}
		\end{equation}
	\end{lemma}
	
	\begin{proof}
		Recall Definition \ref{def.WF}. Recall the assumption that $ \#S \leq k_0 $ where $k_0$ is a controlled quantity. For a given $\vec{P} \in W(S)$, it is clear that $\norm{\vec{P}}_{W(S)} \leq \mathcal{L}(\vec{P})$. For the reverse inequality, we have
		\begin{equation*}
		    \begin{split}
		        \mathcal{L}(\vec{P}) 
		  &\leq \left[k_0(n+1) + \brac{k_0(k_0-1)(n+1)}\right] 
		    \cdot \max_{x, y \in S, x \neq y, \abs{\alpha} \leq 1}\set{ \abs{\da P^x(x)}, \frac{\abs{\da (P^x - P^y)(x)}}{\abs{x-y}^\tma} } \\
		    &\leq C(n)\norm{\vec{P}}_{W(S)}.
		    \end{split}
		\end{equation*}
		This proves \eqref{3.2.0}.
	\end{proof}
	
	For $ \tau > 0 $, we define a function that measures the effect of the range restriction on the optimal Whitney extension.

	\begin{equation}
		\begin{split}
			\mathcal{M}_\tau: W(S,\tau) &\to \pos\\
			(P^x)_{x \in S} &\mapsto \sum_{x \in S, \abs{\alpha} = 1}\frac{\abs{\da P^x}^2}{\tau - P(x)} + \frac{\abs{\da P^x}^2}{\tau + P(x)}\,.
		\end{split}
		\label{M-def}
	\end{equation}
	In \eqref{M-def}, we adopt the convention $ \frac{0}{0} = 0 $. 
	
	\begin{lemma}\label{lem.L+M}
		There exists a controlled constant $ C(n,k_0) $ such that given any $ \vec{P} \in W(S,\tau) $, we have
		\begin{equation}
			C^{-1}(\mathcal{L}+\mathcal{M}_\tau)(\vec{P}) \leq \norm{\vec{P}}_{W(S,\tau)} \leq C(\mathcal{L}+\mathcal{M}_\tau)(\vec{P}).
			\label{L+M-1}
		\end{equation}
	\end{lemma}

	\begin{proof}
		We adopt the notation $ A \approx B $ if there exists a controlled constant $ C(n) $ such that $ C^{-1}A \leq B \leq CA $.
		
		Comparing Definition \ref{def.Ktau} and \eqref{M-def}, we see that
		\begin{equation*}
			\mathcal{M}_\tau(\vp) \approx \inf \set{M \geq  0 : P^x \in \K_\tau(x,M)\forall x \in S}.
		\end{equation*}
		This together with Lemma \ref{lem.L} proves the equivalence in \eqref{L+M-1}. 
	\end{proof}
	
	\newcommand{\atf}{{\mathbb{A}_{\tau,f}}}

	We now explain how to compute the order of magnitude of $ \norm{f}_{\ct(S,\tau)} $. More precisely, using at most a bounded number of operations, we compute a Whitney field $  (P^x)_{x \in S} \in W(S,\tau) $ such that $ P^x(x) = f(x) $ for each $ x \in S $ and $ \norm{(P^x)_{x \in S}}_{W(S,\tau)} \approx \norm{f}_{\ct(S,\tau)} $.
	
	Consider the affine subspace $ \atf \subset W(S) $ defined by
	\begin{equation}
		\atf:= \set{\vec{P} = (P^x)_{x \in S} \in W(S) :\, \begin{matrix*}[l]
				P^x(x) = f(x)\for x \in S.\\
				\text{If }f(x) = \tau \text{ then } P^x \equiv \tau.\\
				\text{If }f(x) = -\tau \text{ then } P^x \equiv -\tau.
		\end{matrix*}}.
	\end{equation}
	Equivalently, 
	\begin{equation*}
		\atf = \set{\vec{P} = (P^x)_{x \in S} \in W(S,\tau) : P^x(x) = f(x)\for x \in S}. 
	\end{equation*}
	Note that $ \atf $ has dimension $ n\cdot \#(S\setminus f^{-1}(\set{\pm\tau})) $.
	
	Let $ \mathcal{L} $ and $ \mathcal{M}_\tau $ be as in \eqref{L-def} and \eqref{M-def}. Thanks to Whitney's Extension Theorem \ref{thm.WT-tau} and Lemma \ref{lem.L+M}, we have
	\begin{equation}
		\norm{f}_\ctet \approx \inf \set{(\mathcal{L}+\mathcal{M}_\tau)(\vec{P}) : \vec{P} \in \atf}.
		\label{quad-0}
	\end{equation}
	
	Let $ d := \dim W(S) = \#S \cdot \dim \P \leq k_0(n+1) $. We identify $ W(S) \cong \R^d $ via
	\begin{equation*}
		(P^x)_{x \in S} \mapsto \brac{P^x(x), \d_1P^x, \cdots, \d_n P^x}_{x \in S}.
	\end{equation*}
	We define the $ \ell_1 $ and $ \ell_2 $ norms, respectively, on $ \R^d $ by
	\begin{equation*}
		\norm{v}_{\ell_1} := \sum_{i = 1}^d\abs{v_i} \text{ and }\norm{v}_{\ell_2} := \brac{\sum_{i = 1}^d\abs{v_i}^2}^{1/2}, \quad v = (v_1, \cdots, v_d) \in \R^d.
	\end{equation*}

	\begin{itemize}
		
		\item Let $ L_w : W(S) \to \R^d $ be a linear isomorphism that maps $ \vp \in W(S) $ to a vector in $ \R^d $ with components
		\begin{equation*}
			\frac{\da(P^y - P^z)(y)}{\abs{y-z}^{2-\abs{\alpha}}}
			\quad,\quad \da P^{x_S}(x_S)\quad, \abs{\alpha} \leq 1
		\end{equation*}
		for suitable $ x_S,y,z \in S $ in certain order, such that
		\begin{equation}\label{quad.0}
			\eqindent
			\norm{L_w(\vp)}_{\ell^1(\R^d)} \approx \mathcal{L}(\vp) \quad \for \vp \in W(S).
		\end{equation}
		One possible construction of such an $ L_w $ is based on the technique of \qt{clustering} introduced in \cite{BM07}. See Remark 3.3 of \cite{BM07}. We can compute $ L_w $ from $ S $ using at most $ C $ operations, since $ \#S $ is controlled. 
		
		\item Let $ V_{\tau,f} \subset W(S) $ be the subset defined by
		\begin{equation*}
			V_{\tau,f} := \set{ (P^x)_{x \in S} : P^x(x) = 0 \for x \in S\setminus f^{-1}(\set{-\tau,\tau}) \text{ and } P^x \equiv 0 \for x \in f^{-1}(\set{-\tau,\tau}) }.
		\end{equation*}
		Let $ \Pi_{\tau,f} = (\Pi^x_{\tau,f})_{x \in S} : W(S)\to V_{\tau,f} $ be the orthogonal projection. Let $ \vp_{f} $ denote the vector $ (f(x),0,0)_{x \in S} $. It is clear that $ \atf = \vp_f + V_f $. 
		
		\item Let $ L_{\tau,f} = (L^x_{\tau,f})_{x\in S} : W(S) \to W(S) $ be a linear endomorphism defined by
		\begin{equation*}
			L^x_{\tau,f}(P^x) = \begin{cases}
				P^x\cdot \brac{\min\set{ \sqrt{\tau - f(x)}, \sqrt{\tau + f(x)} }}^{-1/2} &\for x \in S\setminus f^{-1}(\set{-\tau,\tau})
				\\
				0 &\for x \in f^{-1}(\set{-\tau,\tau})
			\end{cases},
		\end{equation*}
		for $(P^x)_{x \in S} \in W(S).$	
		
	\end{itemize}

	We see from the definition of $ \mathcal{M}_\tau $ that
	\begin{equation}\label{quad.1}
		\mathcal{M}_\tau(\vp) \approx \norm{L_{\tau,f}\Pi_{\tau,f}(\vp)}_{\ell^2(\R^d)}^2
		\quad \for \vp \in \atf. 
	\end{equation}
	
	Combining Lemma \ref{lem.L+M}, \eqref{quad.0}, and \eqref{quad.1}, we have
	\begin{equation}\label{quad.2}
		\norm{\vp}_{W(S,\tau)} \approx \norm{L_{\tau,f}\Pi_{\tau,f}(\vp)}_{\ell^2(\R^d)}^2 + \norm{L_w(\vp)}_{\ell^1(\R^d)}
		\quad \for \vp \in \vp_f + V_{\tau,f}.
	\end{equation}
	
	Setting $ \beta := L_w(\vp) $ and $ X := (L_{\tau,f}\Pi_{\tau,f})^\dagger(L_{\tau,f}\Pi_{\tau,f}) $, we see from \eqref{quad-0} and \eqref{quad.2} that computing the order of magnitude of $ \norm{f}_{\ct(S,\tau)} $ amounts to solving the following minimization problem:
	\begin{equation}\label{quad-std}
		\text{Minimize } \langle\beta,A\beta\rangle + \norm{\beta}_{\ell^1(\R^d)}
		\quad
		\text{subject to }
		L_w^{-1}\beta \in \vp_f + V_{\tau,f}.
	\end{equation}
	In particular, an optimal (feasible) solution to \eqref{quad-std} is a Whitney field $ \vp_* = (P_*^x)_{x \in S} \in W(S,\tau) $ with $ P_*^x(x) = f(x) $ for $ x \in S $ and $ \norm{\vp_*}_{W(S,\tau)} \approx \norm{f}_{\ct(S,\tau)} $.
	
	Finally, we note that \eqref{quad-std} is a convex quadratic programming problem with affine constraints. We can find the exact solution to \eqref{quad-std} by solving for its Karush-Kuhn-Tucker conditions, which consist of a bounded system of linear equations and inequalities\cite{BV-CO}. We can solve such a system, for instance via the simplex method or elimination, using at most $ C(n) $ operations, since the system size is controlled. We refer the readers to the appendix of \cite{JL20-Alg} for an elementary discussion, and \cite{BV-CO} for a detailed treatment of convex programming.
	
	\subsection{Homogeneous convex sets}

	For $ x \in \Rn $, $ S \subset E $, and $ k \geq 0 $, we define
	\begin{equation}\label{sigma-def}
		\begin{split}
			\sigma(x,S)& := \set{P \in \P: \begin{matrix*}[l]
					\text{ There exists $ \phi^S \in \ctrn $ with $ \norm{\phi^S}_\ctrn \leq 1 $,}\\
					\text{ such that  $ \phi^S = 0 $ on $ S $ and $ \jet_x \phi^S = P $.}
			\end{matrix*}}, \text{ and }\\
			\sk(x,k)&:= \bigcap_{S \subset E,\, \#S \leq k}\sigma(x,S).
		\end{split}
	\end{equation}

	\begin{theorem}[Finiteness Principle]\label{thm.fp-sigma}
		There exist controlled constants $k^\sharp_{n,\mathrm{old}}$ and $ C(n) $ for which the following holds. Let $ E \subset \Rn $ be a finite set. 
		\begin{enumerate}[(A)]
			\item Let $ f : E \to \R $. Suppose for every $S\subset E$ with $\#S \leq k^\sharp_{n,\mathrm{old}}$, there exists $F^S \in \ctrn$ with $\norm{F^S}_{\ctrn} \leq M$ and $F^S = f$ on $S$. Then there exists $F \in \ctrn$ with $\norm{F}_\ctrn \leq CM$ and $F = f$ on $E$.
			\item Let $ \sigma $ and $ \sk $ be as in \eqref{sigma-def}. Then for all $ k \geq k^\sharp_{n,\mathrm{old}} $,
			\begin{equation*}
				C^{-1}\sk(x,k) \subset \sigma(x,E)\subset C\cdot \sk(x,k)\forall x \in \Rn.
			\end{equation*}
		\end{enumerate}
	\end{theorem}
	
	For each $ n \in \mathbb{N} $, we fix a choice of $ k^\sharp_{n,\mathrm{old}} $. We further assume that $ k^\sharp_{n+1,\mathrm{old}} > k^\sharp_{n,\mathrm{old}} $ for all $ n $.
	
	See \cite{F05-Sh,BS01,FK09-Data-1,FK09-Data-2} for a proof of Theorem \ref{thm.fp-sigma}. For the special case $ n = 2 $, see also \cite{JL20}.

	\subsection{Main convex sets}

	For the rest of the section, we assume we are given a finite set $ E \subset \Rn $. 
	
	For $ x \in \Rn $, $ S \subset E $, $ f : E \to \itau $, and $ M \geq 0 $, we define
	\begin{equation}\label{G-def}
		\G(x,S,f,M) := \set{P \in \P: \begin{matrix*}[l]
				\text{ There exists $ F^S \in \ctrt $ with $ \norm{F^S}_\ctrn \leq M $,}\\
				\text{ such that  $ F^S = f $ on $ S $ and $ \jet_x F^S = P $.}
		\end{matrix*}}.
	\end{equation}
	
	For $ x \in \Rn $, $ k \in \mathbb{N}_0 $, $ f : E \to \itau $, and $ M \geq 0 $, we define 
	\begin{equation}\label{Gk-def}
		\Gk(x,k,f,M) := \bigcap_{S \subset E,\, \#S \leq k}\G(x,S,f,M).
	\end{equation}

	Note that $ \G $ and $ \Gk $ are (possibly empty) bounded convex subsets of $ \P $. With respect to set inclusion, $ \G $ is decreasing in $ S $ and increasing in $ M $; $ \Gk $ is decreasing in $ k $ and increasing in $ M $. 
	
	It follows from Lemma \ref{lem.Ktau} that
	\begin{equation}\label{Ktau-equiv}
		\K_\tau(x,C^{-1}M) \subset \G(x,\void,f,M) \subset \K_\tau(x,CM)
	\end{equation}
	for some controlled constant $ C(n) $.

	\begin{theorem}[Helly's Theorem]\label{thm.Helly}
		Let $ D \in \mathbb{N}_0 $. Let $ \mathcal{F} $ be a finite family of convex subsets of $ \R^{D} $. If every $ D+1 $ members of $ \mathcal{F} $ have nonempty intersection, then all members of $ \mathcal{F} $ have nonempty intersection.  
	\end{theorem}

	\begin{lemma}\label{lem.helly-1}
		Let $ x,x' \in \Rn $. Suppose $ k \geq (n+2)k' $. Then given $ P \in \Gk(x,k,f,M) $, there exists $ P' \in \Gk(x',k',f,M) $ satisfying
		\begin{equation}
			\abs{\da(P - P')(x)},\,\abs{\da(P - P')(x')} \leq C(n)\cdot M\abs{x-x'}^\tma.
			\label{3.4.0}
		\end{equation}
	\end{lemma}
	
	\begin{proof}
		For $ S \subset E $, we define
		\begin{equation*}
			K(S):= \set{\jet_{x'}F^S : \begin{matrix*}[l]
					F^S \in \ctrt \text{ with } \norm{F^S}_\ctrn \leq M ,\,
					F^S\big|_S = f,\text{ and } \jet_x F^S \equiv P
			\end{matrix*}}.
		\end{equation*}
		Note that $ K(S) $ is convex, and that $ K(S')\subset K(S) $ whenever $ S \subset S' $. It also follows from the definition of $ \Gk $ in \eqref{Gk-def} that if $ \#S \leq k $, then $ K(S) \neq \void $.
		
		Let $ S_1, \cdots, S_{n+2} \subset E $ be given with $ \#S_i \leq k' $ for each $ i $. Setting $ S := \bigcup_{i = 1}^{n+2}S_i $, we see that $ \#S \leq (n+2)k' \leq k $, so that $ K(S) \neq \void $. Therefore, 
		\begin{equation*}
			\bigcap_{i = 1}^{n+2} K(S_i) \supset K(S) \neq \void.
		\end{equation*}
		Since $ S_1, \cdots, S_{n+2} \subset E $ are arbitrary, Helly's Theorem \ref{thm.Helly} applied to the convex sets $ K(S_i) \subset \P $ (with $ \dim \P = n+1 $) yields
		\begin{equation*}
			\Gk(x',k',f,M) = \bigcap_{S' \subset E, \#S' \leq k'}K(S') \neq \void. 
		\end{equation*}
		Pick $ P' \in \bigcap_{S' \subset E, \#S' \leq k'}K(S') $. Estimate \eqref{3.4.0} then follows from Taylor's theorem. 
	\end{proof}

	\newcommand{\apo}{A_{\mathrm{polar}}}
	\newcommand{\acen}{A_{\mathrm{centric}}}

	\begin{lemma}\label{lem.big-small}
		Let $ Q \subset \Rn $ be a cube. Let $ k \geq 2 $. Let $ f : E \to \itau $. Suppose $ \Gk(x,k,f,M) \neq \void $ for each $ x \in E \cap 5Q $. Given any $ \acen > 0 $, there exists $ \apo = C(n)\cdot(\acen^{1/2}+1)^2 > 0 $ such that at least one of the following holds.
		\begin{enumerate}[(A)]
			\item Either $ \tau + f(x)\geq  \acen M \dq^2 $ for all $ x \in E \cap 5Q $, or $ \tau+f(x) \leq \apo M\dq^2 $ for all $ x \in E \cap 5Q $.
			
			\item Either $ \tau - f(x)\geq  \acen M \dq^2 $ for all $ x \in E \cap 5Q $, or $ \tau-f(x) \leq \apo M\dq^2 $ for all $ x \in E \cap 5Q $.
		\end{enumerate}
	\end{lemma}

	\begin{proof}
		We prove (A) here. The proof of (B) is similar.
		
		Fix $ \acen > 0 $. If $ \tau + f(x) \geq \acen M\dq^2 $ for all $ x \in E \cap 5Q $, then there is nothing to prove. 
		
		Suppose not, namely, there exists $ x_0 \in E \cap 5Q $ such that 
		\begin{equation}
			\tau+f(x_0)  < \acen M\dq^2 .
			\label{nearby-1}
		\end{equation}
		
		If $ \#(E\cap 5Q) = 1 $ then there is nothing to prove. We assume that $ \#(E\cap 5Q) \geq 2 $.

		Let $ P_0 \in \Gk(x_0,k,f,M) $. Let $ S \subset E \cap 5Q $ with $ \#(S) \leq 2 $ and $ x_0 \in S $. By \eqref{Gk-def}, there exists a function $ F^S \in \ctrt $ with $ \norm{F^S}_\ctrn \leq M $, $ F^S \big|_S = f $, and $ \jet_{x_0}F^S \equiv P_0 $. In particular, $ P_0 = f(x_0) $. By \eqref{nearby-1}, we have
		\begin{equation}
			\min\set{\sqrt{\tau-P(x_0)}, \sqrt{\tau+P(x_0)}}  < \acen^{1/2} M^{1/2}\dq. 
			\label{nearby-2}
		\end{equation}
		
		Recall $ \K_\tau $ in Definition \ref{def.Ktau}. Since $ P_0 \in \Gk(x_0,k,f,M) $, we also have $ P_0 \in \K_\tau(x_0,CM) $ by Lemma \ref{lem.Ktau}. Therefore, by \eqref{Ktau-def-2} and \eqref{nearby-2}, 
		\begin{equation*}
			\abs{\grad F^S(x_0)} = \abs{\grad P_0} \leq C\acen^{1/2}M\dq.
		\end{equation*}
		Since $ \norm{F^S}_\ctrn \leq M $, Taylor's theorem implies
		\begin{equation*}
			\abs{\grad F^S(x)} \leq C(1+\acen^{1/2})M\dq\quad \for x \in 5Q.
		\end{equation*}

		From \eqref{nearby-1}, we see that $ \tau + {F^S(x_0)} <\acen M\dq^2 $. Writing $ F^S(x) - F^S(x_0) =  \int_{\mathrm{seg}(x_0 \to x)}\grad F^S $, we see that
		\begin{equation*}
			\begin{split}
				\tau + {F^S(x)} &\leq \abs{\tau + F^S(x)}
				\leq\tau + {F^S(x_0)}  + \int_{\mathrm{seg}(x_0 \to x)}\abs{\grad F^S}  \\
				&\leq  C(\acen + \acen^{1/2}+1)M\dq^2
				\leq  C'(\acen^{1/2}+1)^2M\dq^2.
			\end{split}
		\end{equation*}
		In particular, for each $ x \in S $, 
		\begin{equation*}
			\tau + f(x) = \tau + F^S(x) \leq C(\acen^{1/2}+1)^2M\dq^2.
		\end{equation*}
		
		Since $ S $ is any arbitrary subset of $ E \cap 5Q $ containing two points, conclusion (A) follows.

	\end{proof}

	\newcommand{\B}{\mathcal{B}}
	For $ x \in \Rn $ and $ \delta > 0 $, we define
	\begin{equation}\label{B-def}
		\B(x,\delta) := \set{P \in \P : \abs{\da P(x)} \leq \delta^{\tma}\for \abs{\alpha} \leq 2}. 
	\end{equation}
	
	The significance of $ \B $ is that given $ F \in \ctrn $ and $ x,y \in \Rn $, Taylor's theorem implies 
	\begin{equation*}
		\jet_xF - \jet_yF \in C(n)\norm{F}_\ctrn\cdot \B(x,\abs{x-y}).
	\end{equation*}

	\begin{lemma}\label{lem.perturb}
		Let $ k \geq 2 $. Let $ Q \subset \Rn $ be a cube with $ E \cap 5Q \neq \void $. Suppose $ x_Q \in Q $ satisfy $ \dist{x_Q}{E} \geq c_0\dq $ for some controlled constant $ c_0(n) $. Let $ f : E \to \itau $ be given. Suppose $ \Gk(x_Q,k,f,M) \neq \void $. The following are true. 
		\begin{enumerate}[(A)]
			\item There exists a number $ A_{\mathrm{perturb}} $ exceeding a large controlled constant such that the following holds. Suppose $ \min\set{\tau - f(x), \tau + f(x)} \geq A_{\mathrm{perturb}}M\dq^2 $ for each $ x \in E \cap 5Q $. Then
			\begin{equation*}
				\Gk(x_Q,k,f,M) + M\cdot\B(x_Q,\dq) \subset \Gk(x_Q,k,f,AM),\quad A = A(n,A_{\mathrm{perturb}}).
			\end{equation*}

			\item Suppose $ \tau + f(x) \leq A_{\mathrm{flat}}M\dq^2 $ for some $ x \in E \cap 5Q $ and $ A_{\mathrm{flat}} \geq 0 $. Then
			\begin{equation*}
				-\tau \in \Gk(x_Q,k,f,AM), \quad A = A(n,A_{\mathrm{flat}}).
			\end{equation*}
			Here, $ -\tau $ is the constant polynomial.
			
			Similarly, suppose $ \tau - f(x) \leq A_{\mathrm{flat}}M\dq^2 $ for some $ x \in E \cap 5Q $ and $ A_{\mathrm{flat}} \geq 0 $. Then
			\begin{equation*}
				\tau \in \Gk(x_Q,k,f,AM), \quad A = A(n,A_{\mathrm{flat}}).
			\end{equation*}
			Here, $ \tau $ is the constant polynomial.
		\end{enumerate}
	\end{lemma}

	\begin{proof}
		
		We write $ C, C' $, etc., to denote controlled constants depending only on $ n $. Recall $ \K_\tau $ in Definition \ref{def.Ktau}. 
		
		\paragraph{Proof of (A)}

		\newcommand{\apert}{A_{\mathrm{perturb}}}
		\newcommand{\aflat}{A_{\mathrm{flat}}}
		
		We claim that under the hypothesis of (A), given any $ P \in \Gk(x_Q,k,f,M) $, we have
		\begin{equation}\label{perturb-1}
			\min\set{\tau - P(x_Q), \tau + P(x_Q)} \geq c(n)\cdot({\apert}^{1/2}-1)^2\cdot M\dq^2.
		\end{equation}
		
		To see this, we fix $ P \in \Gk(x_Q,k,f,M) $ and $ x \in E \cap 5Q $. By the definition of $ \Gk $, there exists $ F \in \ctrt $ with $ \norm{F}_\ctrn \leq M $, $ \jet_{x_Q}F \equiv P $, and $ F(x) = f(x) $. In particular, 
		\begin{equation}\label{perturb-2}
			\min\set{\tau - F(x), \tau + F(x)} \geq \apert M\dq^2.
		\end{equation}
		Suppose toward a contradiction, that 
		\begin{equation}\label{perturb-3}
			\tau - \abs{F(x_Q)} \leq A_0M \dq^2
		\end{equation}
		for some to-be-determined $ A_0 $ depending only on $ n $ and $ \apert $. Since $ \jet_{x_Q}F \in \K_\tau(x_Q,CM) $, \eqref{Ktau-def-2} and \eqref{perturb-3} would imply
		\begin{equation}\label{perturb-4}
			\abs{\grad F(x_Q)} \leq CA_0^{1/2}M\dq.
		\end{equation}
		Applying Taylor's theorem to \eqref{perturb-3} and \eqref{perturb-4}, we see that
		\begin{equation}\label{perturb-5}
			\tau - \abs{F(x)} \leq C(A_0 + A_0^{1/2} + 1)M\dq^2 \leq C'({A_0^{1/2}}+1)^2M\dq^2
			\quad\for x \in 5Q.
		\end{equation}
		If we pick $ A_0 $ to be so small that $ A_0^{1/2} < \frac{\apert^{1/2}}{C'}-1 $ with $ C' $ as in \eqref{perturb-5}, we see that \eqref{perturb-5} will contradict \eqref{perturb-2}. Hence, \eqref{perturb-1} holds. 
		
		Now we fix a jet $ P \in \Gk(x_Q,k,f,M) $. We know that $ P $ satisfies \eqref{perturb-1}. Let $ \tilde{P} \in M\cdot \B(x_Q,\dq) $. By definition \eqref{B-def}, we have $ \abs{\da \tilde{P}(x_Q)} \leq M\dq^2 $ for $ \abs{\alpha} \leq 2 $. We want to show that $ P + \tilde{P} \in \Gk(x_Q,k,f,CM) $. 
		
		By the definition of $ \Gk $, we want to show that given $ S \subset E $ with $ \#S \leq k $, there exists $ F^S \in \ctrt $ with $ \norm{F^S}_\ctrn \leq CM $, $ F^S(x) =f(x) $ for $ x \in S $, and $ \jet_{x_Q}F^S \equiv P + \tilde{P} $. 
		
		Fix $ S \subset E $ with $ \#S \leq k $. We define $ S^+ := S \cup \set{x_Q} $. Since $ P \in \Gk(x_Q,k,f,M) $, there exists $ F^S \in \ctrt $ with 
		\begin{equation}\label{perturb-6}
			\norm{F^S}_\ctrn \leq M,\,
			F^S(x) = f(x)\for x \in S,\text{ and }
			\jet_{x_Q}F^S \equiv P.
		\end{equation}
		Consider the Whitney field $ \vec{P} \in W(S^+) $ defined by
		\begin{itemize}
			\item $ P^x \equiv \jet_x F^S $ for $ x \in S $, and
			\item $ P^{x_Q} \equiv \jet_{x_Q}F^S + \tilde{P} \equiv P + \tilde{P} $.
		\end{itemize}
		Thanks to Whitney's Extension Theorem \ref{thm.WT-tau}(B), it suffices to show that $ \vec{P}\in W(S^+,\tau) $ and $ \norm{\vec{P}}_{W(S^+,\tau)} \leq CM $.
		
		Thanks to \eqref{perturb-6}, we have
		\begin{align}
			&P^x \in \K_\tau(x,CM) \for x \in S,\quad \text{and}\label{perturb-7}\\
			&\abs{\da(P^x - P^y)(x)} \leq CM\abs{x-y}^\tma \for x,y \in S,\, x \neq y,\, \abs{\alpha} \leq 1.\label{perturb-8}
		\end{align}
		
		On the other hand, using \eqref{perturb-1}, we have
		\begin{equation}
			\tau - \abs{P^{x_Q}(x_Q)} = \tau - \abs{P(x_Q) + \tilde{P}(x_Q) } \geq \brac{C(\apert^{1/2}-1)^2 - 1}M\dq^2.
			\label{perturb-9}
		\end{equation}
		Since $ P \equiv \jet_{x_Q}F^S $, \eqref{perturb-6} implies that $ P \in \K_\tau(x_Q,CM) $. Using property \eqref{Ktau-def-2} of $\K_\tau$ and \eqref{perturb-1}, we have
		\begin{equation}
			\abs{\grad P^{x_Q}} \leq \abs{\grad P} + \abs{\grad \tilde{P}} \leq \brac{C(\apert^{1/2}-1) - 1}
			M\dq. 
			\label{perturb-10}
		\end{equation}
		Combining \eqref{perturb-9} and \eqref{perturb-10}, we see that
		\begin{equation}\label{perturb-11}
			P^{x_Q} \in \K_\tau(x_Q,AM)
			\text{ with } A = A(n,A_{\mathrm{perturb}}). 
		\end{equation}
		
		Combining \eqref{perturb-7} and \eqref{perturb-11}, we see that $ \vec{P} \in W(S^+,\tau) $.
		
		Now we estimate $ \norm{\vec{P}}_{W(S^+,\tau)} $.

		By assumption, $ x_Q \in Q $ satisfies 
		\begin{equation*}
			\abs{x_Q - x} \geq c_0\dq \for x \in E.
		\end{equation*}
		As a consequence,
		\begin{equation}\label{perturb-12}
			\abs{\da\tilde{P}(x_Q)} \leq CM\dq^\tma \leq M\abs{x - x_Q}^\tma \for x \in E, \abs{\alpha} \leq 1. 
		\end{equation}
		By Taylor's theorem and \eqref{perturb-6}, we have
		\begin{equation}
			\abs{\da (P^x - P)(x)}, \abs{\da (P^x - P)(x_Q)} \leq CM\abs{x - x_Q}^\tma\for x \in S, \abs{\alpha} \leq 1.
			\label{perturb-13}
		\end{equation}
		Combining \eqref{perturb-12} and \eqref{perturb-13}, we see that
		\begin{equation}
			\abs{\da(P^x - P^{x_Q})(x_Q)} \leq \abs{\da (P^x - P)(x_Q)} + \abs{\da\tilde{P}(x_Q)} \leq CM\abs{x-x_Q}^\tma \for x \in S, \abs{\alpha} \leq 1.
			\label{perturb-14}
		\end{equation}
		By Taylor's theorem and \eqref{perturb-14}, we also have
		\begin{equation}
			\abs{\da(P^x - P^{x_Q})(x)}  \leq CM\abs{x-x_Q}^\tma \for x \in S, \abs{\alpha} \leq 1.
			\label{perturb-15}
		\end{equation}		
		Finally, we see from \eqref{perturb-8}, \eqref{perturb-14}, and \eqref{perturb-15} that $ \norm{\vec{P}}_{W(S^+,\tau)} \leq CM $. 
		
		This proves Lemma \ref{lem.big-small} (A).

		\paragraph{Proof of (B)} We prove the case when $ \tau+f(x) \leq \aflat M\dq^2 $. The case $ \tau - f(x) $ is similar.
		
		We claim that under the assumption of (B), given any $ P \in \Gk(x_Q,k,f,M) $, we have
		\begin{equation}\label{flat-1}
			\tau + P(x_Q) \leq C(\aflat^{1/2}+1)^2M\dq^2.
		\end{equation}
		The proof is similar to that of Lemma \ref{lem.big-small}(B). We provide the proof here for completeness.
		
		Fix $ x_0 \in E \cap 5Q $ such that $ \tau + f(x_0) \leq \aflat M\dq^2 $. Let $ P \in \Gk(x_Q,k,f,M) $. By the definition of $ \Gk $, there exists a function $ F \in \ctrt $ with $ \norm{F}_\ctrn \leq M $, $ \jet_{x_Q}F \equiv P $, and $ F(x_0) = f(x_0) $. Since $ \jet_{x_0} F \in \K_\tau(x_0,CM) $, property \eqref{Ktau-def-2} of $\K_\tau$ implies $ \abs{\grad F(x_0)} \leq C\aflat^{1/2}M\dq $. \eqref{flat-1} then follows from Taylor's theorem and the estimates immediately above. 
		
		Now we need to show that the constant polynomial $ -\tau \in \Gk(x_Q,k,f,AM) $ for some $ A $ depending only on $ n $ and $ \aflat $. We write $ A, A' $, etc., to denote quantities that depend only on $ n $ and $ \aflat $.
		
		Let $ P \in \Gk(x_Q,k,f,M) $. Then $ P \in \K_\tau(x_Q,AM) $ and satisfies \eqref{flat-1}. By \eqref{Ktau-def-2} and \eqref{flat-1}, 
		\begin{equation}
			\label{flat-2}
			\abs{\da (\tau + P)(x_Q)} \leq AM\dq^{\tma}\for \abs{\alpha}\leq 2.
		\end{equation}
		
		Fix $ S \subset E $ with $ \#S \leq k $. We want to show that $ -\tau \in \G(x_Q,S,f,AM) $.
		
		By the definition of $ \Gk $, there exists $ F^S \in \ctrt $ with
		\begin{equation}\label{flat-3}
			\norm{F^S}_\ctrn \leq M,\,
			F^S(x) = f(x)\for x \in S,\text{ and }
			\jet_{x_Q}F^S \equiv P.
		\end{equation}
		
		From Taylor's theorem, together with \eqref{flat-2} and \eqref{flat-3}, we see that
		\begin{equation}
			\abs{\da (\tau + F^S)(x)} \leq AM\dq^\tma \for x \in B(x_Q,c_0\dq), \abs{\alpha} \leq 2.
			\label{flat-4}
		\end{equation}
		Here, $ c_0 $ is the constant in the hypothesis of Lemma \ref{lem.big-small}.
		
		Let $ \chi \in \ctrn $ be a cutoff function such that
		\begin{enumerate}[($ \chi $1)]
			\item $ 0 \leq \chi \leq 1 $ on $ \Rn $;
			\item $ \chi \equiv 1 $ near $ x_Q $ and $ \supp{\chi} \subset B(x_Q,c_0\dq) $;
			\item $ \abs{\da \chi}\leq C\dq^\tma $ for $ \abs{\alpha} \leq 2 $. 
		\end{enumerate}
		We define
		\begin{equation*}
			\tilde{F}^S := \chi \cdot (-\tau) + (1-\chi)\cdot F^S.
		\end{equation*}

		It is clear that $ \tilde{F}^S \in \ctrn $. 
		
		Thanks to ($ \chi 1 $), $ \tilde{F}^S $ is defined as the convex combination of two functions with range $ \itau $. Therefore, $ \tilde{F}^S \in \ctrt $. Thanks to ($ \chi  2$) and the fact that $ \dist{x_Q}{E}\geq c_0\dq $, we have $ \tilde{F}^S(x) = f(x) $ for $ x \in S $.
		Thanks to ($ \chi 2$) again, $ \jet_{x_Q}\tilde{F}^S \equiv -\tau $. 
		Finally, thanks to \eqref{flat-4} and ($ \chi 3 $), we have $ \norm{\tilde{F}^S}_\ctrn \leq AM $. 
		
		Hence, $ -\tau \in \G(x_Q,S,f,AM) $. Since $ S $ is chosen arbitrarily, we can conclude that $ -\tau \in \G(x_Q,k,f,AM) $. This proves Lemma \ref{lem.big-small} (B).

	\end{proof}

	\section{Base case of the induction}
	\label{sect:1d}
	
	\renewcommand{\ctrn}{{C^2(\R)}}
	\renewcommand{\ctrt}{{C^2(\R,\tau)}}
	\renewcommand{\Rn}{\R}
	\renewcommand{\ctet}{{C^2(E_*,\tau)}}
	
	In this section, we prove a stronger version of Theorem \ref{thm.bd-alg-1d}. We use $ \P $ and $ \P^+ $, respectively, to denote the vector space of single-variable polynomials with degree no greater than one and two. We use $ \jet_x $ and $ \jet_x^+ $, respectively to denote the one-jet and two-jet of a single-variable function twice continuously differentiable near $ x \in \R $. 
	
	\newcommand{\ddxm}{{\frac{d^m}{dx^m}}}

	\begin{lemma}\label{lem.alg-1d-3}
		Suppose we are given a finite set $ E_* \subset \Rn $ with $ \#E_* \leq 3 $. Then there exists a collection of maps $ \set{\Xi_{\tau,x} : \tau \in \pos,\,x \in \Rn} $, where 
		\begin{equation*}
			\Xi_{\tau,x} : \ctet \to \P^+
		\end{equation*}
		for each $ x \in \Rn $, such that the following hold.
		\begin{enumerate}[(A)]

			\item Let $ f\in \ctet $ be given. Then there exists a function $ F \in \ctrt $ such that
			\begin{equation*}
				\begin{split}
					\jet^+_x F = \Xi_{\tau,x}(f) \text{ for all } x \in \Rn,\,
					\norm{F}_{\ctrn} \leq C\norm{f}_{\ct(E_*,\tau)},
					\text{ and }
					F(x) = f(x)
					\for
					x \in E_*.
				\end{split}
			\end{equation*}
			Here, $ C $ depends only on $ n $.

			\item There is an algorithm that takes the given data set $E_*$, performs one-time work, and then responds to queries.
			
			A query consists of a pair $ (\tau,x) \in\pos\times \Rn $, and the response to the query is the map $ \Xi_{\tau,x} $, given in its efficient representation.
			
			The one-time work takes $ C_1 $ operations and $ C_2 $ storage. The work to answer a query is $ C_3 $. Here, $ C_1, C_2, C_3 $ are universal constants. 
		\end{enumerate}
	\end{lemma}

	\begin{proof}
		Let $ \tau > 0 $ and $ f \in \ctet $ be given. 
		Let $ \mathcal{L} $ and $ \mathcal{M} $, respectively, be as in \eqref{L-def} and \eqref{M-def}, with $ S = E_* $. Consider the affine subspace $ \atf \subset W(E) $ given by
		\begin{equation}\label{1d-3-1}
			\atf:= \set{\vec{P} = (P^x)_{x \in E} \in W(E_*) :\, \begin{matrix*}[l]
					P^x(x) = f(x)\for x \in E_*.\\
					\text{If }f(x) = \tau \text{ then } P^x \equiv \tau.\\
					\text{If }f(x) = -\tau \text{ then } P^x \equiv -\tau.
			\end{matrix*}}.
		\end{equation}
		
		Consider the optimization problem
		\begin{equation}\label{1d-3-2}
			\text{Minimize $ \mathcal{L}+\mathcal{M} $ over $ \atf $}.
		\end{equation}
		By Section \ref{sect:quad}, we can find an approximate minimizer $ \vec{P}_{0} \in W(E_*,\tau) $
		of \eqref{1d-3-2}\footnote{
			Namely, $ (\mathcal{L}+\mathcal{M})
			(\vec{P}_{0})
			\leq C\cdot \inf
			\set{ 
				(\mathcal{L}+\mathcal{M})(\vec{P}) : \vec{P} \in \atf 
			} $ for some universal constant $ C $.
		}
		using $ C $ operations. We denote the solution by 
		\begin{equation*}
			\widetilde{Min}_\tau : \ctet \to W(E,\tau).
		\end{equation*}
		It follows from Lemma \ref{lem.L+M} that 
		\begin{equation}\label{1d-3-4}
			\norm{\widetilde{Min}_\tau(f)}_{W(E_*,\tau)} \leq C\norm{f}_{\ct(E_*,\tau)}.
		\end{equation}
		
		Let $ \T_w $ be as in Whitney's Extension Theorem \ref{thm.WT-tau}(B) with $ S = E_* $. We define an extension operator
		\begin{equation}\label{1d-3-3}
			\E_{\tau,*} := \T_w \circ \widetilde{Min}_\tau : \ctet \to \ctrt.
		\end{equation}
		We then define
		\begin{equation*}
			\Xi_{\tau,x} := \jet_x^+ \circ \E_{\tau,*}
		\end{equation*}
		
		Lemma \ref{lem.alg-1d-3}(A) follows from the conclusion of Whitney's Extension Theorem \ref{thm.WT-tau}(B) and \eqref{1d-3-4}. Lemma \ref{lem.alg-1d-3}(B) follows from the discussion in Section \ref{sect:quad}, and the fact that the Whitney extension operator can be constructed using $ C $ operations, since $ \#E  $ is bounded. 
	\end{proof}

	First sorting the set $ E \subset \R $, and then patching together adjacent maps, we arrive at the following Theorem \ref{thm.bd-alg-1d+}. Note that Theorem \ref{thm.bd-alg-1d+} is in fact stronger than Theorem \ref{thm.bd-alg-1d}, since the extension maps do not depend on the parameter $ M $.

	\renewcommand{\ctet}{{C^2(E,\tau)}}
	\begin{theorem}\label{thm.bd-alg-1d+}
		Suppose we are given a finite set $ E \subset \Rn $ with $ \#(E) = N $. Then there exists a collection of maps $ \set{\Xi_{\tau,x} : \tau \in \pos,\,x \in \Rn} $, where 
		\begin{equation*}
			\Xi_{\tau,x} : \ctet \to \P^+
		\end{equation*}
		for each $ x \in \Rn $, such that the following hold.
		\begin{enumerate}[(A)]
			\item There exists a universal constant $ D $ such that for each $ x \in \Rn $, the map $ \Xi_{\tau,x}(\cdot) : \ctet\to \P^+ $ is of depth $ D $. Moreover, the source of $ \Xi_{\tau,x} $ is independent of $ \tau $.

			\item Let $ f\in \ctet $ be given. Then there exists a function $ F \in \ctrt $ such that
			\begin{equation*}
				\begin{split}
					\jet^+_x F = \Xi_{\tau,x}(f) \text{ for all } x \in \Rn,\,
					\norm{F}_{\ctrn} \leq C\norm{f}_\ctet,
					\text{ and }
					F(x) = f(x)
					\for
					x \in E.
				\end{split}
			\end{equation*}
			Here, $ C $ depends only on $ n $.

			\item There is an algorithm that takes the given data set $E$, performs one-time work, and then responds to queries.
			
			A query consists of a pair $ (\tau,x) \in\pos\times \Rn $, and the response to the query is the depth-$ D $ map $ \Xi_{\tau,x} $, given in its efficient representation.
			
			The one-time work takes $ C_1N\log N $ operations and $ C_2N $ storage. The work to answer a query is $ C_3\log N $. Here, $ C_1, C_2, C_3 $ are universal constants. 
		\end{enumerate}
	\end{theorem}

	\begin{proof}
		
		If $ \#E \leq 3 $, then Theorem \ref{thm.bd-alg-1d+} reduces to Lemma \ref{lem.alg-1d-3}. Suppose $ \#E \geq 4 $. We sort $ E $ into an order list $ E = \set{x_1 < \cdots < x_N} $. For convenience, we set $ x_0 := -\infty $ and $ x_{N+1} := +\infty $. We set
		\begin{itemize}
			\item $ E_\nu := \set{x_{\nu-1}, x_{\nu}, x_{\nu+1}} $ for $ \nu = 2, \cdots, N-1 $, $ E_1 := E_2 $, and $ E_N := E_{N-1}$.
			\item $ J_\nu := (x_{\nu-1}, x_{\nu+1}) $ for $ \nu = 1,\cdots, N $. 
		\end{itemize}
		
		Let $ \set{\theta_\nu : \nu = 1, \cdots, N} $ be a nonnegative $ C^2 $-partition of unity subordinate to $ \set{J_\nu: \nu = 1, \cdots, N} $ that satisfies the following.
		\begin{itemize}
			\item[\LA{theta1}] For each $ \nu = 1, \cdots, N $, $ \theta_\nu \equiv 1 $ near $ x_\nu $, and $ \supp{\theta_\nu} \subset J_\nu $. Note that each $ x \in \R $ is supported by at most two $ \theta_\nu $'s. 
			\item[\LA{theta2}] For each $ \nu = 1, \cdots, N $, $ \abs{\ddxm \theta_\nu(x)} \leq \begin{cases}
				C\abs{x_\nu - x_{\nu-1}}^{-m}\for x \in (x_{\nu-1},x_\nu)\\
				C\abs{x_{\nu+1}-x_\nu}^{-m}\for x \in (x_{\nu},x_{\nu+1})
			\end{cases} $, $m = 0,1,2$, for some universal constant $ C $. For $\nu = 1, N$, we use the convention $\infty^{0} = 1$ and $\infty^{-m} = 0$ for $m \geq 1$. 
		\end{itemize}
		
		We can construct each $ \theta_\nu $ using the standard trick with summing and dividing cutoff functions.

		For each $ \nu = 2, \cdots, N-1 $, we define $ \E_{\tau,\nu} $ as in \eqref{1d-3-3} of Lemma \ref{lem.alg-1d-3} with $ * = \nu $. We define an extension operator $ \E_\tau : \ctet \to \ctrt $ by
		\begin{equation*}
			\E_\tau(f)(x) := \sum_{\nu = 1}^N \theta_\nu(x) \cdot \E_{\tau,\nu}(f|_{E_\nu})
			\quad \for f \in \ctet.
		\end{equation*}
		It is clear that $ \E_\tau(f) \in \ctrn $ since each each $ \E_{\tau,\nu}(f|_{E_\nu}) \in \ctrn $. Moreover, since the range of each $ \E_{\tau,\nu}(f|_{E_{\nu}}) $ is $ \itau $ and $ \E_\tau $ is a convex combination of $ \E_{\tau,\nu} $, we have $ \E_\tau(f) \in \ctrt $. 		
		
		Now we show that $ \E_\tau $ is bounded. 
		
		Let $ M := \max_{\nu = 1, \cdots N}\norm{f}_{\ct(E_\nu,\tau)} $. By definition of trace norm, we have $ \norm{f}_{\ctet} \geq M $. 
		
		Let $ F_\nu := \E_{\tau,\nu}(f|_{E_\nu}) $ for each $ \nu $. By Lemma \ref{lem.alg-1d-3}(A), we have
		\begin{equation}
			\norm{F_\nu}_{\ctrn} \leq CM
			\quad \for \nu = 1, \cdots N.
			\label{alg-1d-1}
		\end{equation}
		
		Since $ \set{\theta_\nu:\nu = 1, \cdots, N} $ is a partition of unity, we see from \eqref{alg-1d-1} that
		\begin{equation}
			\abs{\E_\tau(f)(x)} \leq CM\quad \for x \in \R.
			\label{alg-1d-11}
		\end{equation}
		
		On $ J_1 $ or $ J_N $, we see from and the definitions of $ E_1 $, $ E_N $ and the support condition \eqref{theta1} that $ \E_\tau(f)\equiv \E_{\tau,1}(f) $ or $ \E_{\tau}(f) \equiv \E_{\tau,N}(f) $. Therefore,
		\begin{equation}\label{alg-1d-2}
			\abs{\ddxm \E_\tau (f)(x)} \leq CM
			\quad \for x \in J_1 \cup J_N.
		\end{equation}
		
		Suppose $ x \in [x_2, x_N] $. Let $ \nu(x) $ be the least integer such that $ x \in J_{\nu(x)} $. Then the only partition functions possibly nonzero at $ x $ is $ \theta_{\nu(x)} $ and $ \theta_{\nu(x)+1} $. Since $ \ddxm \theta_{\nu(x)} = -\ddxm\theta_{\nu(x)+1}(x) $ for $ m = 1 $ and $ 2 $, we have
		\begin{equation}\label{alg-1d-3}
			\ddxm \E_\tau(f)(x) =  \sum_{0 \leq m' \leq m}C_{m,m'}\frac{d^{m'}}{dx^{m'}}
			\brac{F_{\nu(x)}-F_{\nu(x)+1}}
			(x)\frac{d^{m-m'}}{dx^{m-m'}}\theta_{\nu(x)}(x).
		\end{equation}
		
		\renewcommand{\ddxm}{{\frac{d^{m'}}{dx^{m'}}}}
		We claim that
		\begin{equation}\label{alg-1d-4}
			\abs{
				\ddxm
				\brac{
					F_{\nu(x)} - F_{\nu(x)+1}
				}(x)
			}
			\leq CM\abs{x_{\nu(x)+1}-x_{\nu(x)}}^{2-m'}
			\quad \for 0\leq m' \leq 2.
		\end{equation}
		To see this, observe that $ F_{\nu(x)}= F_{\nu(x)+1} $ at $ x_{\nu(x)} $ and $ x_{\nu(x)+1} $. Therefore, by Rolle's theorem, there exists $ \hat{x}_{\nu(x)} \in (x_{\nu(x)}, x_{\nu(x)+1}) $ such that 
		\begin{equation}\label{alg-1d-40}
			\frac{d}{dx}\brac{F_{\nu(x)} - F_{\nu(x)+1}}(\hat{x}_{\nu(x)})= 0
		\end{equation}
		Now, \eqref{alg-1d-4} follows from \eqref{alg-1d-40} and Taylor's theorem. 
		
		Using \eqref{theta2} and \eqref{alg-1d-4} to estimate \eqref{alg-1d-3}, we can conclude that
		\begin{equation}\label{alg-1d-5}
			\abs{\frac{d^m}{dx^m}\E_\tau(f)(x)} \leq CM
			\quad 
			\for x \in [x_2,x_N]. 
		\end{equation}
		
		From \eqref{alg-1d-11}, \eqref{alg-1d-2}, and \eqref{alg-1d-5}, we see that
		\begin{equation*}
			\norm{\E_{\tau}(f)}_\ctrn \leq CM \leq C\norm{f}_\ctet.
		\end{equation*} 
		Namely, $ \E_\tau : \ctet \to \ctrt $ is bounded. 
		
		Finally, we set
		\begin{equation*}
			\Xi_{\tau,x}:= \jet_x^+\circ \E_\tau.
		\end{equation*}
		
		Theorem \ref{thm.bd-alg-1d+}(A) then follows from the boundedness of $ \E_\tau $. 
		
		The one-time work consists of sorting the set $ E $, computing $ E_\nu $, $ J_\nu $, and $ \theta_\nu $ for $ \nu = 1, \cdots, N $. This requires at most $ CN\log N $ and $ CN $ operations. 
		
		Now we discuss query work. Let $ (\tau,x) \in \pos \times \R $ and $ f \in \ctet $ be given. It requires at most $ C\log N $ operations to locate $ \nu(x) $, where $ \nu(x) $ is the least integer that $ x \in J_{\nu(x)} $. Note that $ \Xi_{\tau,x}(f) $ is a linear combination of $ \jet_x^+\E_{\tau,\nu(x)}(f|_{E_{\nu(x)}}) $, $ \jet_x^+\E_{\tau,\nu(x)+1}(f|_{E_{\nu(x)+1}}) $, $ \jet_x^+\theta_{\nu(x)}$, and $ \jet_x^+\theta_{\nu(x)+1} $. By Lemma \ref{lem.alg-1d-3}, we can compute $ \jet_x^+\E_{\tau,\nu(x)}(f|_{E_{\nu(x)}}) $ and $ \jet_x^+\E_{\tau,\nu(x)+1}(f|_{E_{\nu(x)+1}}) $from $ (E_{\nu(x)},f) $ and $ (E_{\nu(x)},f) $, respectively, using at most $ C $ operations. On the other hand, we can compute $ \jet_x^+\theta_{\nu(x)}$ and $ \jet_x^+\theta_{\nu(x)+1} $ from $ J_{\nu(x)} $ and $ J_{\nu(x)+1} $ using at most $ C $ operations. Therefore, the work to answer a query is $ C\log N $.

	\end{proof}

	\renewcommand{\ctrn}{{C^2(\R^n)}}
	\renewcommand{\ctrt}{{C^2(\R^n,\tau)}}
	\renewcommand{\Rn}{\R^n}

	\section{Set up for the induction}
	\label{sect:induction}

	\renewcommand{\bar}[1]{\overline{#1}}
	We write $ \overline{x}\in\R^{n-1} $ to denote points in $ \R^{n-1} $. We write $ \overline{\P} $ and $ \overline{\P}^+ $ to denote the vector spaces of polynomials on $ \R^{n-1} $ with degree no greater than one and two, respectively. We write $ \bar{\jet}_{\bar{x}} $ and $ \bar{\jet}_{\bar{x}}^+ $ to denote the one-jet and two-jet of a function (twice) differentiable near $ \bar{x}\in\R^{n-1} $.

	Given any finite set $ \overline{E} \subset \R^{n-1} $ with $ \#\overline{E} = \overline{N} $, we assume the following.
	\begin{itemize}

		\item[\LA{induction:IH}](\textbf{IH}) There exists a collection of maps $ \set{\overline{\Xi}_{\tau,\overline{x}} : \tau \in \pos,\,\overline{x} \in \R^{n-1}} $, where
		\begin{equation*}
			\overline{\Xi}_{\tau,\overline{x}} : \ct(\bar{E},\tau) \times\pos \to \overline{\P}^+
		\end{equation*}
		for each $ x \in \Rn $, such that the following hold.
		\begin{enumerate}[(A)]
			\item There exists a controlled constant $ \overline{D} $, depending only on $ n $, such that for each $ \overline{x} \in \R^{n-1} $, the map $ \overline{\Xi}_{\tau,\overline{x}} $ is of depth $ \overline{D} $, and the source of $ \overline{\Xi}_{\tau,\overline{x}} $ is independent of $ \tau $.
			\item Suppose we are given $ (\overline{f},M) \in C^2(\overline{E},\tau) \times\pos $ with $ \norm{\overline{f}}_{\ct(\overline{E},\tau)} \leq M $. Then there exists a function $ \overline{F} \in \ct(\R^{n-1},\tau) $ such that
			\begin{equation*}
				\overline{\jet}^+_{\overline{x}}\overline{F} = \bar{\Xi}_{\tau,x}(\bar{f},M)\forall \bar{x}\in\R^{n-1},\,
				\norm{\bar{F}}_{\ct(\R^{n-1})} \leq C(n)\cdot M, \text{ and }\bar{F}(\bar{x}) = \bar{f}(\bar{x}) \for \bar{x}\in\bar{E}.
			\end{equation*}
			\item There exists an algorithm that takes the given data set $ \bar{E} $, and then responds to queries.
			
			A query consists of a pair $ (\tau,\bar{x})\in\pos\times\R^{n-1} $, and the response to the query is the depth-$ \bar{D} $ map $ \bar{\Xi}_{\tau,x} $, given in its efficient representation.
			
			The one-time work takes $ \bar{C}_1\bar{N}\log\bar{N} $ operations and $ \bar{C}_2\bar{N} $ storage. The work to answer a query is $ \bar{C}_3\log\bar{N} $. Here, $ \bar{C}_1, \bar{C}_2, \bar{C_3} $ depend only on $ n $. 
			
			\item For each $ \tau > 0 $, we use $ \bar{\E}_{\tau}: \ct(\bar{E},\tau)\times\pos \to \ctrn $ to denote the operator associated with $ \set{\Xi_{\tau,\bar{x}}:\bar{x}\in\R^{n-1}} $ determined by the relation
			\begin{equation*}
				\bar{\jet}^+_{\bar{x}}\circ \bar{\E}_\tau (\bar{f},M) \equiv \bar{\Xi}_{\tau,\bar{x}}(\bar{f},M)
				\quad \text{ for all } (\bar{f},M) \in \ctet\times\pos.
			\end{equation*}
			Thanks to (A), for each $ \bar{x} \in \Rn $, there exists $ \bar{S}(\bar{x}) \subset \bar{E} $ with $ \#\bar{S}(\bar{x}) \leq \bar{D}(n) $, independent of $ \tau $, such that for all $ \bar{f},\bar{g} \in \ctet $ with $ \bar{f} = \bar{g} $ on $ \bar{S}(\bar{x}) $, we have
			\begin{equation*}
				\da\bar{\E}_\tau(\bar{f},M)(\bar{x}) = \da\bar{\E}_\tau(\bar{g},M)(\bar{x})
				\for \abs{\alpha}\leq 2 \text{ and } M \geq 0.
			\end{equation*}
		\end{enumerate}

		\item[\LA{induction:FK}](\textbf{FK}) We also assume that we are given the Fefferman-Klartag interpolation maps, i.e., a collection of linear maps $ \set{\overline{\Psi}_{\overline{x}} : \overline{x} \in \R^{n-1}} $, where
		\begin{equation*}
			\overline{\Psi}_{\overline{x}} : \R^{\bar{N}} \to \overline{\P}^+
		\end{equation*}
		for each $ x \in \Rn $, such that the following hold.
		\begin{enumerate}[(A)]
			\item There exists a controlled constant $ \overline{D} $, depending only on $ n $, such that for each $ \overline{x} \in \R^{n-1} $, the map $ \overline{\Psi}_{\overline{x}} $ is of depth $ \overline{D} $.
			\item Suppose we are given $ \overline{\phi} \in \R^{\bar{N}} $. Then there exists a function $ \overline{\Phi} \in \ct(\R^{n-1}) $ such that
			\begin{equation*}
				\overline{\jet}^+_{\overline{x}}\overline{\Phi} = \bar{\Psi}_{x}(\phi)\forall \bar{x}\in\R^{n-1},\,
				\norm{\bar{\Phi}}_{\ct(\R^{n-1})} \leq C\norm{\bar{\phi}}_{\ct(\bar{E})}, \text{ and }\bar{\Phi}(\bar{x}) = \bar{\phi}(\bar{x}) \for \bar{x}\in\bar{E}.
			\end{equation*}
			Here, $ \norm{\bar{\phi}}_{\ct(\bar{E})}:= \inf \set{\norm{\tilde{\Phi}}_{\ct(\R^{n-1})} : \tilde{\Phi} \in \ct(\R^{n-1})\text{ and }\tilde{\Phi}\big|_{\bar{E}} = \bar{\phi}}$.

			\item There exists an algorithm that takes the given data set $ \bar{E} $, and then responds to queries.
			
			A query consists of a point $ \bar{x}\in \R^{n-1} $, and the response to the query is the depth-$ \bar{D} $ map $ \bar{\Psi}_{\bar{x}} $, given in its efficient representation.
			
			The one-time work takes $ \bar{C}_1\bar{N}\log\bar{N} $ operations and $ \bar{C}_2\bar{N} $ storage. The work to answer a query is $ \bar{C}_3\log\bar{N} $. Here, $ \bar{C}_1, \bar{C}_2, \bar{C_3} $ depend only on $ n $. 
			
			\item We use $ \bar{\E}_{\infty} $ to denote the operator associated with $ \set{\bar{\Psi}_{\bar{x}}:\bar{x}\in\R^{n-1}} $, mapping from $ C(\bar{E}) $ into $ \ct(\R^{n-1}) $, determined by the relation
			\begin{equation*}
				\bar{\jet}^+_{\bar{x}}\circ \bar{\E}_\infty (\bar{\phi}) \equiv \bar{\Psi}_{\bar{x}}(\bar{\phi})
				\quad \text{ for all } \bar{\phi}:\bar{E}\to\R.
			\end{equation*}
			Thanks to (A), for each $ \bar{x} \in \Rn $, there exists $ \bar{S}(\bar{x}) \subset \bar{E} $ with $ \#\bar{S}(\bar{x}) \leq \bar{D}(n) $, such that for all $ \bar{\phi},\bar{\gamma} :\bar{E}\to \R $ with $ \bar{\phi} = \bar{\gamma} $ on $ \bar{S}(\bar{x}) $, we have
			\begin{equation*}
				\da\bar{\E}_\infty(\bar{\phi})(\bar{x}) = \da\bar{\E}_\infty(\bar{\gamma})(\bar{x})
				\for \abs{\alpha}\leq 2.
			\end{equation*}
			
		\end{enumerate}
		
	\end{itemize}
	
	For the construction of the Fefferman-Klartag maps, see \cite{FK09-Data-1,FK09-Data-2}. Such maps can also be constructed using the techniques in \cite{JL20,JL20-Ext,JL20-Alg}, which are adapted from \cite{FK09-Data-1,FK09-Data-2}.

	We will be working with the Finiteness Constants $ k^\sharp_{n-1,\mathrm{old}} $ and $ k^\sharp_{n,\mathrm{old}} $ in Theorem \ref{thm.fp-sigma}. Their precise values do not matter. We will also be working with a constant $ k^\sharp_{LIP} $ associated with the local interpolation problems, where we take $ k^\sharp_{LIP} \geq (n+2)^2k^\sharp_{n,\mathrm{old}} $. We will remind the readers of these quantities when necessary.

	\section{Preliminary data structure}

	\newcommand{\rbn}{\mathbb{R}^{\bar{N}}}
	\newcommand{\cte}{C(E)}

	Recall Theorem \ref{thm.fp-sigma}. We begin by reviewing some key objects introduced in \cite{FK09-Data-1,FK09-Data-2}, which we will use to effectively approximate $ \sk $ for $ x \in E $. 
	
	We will be working with $ \ctrn $ functions instead of $ \ctrt $ functions. 
	
	Let $ E \subset \Rn $ be a finite set with $ \# E = N $. We assume that $ E $ is labeled: $ E = \set{x_1, \cdots, x_N}  $. We write $ \cte $ to denote the collection of functions $ \phi : E \to \R $, which we can identify with $ \R^N $.

	\subsection{Parameterized approximate linear algebra problems (PALP)}
	
	We equip $ \R^{N} $ with the standard coordinate basis $ \set{\xi_1, \cdots, \xi_{N}} $. We recall the following definition in Section 6 of \cite{FK09-Data-2}.
	
	\begin{definition}\label{def.palp}
		A \underline{parameterized approximate linear algebra problem (PALP for short)} is an object of the form:
		\begin{equation}
			\A = \left[ (\ul{\lambda}_1, \cdots, \ul{\lambda}_{i_{\max}}), (\ul{b}_1, \cdots, \ul{b}_{i_{\max}}), (\epsilon_1, \cdots, \epsilon_{i_{\max}}) \right],
			\label{PALP-def}
		\end{equation}
		where \begin{itemize}
			\item Each $ \ul{\lambda}_i $ is a linear functional on $ \P $, which we will refer to as a ``linear functional'';
			\item Each $ \ul{b}_i $ is a linear functional on $ \cte $, which we will refer to as a ``target functional''; and
			\item Each $ \epsilon_i \in \pos $, which we will refer to as a ``tolerance''.
		\end{itemize}
		Given a PALP $ \A $ in the form \eqref{PALP-def}, we introduce the following terminologies:
		\begin{itemize}
			\item We call $ i_{\max} $ the \ul{length} of $ \A $;
			\item We say $ \A $ \ul{has depth $ D $} if each of the linear functionals $ \ul{b}_i $ on $ \R^{N} $ has depth $ D $ with respect to the basis $ \set{\xi_1, \cdots, \xi_{N}} $ (see Definition \ref{def.depth}).
		\end{itemize}

	\end{definition}

	Recall Definition \ref{def.depth}. We assume that every PALP is \qt{efficiently stored}, namely, each of the target functionals are stored in its efficient representation. (In fact, each PALP admits a \qt{compact representation} in the sense of \cite{FK09-Data-2}. See Remark \ref{rem.depth}.) In particular, given a PALP $ \A $ of the form \eqref{PALP-def} and a target $ \ul{b}_i $ of $ \A $, we have access to a set of indices $ \set{i_1, \cdots, i_D} \subset \set{1,\cdots,N} $, such that $ \ul{b}_i $ is completely determined by its action on $ \set{\xi_{i_1}, \cdots, \xi_{i_D}} \subset \set{\xi_1, \cdots, \xi_N} $. Here $ i_D = \depth(\ul{b}_i) $. We define
	\begin{equation}
		S(\ul{b}_i):= \set{x_{i_1}, \cdots, x_{i_D}} \subset E.
		\label{S(b)-def}
	\end{equation}
	Given a PALP of the form \eqref{PALP-def}, we define
	\begin{equation}
		S(\A) := \bigcup_{i = 1}^{i_{\max}}S(\ul{b}_i) \subset E
		\label{S(A)-def}
	\end{equation}
	with $ S(\ul{b}_i) $ as in \eqref{S(b)-def}.

	\subsection{Blobs and PALPs}
	
	\renewcommand{\K}{\Omega}
	
	\begin{definition}
		A \ul{blob} in $ \P $ is a family $ \vec{\K} = (\K_M)_{M \geq 0} $ of (possibly empty) convex subsets $ \K_M \subset \P $ parameterized by $ M \in \pos $, such that $ M < M' $ implies $ \K_M \subseteq \K_{M'} $. We say two blobs $ \vec{\K} = (\K_M)_{M \geq 0} $ and $ \vec{\K}' = (\K'_M)_{M \geq 0} $ are \ul{$ C $-equivalent} if $ \K_{C^{-1}M}\subset \K'_M \subset \K_{CM} $ for each $ M \in \pos $.
	\end{definition}

	Let $ \A $ be a PALP of the form \eqref{PALP-def}. For each $ \phi \in \cte $, we have a blob defined by
	\begin{equation}
		\begin{split}
			\vec{\K}_\phi(\A)  &= \brac{\K_\phi(\A,M)}_{M \geq 0}
			\text{, where}\\
			\K_\phi(\A,M) &:= \set{P \in \P : \abs{\ul{\lambda}_i(P) - \ul{b}_i(\phi)} \leq M\epsilon_i\for i = 1, \cdots, i_{\max} } \subset V.
		\end{split}
		\label{K1}
	\end{equation}
	In this paper, we will be mostly interested in the centrally symmetric (called \qt{homogeneous} in \cite{FK09-Data-2}) polytope defined by setting $ \phi \equiv 0 $:
	\begin{equation}
		\sigma(\A) := \K_0(\A,1).
		\label{sigma(A)-def}
	\end{equation}
	Note that $ \sigma(\A) $ is never empty, since it contains the zero polynomial.

	\subsection{Essential PALPs and Blobs}

	\begin{definition}
		\renewcommand{\S}{\Sigma} Let $ E \subset \Rn $ be finite. For each $ x \in \Rn $ and $ \phi : E \to \R $, we define a blob
		\begin{equation}
			\begin{split}
				\vec{\S}_\phi(x) &= \brac{\S_\phi(x,M)}_{M\geq 0}\text{ where}\\
				\S_\phi(x,M) &:= \set{ P \in \P : \begin{matrix}
						\text{ There exists } G \in \ctrn \text{ with}\\\norm{G}_{\ctrn} \leq M, G|_E = \phi, \text{ and }\jet_x G= P.
				\end{matrix} }
			\end{split}
			\label{Sigma(x)}
		\end{equation}
	\end{definition}

	It is clear from the definition of $ \sigma $ in \eqref{sigma-def} that
	\begin{equation*}
		\sigma(x,E) = \Sigma_0(x,1).
	\end{equation*}
	Therefore, thanks to Theorem \ref{thm.fp-sigma}, we have, for $ k \geq k^\sharp_{n,\mathrm{old}} $ and $x \in E$,
	\begin{equation}
		C^{-1}\cdot\sk(x,k) \subset \Sigma_0(x,1) \subset C\cdot\sk(x,k),
		\label{sk-Sigma_0}
	\end{equation}
	for some controlled constants $ k_{n,\sigma}^\sharp(n) $ and $ C(n) $.

	We summarize some relevant results from \cite{FK09-Data-2}.
	
	\begin{lemma}\label{lem.FK-palp}
		Let $ E \subset \Rn $ be finite. Using at most $ C(n)N\log N $ operations and $ C(n)N $ storage, we can compute a list of PALPs $ \set{\A(x) : x \in E} $ such that the following hold.
		\begin{enumerate}[(A)]
			\item There exists a controlled constant $ D_0(n) $ such that for each $ x \in E $, $ \A(x) $ has length no greater than $ (n+1) = \dim \P $ and has depth $ D_0 $.
			\item For each given $ x \in \Rn $ and $ \phi \in \ct(E) $, the blobs $ \vec{\K}_\phi(\A(x)) $ in \eqref{K1} and $ \vec{\Sigma}_\phi(x) $ in \eqref{Sigma(x)} are $ C(n) $-equivalent. 
		\end{enumerate}
	\end{lemma}

	See Section 11 of \cite{FK09-Data-2} for Lemma \ref{lem.FK-palp}(A), and Sections 10, 11, and Lemma 34.3 of \cite{FK09-Data-2} for Lemma \ref{lem.FK-palp}(B). 
	
	The main lemma of this section is the following.
	
	\begin{lemma}\label{lem.sigma-main}
		Let $ k^\sharp_{n,\mathrm{old}} $ be as in Theorem \ref{thm.fp-sigma}. There exists a controlled constant $ C(n) $ such that the following holds.
		Let $ E \subset \Rn $ be given. Let $ \set{\A(x) : x \in E} $ be as in Lemma \ref{lem.FK-palp}. Recall the definitions of $ \sigma $ and $ S(\A(x)) $ as in \eqref{sigma-def} and \eqref{S(A)-def}. Then for $ k \geq k^\sharp_{n,\mathrm{old}} $ and $ x \in E $, 
		\begin{equation*}
			C^{-1}\cdot \sigma(x,S(\A(x))) \subset \sk(x, k) \subset C\cdot \sigma(x,S(\A(x))).
		\end{equation*}
	\end{lemma}
	
	\begin{proof}
		\renewcommand{\ksh}{{k^\sharp}}
		For centrally symmetric $ \sigma,\sigma' \subset \P $, we write $ \sigma \approx \sigma' $ if there exists a controlled constant $  C(n) $ such that $ C^{-1}\cdot \sigma \subset \sigma' \subset C\cdot \sigma $. Thus, we need to show $ \sigma(x,\A(x)) \approx \sk(x, k^\sharp_{n,\mathrm{old}}) $ for $ x \in E $. 
		
		Thanks to Theorem \ref{thm.fp-sigma}, Lemma \ref{lem.FK-palp}(B) (applied to $ \phi \equiv 0 $), \eqref{sigma(A)-def}, and \eqref{sk-Sigma_0}, we have
		\begin{equation}
			\sk(x, k) \approx \sk(x,E) \approx \K_0(\A(x),1) = \sigma(\A(x))\for x \in E.
			\label{3.8}
		\end{equation}
		Therefore, it suffices to show that
		\begin{equation*}
			\sigma(x,S(\A(x))) \approx \sigma(\A(x))
			\for x \in E.
		\end{equation*}
		From \eqref{3.8} and the definition of $ \sigma $ in \eqref{sigma-def}, we see that
		\begin{equation*}
			\sigma(\A(x)) \subset C\cdot \sigma(x,E) \subset C\cdot \sigma(x,S(\A(x))).
		\end{equation*}
		It remains to show that
		\begin{equation*}
			\sigma(x,S(\A(x))) \subset C\cdot \sigma(\A(x)) .
		\end{equation*}
		
		Let $ x \in E $ and let $ P \in \sigma(x,S(\A(x))) $. Then there exists $ \phi \in \ctrn $ such that $ \norm{\phi}_{\ctrn} \leq 1 $, $ \phi(x) = 0 $ for all $ x \in S(\A(x)) $, and $ \jet_x(\phi) = P $. Note that $ \phi|_E \in \ctrn $. We abuse notation and write $ \phi $ in place of $ \phi|_E $ when there is no possibility of confusion.
		
		It is clear from the definition of $ \Sigma_\phi(x,M) $ in \eqref{Sigma(x)} that
		\begin{equation*}
			P \in \Sigma_\phi(x,1).
		\end{equation*}
		By Lemma \ref{lem.FK-palp}(B), we have
		\begin{equation*}
			P \in \K_\phi(\A(x),C)
		\end{equation*}
		with $ \K_\phi(\A(x),C) $ as in \eqref{K1}. In particular, we have
		\begin{equation}
			\abs{\ul{\lambda}_i(P) - \ul{b}_i(\phi)} \leq C\epsilon_i
			\for i = 1, \cdots, L = \mathrm{length}(\A(x)).
			\label{3.14}
		\end{equation}
		Here, the $ \ul{\lambda}_1, \cdots, \ul{\lambda}_L$,$ \ul{b}_1, \cdots, \ul{b}_L $, and $ \epsilon_1, \cdots, \epsilon_L $, respectively, are the linear functionals, target functionals, and the tolerance of $ \A(x) $. However, by the definition of $ S(\A(x)) $ in \eqref{S(A)-def} and the fact that $ \phi \equiv 0 $ on $ S(\A(x)) $, we see that \eqref{3.14} simplifies to 
		\begin{equation*}
			\abs{\ul{\lambda}_i(P)} \leq C\epsilon_i
			\for i = 1, \cdots, L = \mathrm{length}(\A(x)).
		\end{equation*}
		This is equivalent to the statement
		\begin{equation*}
			P \in \K_0(\A(x),C) = C\cdot \sigma(\A(x)).
		\end{equation*}
		Lemma \ref{lem.sigma-main} is proved. 
	\end{proof}

	\renewcommand{\K}{\mathcal{K}}

	\section{Calder\'on-Zygmund cubes}

	Let $ \tilde{\sigma} \subset \Rn $ be a symmetric convex set. We define
	\begin{equation}\label{diam-def}
		\diam \tilde{\sigma} := 2\cdot \sup_{u \in \Rn, \abs{u} = 1}p_{\tilde{\sigma}}(u),
	\end{equation}
	where $ p_{\tilde{\sigma}}(u) $ is a gauge function given by
	\begin{equation}
		p_{\tilde{\sigma}}(u) := \sup \set{r \geq 0: ru\subset{\tilde{\sigma}}}. 
		\label{Mink-def}
	\end{equation}

	Let $ \set{\A(x):x \in E} $ be as in Lemma \ref{lem.FK-palp}, and let $ \sigma(\A(x)) \subset \P $ be as in \eqref{sigma(A)-def}. Note that for each $ x \in E $, $ \sigma(\A(x)) \subset \P $ is $ n $-dimensional. Indeed, thanks to Lemma \ref{lem.FK-palp}(B) (with $ \phi \equiv 0 $), any $ P \in \sigma(\A(x)) $, $ x \in E $, must have $ P(x) = 0 $. Thus, for each $ x \in E $, we can identify $ \sigma(\A(x)) $ as a subset of $ \Rn $ via the map
	\begin{equation}
		\sigma(\A(x))\ni P \mapsto (\grad P \cdot e_1,\cdots, \grad P \cdot e_n),
		\label{jet-id}
	\end{equation}
	where $ \set{e_1,\cdots, e_n} $ is the chosen orthonormal system.

	\subsection{OK cubes}

	\begin{definition}\label{def.OK}
		Let $ A_1, A_2 > 0 $ be sufficiently large dyadic numbers. Let $ \set{\A(x):x \in E} $ be as in Lemma \ref{lem.FK-palp}. Let $ Q $ be a dyadic cube. We say \ul{$ Q $ is OK} if the following hold.
		\begin{itemize}
			
			\item Either $ \#(E \cap 5Q) \leq 1 $, or $ \diam\sigma(\A(x)) \geq A_1\dq $ for all $ x \in E \cap 5Q $. Here and below, the $ \diam(\sigma(\A(x))) $ is defined using the formula \eqref{diam-def} via the identification \eqref{jet-id}. 
			
			\item $ \dq \leq A_2^{-1} $. 
		\end{itemize}
	\end{definition}

	The importance of OK cubes is illustrated in the following lemma. Roughly speaking if $ Q $ is OK, then $ E $ lies on a hypersurface near $ Q $ with controlled mean curvature. Moreover, this hypersurface can be realized as the null set of a $ C^2 $ function.
	
	\begin{lemma}\label{lem.diffeo}
		\newcommand{\xb}{\overline{x}}
		Let $ Q $ be OK. Suppose $ E \cap 5Q \neq \void $. Let $ x_0 \in E \cap 5Q $. Let $ u_0 \in \Rn $ be a unit vector such that
		\begin{equation}\label{diffeo-1}
			\diam \sigma(\A(x_0)) = p_{\sigma(\A(x_0))}(u_0),
		\end{equation}
		with $ \diam\sigma(\A(x_0)) $ and $ p_{\sigma(\A(x_0))} $ as in \eqref{diam-def} and \eqref{jet-id}, respectively.
		Let $ \rho $ be a rigid motion of $ \Rn $ given by the simple rotation $ \begin{cases}
		u_0 \mapsto e_n\\
		\text{identity on $(\R u_0 \oplus \R e_n)^\perp$}
		\end{cases} $ 
		and translation $x\mapsto x-x_0$. Then there exists $ \phi \in C^2(\R^{n-1}) $ satisfying the following.
		\begin{align}
			&\rho(E \cap 5Q) \subset \set{(\xb,\phi(\xb)):\xb \in \R^{n-1}}.\label{diffeo-2}\\
			&\abs{\grad_{\xb}^m\phi(\xb)} \leq CA_1^{-1}\dq^{1-m}\for \xb \in \R^{n-1}, m = 1,2, \text{ with $ A_1 $ as in Definition \ref{def.OK}.}\label{diffeo-3}\\
			&x_0 = (\overline{0},\phi(\overline{0}))
			\text{, where $ \overline{0} $ is the origin of $ \R^{n-1} $.}\label{diffeo-4}
		\end{align}
		
		Moreover, let $ \Phi:\Rn \to \Rn $ be defined by 
		\begin{equation*}
			\Phi\circ\rho (\xb,x_n):= (\xb,x_n - \phi(\xb)).
		\end{equation*}
		Then $ \Phi $ is a $ C^2 $ diffeomorphism of $ \Rn $ satisfying $ \Phi(E \cap 5Q) \subset \R^{n-1}\times\set{x_n = 0} $, and $ \abs{\grad^m\Phi}, \abs{\grad^m\Phi^{-1}} \leq CA_1^{-1}\dq^{1-m} $ for $ m = 1,2 $ and $ A_1 $ as in Definition \ref{def.OK}. 
	\end{lemma}

	\begin{proof}
		\newcommand{\xb}{\overline{x}}
		If $ \#(E \cap 5Q) \leq 1 $, then we may take $ \phi $ to be the constant function. The conclusions of the lemma are trivially satisfied.
		
		Assume $ \#(E\cap 5Q) > 1 $. 
		
		Since $ Q $ is OK, we see that $ \diam\sigma(\A(x)) \geq A_1\dq $ for all $ x \in E\cap 5Q $. By Lemma \ref{lem.sigma-main}, we see that
		\begin{equation}
			\diam \sk(x, k^\sharp_{n,\mathrm{old}}) \geq cA_1\dq \quad \forall x \in E \cap 5Q,
			\label{diffeo-5}
		\end{equation}
		with $  k^\sharp_{n,\mathrm{old}} $ as in Theorem \ref{thm.fp-sigma}.
		
		Let $ x_0 $ and $ u_0 $ be as in the hypothesis. Without loss of generality, we may assume that $x_0 = 0$, $u_0 = e_n$, and $ \rho $ is the identity map. Let $ \pi : \Rn \to \R^{n-1} $ be the natural projection that eliminates the last coordinate. 
		
		By \eqref{diffeo-5} and the symmetry of $ \sk $, there exists $ P_0 \in \sk(x_0, k^\sharp_{n,\mathrm{old}}) $ such that
		\begin{equation}\label{diffeo-6}
			\d_nP_0 \geq cA_1\dq
			\text{ and }
			\d_i P_0 = 0 \for i = 1, \cdots, n-1.
		\end{equation}
		
		\begin{claim}\label{claim.graph}
			For any $ S \subset E \cap 5Q $ with $ \#S \leq  k^\sharp_{n,\mathrm{old}} - 1 $, there exists $ \phi^S \in \ct(\R^{n-1}) $ such that 
			\begin{align}
				&S \subset \set{(\xb,\phi^S(\xb)) : \xb \in \R^{n-1}},\text{ and }\label{diffeo-claim-1}\\
				&\abs{\grad_{\xb}^m\phi^S(\xb)} \leq CA_1^{-1}\dq^{1-m}\for \xb \in \R^{n-1}, m = 1,2, \text{ with $ A_1 $ as in Definition \ref{def.OK}}.\label{diffeo-claim-2}
			\end{align}
		\end{claim}

		\begin{proof}[Proof of Claim \ref{claim.graph}]
			Fix $ S \subset E \cap (1+c_G)Q $ with $ \#S \leq k^\sharp_{n-1,\mathrm{old}}  \leq k^\sharp_{n,\mathrm{old}} - 1   $. Here, $ k^\sharp_{n-1,\mathrm{old}} $ and $ k^\sharp_{n,\mathrm{old}} $ are as in Theorem \ref{thm.fp-sigma}. Let $ S_0 := S \cup \set{x_0} $. Then $ \#S_0 \leq k^\sharp_{n-1,\mathrm{old}} + 1 $. Since $ P_0 \in \sk(x_0, k^\sharp_{n,\mathrm{old}}) \subset \sk(x_0,k^\sharp_{n-1,\mathrm{old}}+1) $, there exists $ \Psi \in \ctrn $ such that $ S \subset \set{\Psi = 0} $, $ \norm{\Psi}_\ctrn \leq 1 $, and $ \jet_{x_0}\Psi = P_0 $. For $ A_1 $ sufficiently large, from Taylor's theorem and \eqref{diffeo-6}, we see that
			\begin{equation}\label{diffeo-7}
				\d_n\Psi(x) \geq cA_1\dq \text{ and } \abs{\d_i\Psi(x)} \leq C\dq\quad \for x \in 5Q.
			\end{equation}
			Thanks to \eqref{diffeo-7} and the implicit function theorem, there exists a well-defined function $ \phi^S \in \ct_{loc}(\Rn) $ such that $ S \subset \set{(\xb,\phi^S(\xb)): \xb \in \R^{n-1}} $. This proves \eqref{diffeo-claim-1}.
			
			Let $ i,j \in \set{1,\cdots, n-1} $ and $ \xb \in \pi(5Q) $. We have
			\begin{equation}\label{diffeo-8}
				\begin{split}
					\d_i \phi^S(\xb) &= \frac{\d_i \Psi}{\d_n \Psi}(x),\quad\text{ and }\\
					\d^2_{ij}\phi^S(\xb) &= \frac{
						-(\d_n\Psi)^2\d^2_{ij}\Psi 
						+ 
						\brac{\d_{ni}^2\Psi\d_j\Psi + \d_{nj}^2\Psi\d_i\Psi}(\d_n\Psi)^2 
						-
						\d_n^2\Psi\d_i\Psi\d_j\Psi
					}{(\d_n\Psi)^3}(\xb).
				\end{split}
			\end{equation}
			We see that \eqref{diffeo-claim-2} follows from \eqref{diffeo-7} and \eqref{diffeo-8}. Claim \ref{claim.graph} is proved.
			
		\end{proof}
		
		Consider the function $ \phi_0 : \pi(E \cap 5Q) \to \R $ defined by
		\begin{equation*}
			\phi_0 (\xb) := x_n
			\quad \for x = (\xb,x_n) \in E \cap 5Q.
		\end{equation*}
		
		By Claim \ref{claim.graph}, given $ \overline{S} \subset \pi(E \cap 5Q) $ with $ \#\overline{S} \leq  k^\sharp_{n-1,\mathrm{old}} $, there exists $ \phi^{\overline{S}} \in \ct_{loc}(\R^{n-1}) $ such that $ \phi^{\overline{S}}\big|_{\overline{S}} = \phi_0 $, $ \abs{\grad_{\xb}^m\phi^{\overline{S}}(\xb)} \leq CA_1^{-1}\dq^{1-m} $ for $ \xb \in \pi(5Q) $ and $ m = 1,2 $. 
		
		By Theorem \ref{thm.fp-sigma}(A) together with the rescaling $ \tilde{\phi}(\xb) \mapsto \dq^{-1}\tilde{\phi}(\dq\cdot \xb) $, there exists $ \phi \in \ct(\R^{n-1}) $ such that $ \phi|_{\pi(E \cap 5Q)} = \phi_0 $ and $ \abs{\grad_{\xb}^m\phi(\xb)} \leq CA_1^{-1}\dq^{1-m} $ for $ \xb \in \R^{n-1} $ and $ m = 1,2 $. 
		
		The properties of $ \Phi $ follow immediately from those of $ \phi $. 
		
		Lemma \ref{lem.diffeo} is proved. 
		
	\end{proof}

	\subsection{CZ cubes}

	\begin{definition}\label{def.CZ}
		We write $ \Lz $ to denote the collection of dyadic cubes $ Q $ such that both of the following hold.
		\begin{enumerate}[(A)]
			\item $ Q $ is OK (see Definition \ref{def.OK}).
			\item Suppose $ \dq < A_2^{-1} $. Then $ Q^+ $ is not OK.
		\end{enumerate}
	\end{definition}

	\begin{lemma}[Lemma 21.2 of \cite{FK09-Data-2}]\label{lem.CZ0}
		$ \Lz $ forms a cover of $ \Rn $. Moreover, if $ Q, Q' \in \Lz $ with $ (1+2c_G)Q \cap (1+2c_G)Q' \neq \void $, then 
		\begin{equation*}
			C^{-1}\dq \leq \delta_{Q'} \leq C\dq.
		\end{equation*}
		As a consequence, for each $ Q \in \Lz $, 
		\begin{equation*}
			\#\set{Q' \in \Lz : (1+c_GQ') \cap (1+c_G)Q \neq \void} \leq C'.
		\end{equation*}
		Here, $ C, C' $ are controlled constants depending only on $ n $, and $ c_G $ is a fixed small dyadic number, say $ c_G = 2^{-5} $. 
	\end{lemma}

	We recall the following results from \cite{FK09-Data-2}.
	
	\newcommand{\erep}{Rep}

	\begin{lemma}\label{lem.FK-CZ}
		After one-time work using at most $ CN\log N $ operations and $ CN $ storage, we can perform each of the the following tasks using at most $ C\log N $ operations.
		\begin{enumerate}[(A)]
			\item {\rm (Section 26 of \cite{FK09-Data-2})} Given a point $ x \in \Rn $, we compute a list $ \Lambda(x):=\set{Q \in \Lz : (1+c_G)Q \ni x} $.
			
			\item {\rm (Section 27 of \cite{FK09-Data-2})} Given a dyadic cube $ Q\subset \Rn $, we can compute $ Empty(Q) $, with $ Empty(Q) = True $ if $ E \cap 25Q = \void $, and $ Empty(Q) = False $ if $ E \cap 25Q \neq \void $. 
			
			\item {\rm (Section 27 of \cite{FK09-Data-2})} Given a dyadic cube $ Q \subset \Rn $ with $ E \cap 25Q \neq \void $, we can compute $ \erep(Q) \in E \cap 25Q $, with the property that $ \erep(Q) \in E \cap 5Q $ if $ E \cap 5Q \neq \void $. 
		\end{enumerate}

	\end{lemma}

	We define the following subcollections of $ \Lz $.
	
	\begin{equation}\label{Lsk-def}
		\Lsk:= \set{Q \in \Ls : E \cap (1+c_G)Q \neq \void}.
	\end{equation}
	\begin{equation}\label{Ls-def}
		\Ls:= \set{Q \in \Lz : E \cap 5Q \neq \void}.
	\end{equation}
	\begin{equation}\label{Le-def}
		\Le:=\set{Q \in \Lz \setminus \Ls : \dq < A_2^{-1}}\quad\text{with $ A_2 $ as in Definition \ref{def.CZ}.}
	\end{equation}

	\begin{lemma}\label{lem.Ls-Le}
		After one-time work using at most $ CN\log N $ operations and $ CN $ storage, we can perform the following task using at most $ C\log N $ operations: Given $ Q\in \Lz $, we can decide if $ Q \in \Ls $, $ Q \in \Le $, or $ Q \in \Lz\setminus (\Ls \cup\Le) $. 
	\end{lemma}
	
	\begin{proof}
		This is a direct application of Lemma \ref{lem.FK-CZ}(B,C) to $ Q $. 
	\end{proof}
	
	\begin{lemma}\label{lem.mu}
		We can compute a map
		\begin{equation}
			\mu:\Le \to \Ls
			\label{mu-1}
		\end{equation}
		that satisfies
		\begin{equation}
			(1+c_G)\mu(Q) \cap 25Q \neq \void
			\for Q \in \Le\,.
			\label{mu-2}
		\end{equation}
		The one-time work uses at most $ CN\log N $ operations and $ CN $ storage. After that, we can answer queries using at most $ C\log N $ operations. A query consists of a cube $ Q \in \Le $, and the response to the query is a cube $ \mu(Q) $ that satisfies \eqref{mu-2}. 
		
	\end{lemma}

	\begin{proof}
		
		Suppose $ Q \in \Le $. Then we have $ E \cap 5Q^+ \neq \void $. On the other hand, $ 5Q^+ \subset 25Q $. Hence, $ E \cap 25Q \neq \void $. Therefore, the map $ \erep $ in Lemma \ref{lem.FK-CZ}(C) is defined for $ Q $. 
		
		We set
		\begin{equation}
			x :=  \erep(Q) \subset E \cap 25Q,
			\label{5.3.2}
		\end{equation}
		with $ \erep $ as in Lemma \ref{lem.FK-CZ}. Note that $ x \notin 5Q $, since $ Q \in \Le $. 
		
		Let $ \Lambda(x) \subset \Lz $ be as in Lemma \ref{lem.FK-CZ}(A). Let $ Q' \in \Lambda(x) $. By the defining property of $ \Lambda(x) $ and the fact that $ x \in E $, we have $ Q' \in \Ls $. Set
		\begin{equation*}
			\mu(Q) := Q' \in \Ls.
		\end{equation*}
		By the previous comment, we have
		\begin{equation}
			(1+c_G)\mu(Q) \ni x.
			\label{5.3.1}
		\end{equation}
		Combining \eqref{5.3.2} and \eqref{5.3.1}, we see that $ (1+c_G)\mu(Q)\cap 25Q \neq \void $. \eqref{mu-2} is satisfied.
		
		By Lemma \ref{lem.FK-CZ}(A,C), the tasks $ \Lambda(\cdot) $ and $ \erep(\cdot) $ require at most $ C\log N $ operations, after one-time work using at most $ CN\log N $ operations and $ CN $ storage. Therefore, computing $ \mu(Q) $ requires at most $ C\log N $ operations, after one-time work using at most $ CN\log N $ operations and $ CN $ storage.

		This proves Lemma \ref{lem.mu}.
	\end{proof}
	
	\begin{lemma}\label{lem.uQ}
		After one-time work using at most $ CN\log N $ operations and $ CN $ storage, we can perform the following task using at most $ C\log N $ operations: Given $ Q \in \Ls $, compute an orthonormal frame $ [u_1, \cdots, u_{n-1},u_Q] $ of $ \Rn $, such that the following hold.
		\begin{enumerate}[(A)]
			\item The orthonormal frame $ [u_1, \cdots, u_{n-1},u_Q] $ has the same orientation as $ [e_1, \cdots, e_{n-1}, e_n] $.
			
			\item Let $ \rho $ be the rigid motion given by the simple rotation $ \begin{cases}
		u_0 \mapsto e_n\\
		\text{identity on $(\R u_Q \oplus \R e_n)^\perp$}
		\end{cases} $ and the translation $x\mapsto x-\erep(Q)$. Then there exists a function $ \phi \in \ct(\R^{n-1}) $ that satisfies \eqref{diffeo-2} and \eqref{diffeo-4} with this particular $ \rho $. 
		\end{enumerate}
		
	\end{lemma}

	\begin{proof}
		Fix $ Q \in \Ls $. This means that $ E \cap 5Q \neq \void $. In particular, $ \erep(Q) $ is defined, and by Lemma \ref{lem.FK-CZ}(C),
		\begin{equation*}
			x_0 := \erep(Q) \in E \cap 5Q.
		\end{equation*}
		Computing $ x_0 $ requires at most $ C\log N $ operations, after one-time work using at most $ CN\log N $ operations and $ CN $ storage. 
		
		Let $ \A(x_0) $ be as in Lemma \ref{lem.FK-palp}, and let $ \sigma(\A(x_0)) $ be as in \eqref{sigma(A)-def}. By Lemma \ref{lem.FK-palp}(B) (with $ \phi \equiv 0 $), any $ P \in \sigma(\A(x_0)) $ must satisfy $ P(x_0) = 0 $. By Lemma \ref{lem.FK-palp}(A) and definitions \eqref{K1}, \eqref{sigma(A)-def} of $ \sigma(\A(x_0)) $, we see that $ \sigma(\A(x_0)) $ is an $ n $-dimensional parallelepiped in $ \P $ centered at the zero polynomial. Therefore, we have
		\begin{equation*}
			\diam \sigma(\A(x_0)) = length(\Delta_0),
		\end{equation*}
		where $ \diam $ is defined in \eqref{diam-def} and $ \Delta_0 $ is one of the longest diagonals of $ \sigma(\A(x_0)) $.
		
		Set $ u_Q $ to be a unit vector parallel to $ \Delta_0 $. Lemma \ref{lem.uQ}(B) then follows from Lemma \ref{lem.diffeo}.
		
		Using the Gram-Schmidt process, we can compute the rest of the vectors $ u_1, \cdots, u_{n-1} $ such that $ [u_1, \cdots, u_{n-1},u_Q] $ satisfies Lemma \ref{lem.uQ}(A). Computing $ [u_1, \cdots, u_{n-1},u_Q] $ from $ \sigma(\A(x_0)) $ uses elementary linear algebra, and requires at most $ C $ operations.
		
		Lemma \ref{lem.uQ} is proved.

	\end{proof}
	
	\begin{lemma}\label{lem.rep}
		After one-time work using at most $ CN\log N $ operations and $ CN $ storage, we can perform the following task using at most $ C\log N $ operations: Given $ Q \in \Lz $, we can compute a point $ \xqs \in Q $ such that
		\begin{equation}
			\dist{\xqs}{E} \geq a_0\dq
			\label{xqs}
		\end{equation}
		for some $a_0 = a_0(n,A_1)$
	\end{lemma}

	\begin{proof}
		Let $ Q \in \Lz$ be given. 
		
		Suppose $ Empty(Q) = True $, with $ Empty(\cdot) $ as in Lemma \ref{lem.FK-CZ}(B). We set
		\begin{equation*}
			\xqs := center(Q).
		\end{equation*}
		It is clear that $ \xqs \in Q $ and \eqref{xqs} holds.
		
		Suppose $ Empty(Q) = False $. Let $ x_0 := \erep(Q) \in E \cap 25Q $. 
		
		Suppose $ x_0 \notin 5Q $. Then $ E \cap 5Q = \void$ by Lemma \ref{lem.FK-CZ}(C). Again, we set
		\begin{equation*}
			\xqs:= center(Q).
		\end{equation*}

		Suppose $ x_0 \in 5Q $. This means that $ Q \in \Ls $ with $ \Ls $ as in \eqref{Ls-def}. Let $ u_Q $ be as in Lemma \ref{lem.uQ}. 
		
		\newcommand{\xb}{\overline{x}}
		By Lemma \ref{lem.diffeo}, we have $ E \cap 5Q \subset \set{(\xb,\phi(\xb)):\xb \in \R^{n-1}} $ up to the rotation $ u_Q \mapsto e_n $, and the function $ \phi $ satisfies $ \abs{\grad_{\xb}^m\phi} \leq CA_1^{-1}\dq^{1-m} $ for $ m = 1,2 $, with $ A_1 $ as in Definition \ref{def.CZ}. Therefore, by the defining property of $ u_Q $ in Lemma \ref{lem.uQ}, we have 
		\begin{equation*}
			E \cap 5Q \subset \set{y \in \Rn: \abs{(y-x_0)\cdot u_Q} \leq CA_1^{-1}\abs{y-x_0}}=:Z(x_0).
		\end{equation*}
		
		Suppose $ \dist{center(Q)}{Z(x_0)} \geq \dq/1024 $. We set
		\begin{equation*}
			\xqs := center(Q).
		\end{equation*}
		In this case, it is clear that $ \xqs \in Q $ and \eqref{xqs} holds.
		
		Suppose $ \dist{center(Q)}{Z(x_0)} < \dq/1024 $. We set 
		\begin{equation*}
			\xqs:= center(Q) + \frac{\dq}{4}\cdot  u_Q.
		\end{equation*}
		It is clear that $ \xqs \in Q $. For sufficiently large $ A_1 $, we also have $ \dist{\xqs}{Z(x_0)} \geq a_0\dq $ for some $ a_0 $ depending only on $ A_1 $. Thus, \eqref{xqs} holds.
		
		After one-time work using at most $ CN\log N $ operations and $ CN $ storage, the procedure $ Empty(Q) $ requires at most $ C\log N $ operations by Lemma \ref{lem.FK-CZ}(B); the procedure $ \erep(Q) $ requires at most $ C\log N $ operations by Lemma \ref{lem.FK-CZ}(C); computing the vector $ u_Q $ requires at most $ C\log N $ operations; and computing the distance between $ center(Q) $ and $ Z(x_0) $ is a routine linear algebra problem, and requires at most $ C $ operations. 
		
		Lemma \ref{lem.rep} is proved. 
	\end{proof}

	We now turn our attention to $ \Lsk $ as in \eqref{Lsk-def}. 
	
	\begin{lemma}\label{lem.Lsk}
		Using at most $ CN\log N $ operations and $ CN $ storage, we can compute the list $ \Lsk $ as in \eqref{Lsk-def}. 
	\end{lemma}

	\begin{proof}
		This is a direct application of Lemma \ref{lem.FK-CZ}(A) to each $ x \in E $. 
	\end{proof}

	The next lemma states that we can efficiently sort the data contained in cubes in $ \Lsk $. 
	
	\begin{lemma}\label{lem.projection}
		Using at most $ CN\log N $ operations and $ CN $ storage, we can do the following.
		
		For each $ Q \in \Lsk $ with $ \Lsk $ as in \eqref{Lsk-def}, we can compute a list of points
		\begin{equation*}
			Proj_{u_Q^\perp}(E \cap (1+c_G)Q - \erep(Q)) \subset \R^{n-1}.
		\end{equation*}
		Here, $ u_Q $ is as in Lemma \ref{lem.uQ}, $ u_Q^\perp $ is the subspace orthogonal to $ u_Q $, $Proj_{u_Q^\perp} $ is the orthogonal projection onto $ u_Q^{\perp} $, and $ \erep(Q) $ is as in Lemma \ref{lem.FK-CZ}(C).
	\end{lemma}

	\begin{proof}
		By the bounded intersection property in Lemma \ref{lem.CZ0}, we have
		\begin{equation}
			\#(\Lsk) \leq CN.
			\label{5.8.2}
		\end{equation}
		
		From the definitions of $ \Lsk $ and $ \Ls $ in \eqref{Lsk-def} and \eqref{Ls-def}, we see that $ \Lsk \subset \Ls $. Therefore, we can compute $ \erep(Q) $ and $ u_Q^\perp $ for each $ Q \in \Lsk $ using at most $ C\log N $ operations, by Lemma \ref{lem.FK-CZ}(B) and Lemma \ref{lem.uQ}.
		
		Recall the proof of Lemma \ref{lem.Lsk} that we can compute the list $ \Lsk $ by computing each $ \Lambda(x) $ for $ x \in E $, with $ \Lambda(x) $ as in Lemma \ref{lem.FK-CZ}(A). During this procedure, we can store the information $ Q \ni x $ for $ Q \in \Lambda(x) $.
		
		By the bounded intersection property in Lemma \ref{lem.CZ0}, we have
		\begin{equation}
			\sum_{Q \in \Lsk}\#(E \cap (1+c_G)Q) \leq CN.
			\label{5.8.1}
		\end{equation}
		By Lemma \ref{lem.FK-CZ}(A) and \eqref{5.8.1}, we can compute the list
		\begin{equation*}
			\set{E \cap (1+c_G)Q : Q \in \Lsk}
		\end{equation*}
		using at most $ CN\log N $ operations and $ CN $ storage. Then, by Lemma \ref{lem.FK-CZ}(C), Lemma \ref{lem.uQ}, and \eqref{5.8.2}, we can compute the  list
		\begin{equation}
			Proj_{u_Q^\perp}(E \cap (1+c_G)Q - \erep(Q))
			\label{list}
		\end{equation}
		for each $ Q \in \Lsk $
		using at most $ CN\log N $ operations and $ CN $ storage. 
		
		Lemma \ref{lem.projection} is proved.  
		
	\end{proof}

	\begin{lemma}\label{lem.diffeo-jet}
		\newcommand{\xb}{\overline{x}}
		Suppose we are given 
		\begin{itemize}
			\item $ Q\in\Lsk $,
			\item $ E \cap (1+c_G)Q $,
			\item $ u_Q $ as in Lemma \ref{lem.uQ}, and
			\item $ Proj_{u_Q^\perp}(E \cap (1+c_G)Q - \erep(Q))\subset\R^{n-1} $.
		\end{itemize}
		Let $ \bar{N}:= \#(E\cap(1+c_G)Q) $. After one-time work using at most $ C\bar{N}\log \bar{N} $ operations and $ C\bar{N} $ storage, we can compute a function $ \phi \in \ct(\R^{n-1}) $ and a query algorithms with the following properties.
		\begin{enumerate}[(A)]
			\item\label{diffeo-jet-1} $ \rho(E \cap (1+c_G)Q) \subset \set{(\xb,\phi(\xb)):\xb \in \R^{n-1}} $. Here, $ \rho $ is the rigid motion specified by the planar rotation $ u_Q \mapsto e_n $ and the translation $x\mapsto x-\erep(Q)$, with $ u_Q $ as in Lemma \ref{lem.uQ}.
			\item\label{diffeo-jet-2} $ \abs{\grad_{\xb}^m\phi(\xb)} \leq CA_1^{-1}\dq^{1-m} $ for $ m =1,2 $, with $ A_1 $ as in Definition \ref{def.OK}.
			\item\label{diffeo-jet-3} $ \erep(Q) = (\overline{0},\phi(\overline{0})) $, where $ \erep(\cdot) $ is the map in Lemma \ref{lem.FK-CZ}(C) and $ \overline{0} $ is the origin of $ \R^{n-1} $.
		\end{enumerate}
		A query consists of a point $ \overline{x} \in Proj_{u_Q^\perp}(1+c_G)Q \subset \R^{n-1} $. An answer to a query is the two-jet $ \jet^+_{\overline{x}}\phi $. The time to answer a query is $ C\log \bar{N} $.

	\end{lemma}
	
	\begin{proof}
		We set $ \overline{E} := Proj_{u_Q^\perp}(E \cap (1+c_G)Q - \erep(Q)) $. We define a function $ \phi_0 : \overline{E} \to \R $ by
		\begin{equation*}
			\phi_0(\overline{x}) := (x-\erep(Q))\cdot u_Q
			\text{, where } x = (\overline{x},x_n)\in E.
		\end{equation*}
		For each $ \overline{x} \in \R^{n-1} $, let $ \overline{\Psi}_{\overline{x}} : \R^{\overline{N}} \to \overline{\P}^+ $ be as in \eqref{induction:FK} in Section \ref{sect:induction}. We define the function $ \phi $ by specifying
		\begin{equation*}
			\jet^+_{\overline{x}}\phi:= \overline{\Psi}_{\overline{x}}\phi_0.
		\end{equation*}
		We see from \eqref{induction:FK}(B) that $ \phi \in \ct(\R^{n-1}) $. By construction, $ \phi $ satisfies Lemma \ref{lem.diffeo-jet}(A,C). From \eqref{induction:FK}(B), Lemma \ref{lem.diffeo}, and an obvious rescaling, we see that Lemma \ref{lem.diffeo-jet}(B) follows.
		
		By \eqref{induction:FK}(C), the one-time work uses at most $ C\overline{N}\log\overline{N} $ operations and $ C\overline{N} $ storage, and the time to answer a query is $ C\log\overline{N} $. 
	\end{proof}

	\subsection{Compatible jets on CZ cubes}

	Recall that $ \B(x,\delta) = \set{P \in \P : \abs{\da P(x)} \leq \delta^\tma \for \abs{\alpha} \leq 1} $. 
	
	\begin{lemma}\label{lem.sigma-small-2}
		Let $ Q \subset \Lz $ and let $ \xqs \in Q $ be as in Lemma \ref{lem.rep}. Let $ k^\sharp_{n,\mathrm{old}} $ be as in Theorem \ref{thm.fp-sigma} and let $ k \geq (n+2)k^\sharp_{n,\mathrm{old}} $. Then
		\begin{equation}
			\sk(\xqs,k^\sharp_{LIP}) \subset A\cdot \B(\xqs,\dq)
			\label{7.11.1}
		\end{equation}
		with $A = A(n,A_1,A_2)$.
	\end{lemma}
	
	\begin{proof}
		
		First we note that \eqref{7.11.1} holds when $ \dq = A_2^{-1} $ with $ A_2 $ as in Definition \ref{def.CZ}.
		
		Suppose $ \dq < A_2^{-1} $. Then $ Q^+ $ is not OK (see Definition \ref{def.OK}). Combining this with Lemma \ref{lem.sigma-main}, we see that $ \#(E\cap 5Q) > 1 $ and there exists $ \hat{x} \in E \cap 5Q $ such that 
		\begin{equation}\label{7.11.2}
			\diam\sk(\hat{x},k^\sharp_{n,\mathrm{old}}) < CA_1\dq.
		\end{equation}
		Using Helly's Theorem \ref{thm.Helly} and a similar argument as in Lemma \ref{lem.helly-1}, we see that given $ P_0 \in \sk(\xqs,k) $, there exists $ P \in \sk(\hat{x},k^\sharp_{n,\mathrm{old}}) $ such that
		\begin{equation}\label{7.11.3}
			\abs{\da(P_0 - P)(\xqs)}, \, \abs{\da(P_0 - P)(\hat{x})} \leq C\dq^\tma 
			\for \abs{\alpha} \leq 1.
		\end{equation}
		Combining \eqref{7.11.2} and \eqref{7.11.3}, we see that any $ P \in \sk(\xqs,k) $ satisfies $ \abs{\da P(\xqs)} \leq C\dq^\tma $ for $ \abs{\alpha} \leq 1 $, which is precisely \eqref{7.11.1}.
		
	\end{proof}
	
	The next lemma is crucial in controlling the derivatives when we patch together nearby local solutions. 
	
	\begin{lemma}\label{lem.nearby-jet}
		Let $ Q, Q' \in \Lz $ and $ \xqs, x_{Q'}^\sharp  $ be as Lemma \ref{lem.rep}. Let $ f : E \to \itau $ be given, and let $ \Gk $ be as in \eqref{Gk-def}. Let $ k^\sharp_{LIP} \geq (n+2)^2k^\sharp_{n,\mathrm{old}} $ with $ k^\sharp_{n,\mathrm{old}} $ as in Theorem \ref{thm.fp-sigma}. Let $ P \in \Gk(\xqs,k^\sharp_{LIP},f,M) $ and $ P' \in \Gk(x_{Q'}^\sharp,k^\sharp_{LIP},f,M) $. Then for $ x \in 25Q \cup 25Q' $, 
		\begin{equation}\label{7.13.0}
			\abs{\da(P - P')(x)}  \leq A M \brac{\abs{\xqs-x_{Q'}^\sharp} + \dq + \delta_{Q'}^\sharp}^{\tma}\for\abs{\alpha} \leq 2
		\end{equation}
		with $ A = A(n,A_1,A_2) $.
	\end{lemma}
	
	\begin{proof}
		We write $ A, A' $, etc., to denote constants depending only on $n,A_1,A_2$. 
		
		Set
		\begin{equation*}
			\delta := \max\set{\abs{\xqs - x_{Q'}^\sharp}, \dq, \delta_{Q'}}.
		\end{equation*}
		By Lemma \ref{lem.CZ0} and Lemma \ref{lem.rep}, we see that
		\begin{equation}\label{lem.7.13.1}
			\abs{x - \xqs}, \abs{x - x_{Q'}^\sharp} \leq A\delta\quad \for x \in 25Q \cup 25Q'.
		\end{equation}
		
		Lemma \ref{lem.helly-1} applied to $ P $ yields $ P_0 \in \Gk(x_{Q'}^\sharp,(n+1)k^\sharp_{n,\mathrm{old}},f,M) $ such that 
		\begin{equation}\label{lem.7.13.2}
			\abs{\da (P - P_0)(x_{Q'}^\sharp)} \leq CM\abs{\xqs - x_{Q'}^\sharp} \leq AM\delta^\tma
			\for \abs{\alpha} \leq 2. 
		\end{equation}
		
		Observe that $ \Gk(\xqs,k^\sharp_{n,\mathrm{old}},f,M) - \Gk(\xqs,k^\sharp_{n,\mathrm{old}},f,M)\subset CM\cdot\sk(\xqs,(n+1)k^\sharp_{n,\mathrm{old}}) $\footnote{For $ X,Y \subset V $ with $ V $ a vector space, we write $ X-Y $ to denote the set $ \set{z : z = x-y, x \in X, y \in Y} $.}. Therefore, we have
		\begin{equation}\label{lem.7.13.3}
			P' - P_0 \in CM\cdot \sk(x_{Q'}^\sharp, (n+1)k^\sharp_{n,\mathrm{old}}). 
		\end{equation}
		By Lemma \ref{lem.sigma-small-2}, we see that
		\begin{equation}\label{lem.7.13.4}
			\abs{\da(P' - P_0 )(x_{Q'}^\sharp)} \leq CM\delta_{Q'}^\tma \leq AM\delta^\tma
			\for \abs{\alpha} \leq 2.
		\end{equation}
		
		Taylor's Theorem, combined with \eqref{lem.7.13.1}, \eqref{lem.7.13.2} and \eqref{lem.7.13.4}, yields
		\begin{equation}\label{7.13.5}
			\abs{\da(P - P')(x')} \leq AM\dq^\tma \for \abs{\alpha} \leq 2 \text{ and } x' \in 25Q'.
		\end{equation}
		By Taylor's Theorem, \eqref{lem.7.13.1}, and \eqref{7.13.5}, we have
		\begin{equation}\label{7.13.6}
			\abs{\da(P - P')(x)} \leq AM\dq^\tma \for \abs{\alpha} \leq 2 \text{ and } x \in 25Q.
		\end{equation}
		Estimate \eqref{7.13.0} follows from \eqref{7.13.5} and \eqref{7.13.6}.
	\end{proof}

	\section{Local interpolation problem}

	\subsection{Distortion estimate}

	\begin{lemma}\label{lem.distortion}
		Let $ 0 < \delta \leq 1 $. Let $ \Psi : \Rn \to \Rn $ be a $ C^2 $-diffeomorphism such that 
		\begin{equation*}
			\abs{\grad^m\Psi(x)} \leq A\delta^{1-m} \quad\for m = 1,2, \text{ and } x \in \Rn.
		\end{equation*}
		Let $ \Omega \subset \Rn $ be a domain and let $ F \in \ct(\overline{\Omega}) $. Suppose
		\begin{equation*}
			\abs{\da F(x)} \leq M\dq^{2-\abs{\alpha}} \quad \for \abs{\alpha} \leq 2 \text{ and } x \in \Omega.
		\end{equation*}
		Then 
		\begin{equation}\label{8.1.0}
			\abs{\da(F\circ\Psi)(x)} \leq C(n)A M\delta^{\tma} \for\abs{\alpha}\leq 2 \text{ and } x \in \Psi^{-1}(\bar{\Omega}).
		\end{equation}
	\end{lemma}
	
	\begin{proof}
		We expand $ \Psi = (\Psi_1, \cdots, \Psi_n) $ in coordinates. Then
		\begin{equation*}
			\begin{split}
				\d_i (F\circ \Psi) &=\sum_{k = 1}^n \d_i\Psi_k \cdot \d_k F\circ \Psi   \text{ and } \\
				\d_{ij}(F\circ\Psi) &= \sum_{k,l = 1}^nc_{k,l} \cdot \d_i\Psi_k \cdot\d_j\Psi_l\cdot\d_{kl}^2F\circ\Psi + \sum_{k = 1}^n\d_{ij}^2\Psi_k\cdot\d_{k}F\circ\Psi.
			\end{split}
		\end{equation*}
		Then \eqref{8.1.0} follows from the derivative estimates on $ F $ and $ \Psi $.
	\end{proof}

	\subsection{Local clusters}

	The next lemma shows how to relay local information to the point $ \xqs $.

	\begin{lemma}\label{lem.SQ}
		Let $ Q \in \Ls $. Let $ \xqs $ be as in Lemma \ref{lem.rep}. Let $ x \in E \cap 5Q $. Let $ \A(x) $ be as in Lemma \ref{lem.FK-palp}. Let $ S(\A(x)) $ be as in \eqref{S(A)-def}. Let $ k^\sharp_{n,\mathrm{old}} $ be as in Theorem \ref{thm.fp-sigma}. Then 
		\begin{equation}
			\sigma(\xqs,S(\A(x))) \subset A\cdot \sk(\xqs,k^\sharp_{n,\mathrm{old}})
			\label{5.6.0}
		\end{equation}
		with $A = A(n,A_1,A_2)$.
	\end{lemma}
	
	\begin{proof}
	    We write $A$, $a$, etc., to denote quantities depending only on $n,A_1, A_2$. 
        	
		\newcommand{\sax}{S(\A(x))}
		\renewcommand{\set}[1]{\{#1\}}
		Fix $ x $ as in the hypothesis. By our choice of $ \xqs $ in Lemma \ref{lem.rep}, we have
		\begin{equation}
			\abs{\xqs - x} \geq a\dq. 
			\label{5.10.1}
		\end{equation}

		Let $ P_0 \in \sigma(\xqs, S(\A(x))) $.
		By the definition of $ \sigma $, there exists $ \phi\in \ctrn $ with $ \norm{\phi}_{\ctrn} \leq 1 $, $ \phi|_{S(\A(x))} = 0 $, and $ \jet_\xqs\phi \equiv P_0 $. Set $ P :\equiv \jet_x\phi $. Then
		\begin{equation*}
			P \in \sigma(x,\sax). 
		\end{equation*}
		Since $ x \in E $, by Lemma \ref{lem.sigma-main}, we have
		\begin{equation*}
			P \in \sk(x,,k^\sharp_{n,\mathrm{old}}). 
		\end{equation*}
		
		Let $ S \subset E $ with $ \#(S) \leq k^\sharp_{n,\mathrm{old}} $. By the definition of $ \sk $ in \eqref{sigma-def} and Taylor's Theorem \ref{thm.WT}(A), there exists a Whitney field $ \vec{P} = (P,(P^y)_{y \in S}) \in W(S\cup\set{x}) $, with $ \norm{\vec{P}}_{W(S\cup\set{x})} \leq C $ and $ P^y(y) = 0 $ for $ y \in S $. 
		
		Consider another Whitney field $ \vec{P}_0 = (P_0, (P^y)_{y \in S}) \in W(S\cup\set{\xqs}) $ defined by replacing $ P $ by $ P_0 $ in $ \vec{P} $. By Whitney's Extension Theorem \ref{thm.WT}(B), it suffices to show that $ \vec{P}_0 $ satisfies 
		\begin{equation}\label{5.10.2}
			P^y(y) = 0\for y \in S, \text{ and }
		\end{equation}
		\begin{equation}\label{5.10.3}
			\norm{\vec{P}_0}_{W(S\cup\set{\xqs})} \leq C.
		\end{equation}
		
		Note that \eqref{5.10.2} is obvious by construction.
		
		We turn to \eqref{5.10.3}. Since $ P_0 = \jet_\xqs\phi $ and $ P = \jet_x\phi $, Taylor's theorem implies 
		\begin{equation}
			\abs{\d^\alpha (P - P_0)(\xqs)}, \abs{\d^\alpha(P - P_0)(x)} \leq C\abs{x - \xqs}^{2-\abs{\alpha}}
			\for\abs{\alpha} \leq 1.
			\label{5.10.4}
		\end{equation}
		Since the Whitney field $ \vec{P} = (P,(P^y)_{y \in S}) $ satisfies $ \norm{\vec{P}}_{W(S\cup\set{x})} \leq C $, we have
		\begin{equation}
			\norm{(P^y)_{y \in S}}_{W(S)} \leq C, 
			\label{5.10.5}
		\end{equation}
		and
		\begin{equation}
			\abs{\d^\alpha(P - P^y)(x)}, \abs{\d^\alpha(P - P^y)(y)} \leq C\abs{x - y}^{2-\abs{\alpha}}
			\for\abs{\alpha} \leq 2, y \in S.
			\label{5.10.6}
		\end{equation}
		Applying the triangle inequality to \eqref{5.10.4} and \eqref{5.10.6}, and using \eqref{5.10.1}, we see that
		\begin{equation}
			\abs{\d^\alpha (P_0 - P^y)(\xqs)}, \abs{\d^\alpha (P_0 - P^y)(y)} \leq A\abs{\xqs - y}^{2-\abs{\alpha}}
			\for\abs{\alpha} \leq 1.
			\label{5.10.7}
		\end{equation}
		Moreover, since $ P_0 \in \sigma(\xqs,\sax) $, we have
		\begin{equation}
			\abs{\d^\alpha P_0(\xqs)} \leq 1
			\for\abs{\alpha} \leq 1.
			\label{5.10.8}
		\end{equation}
		Then, \eqref{5.10.3} follows from \eqref{5.10.5}, \eqref{5.10.7}, and \eqref{5.10.8}. 
		
		Lemma \ref{lem.SQ} is proved. 
		
	\end{proof}

	\newcommand{\ssq}{{S^\sharp(Q)}}
	\newcommand{\qqs}{{\Q^\sharp}}
	\newcommand{\mqs}{{\M_Q^\sharp}}
	
	Let $ Q \in \Ls $ with $ \Ls $ as in \eqref{Ls-def}. Let $ \A(x), x \in E $ be as in Lemma \ref{lem.FK-palp}. Let $ S(\A(x)) $ be as in \eqref{S(A)-def}. Let $ \erep(Q) $ be as in Lemma \ref{lem.FK-CZ}(C). Let $ \xqs $ be as in Lemma \ref{lem.rep}. We set
	\begin{equation}
		\ssq := S(\A(\erep(Q))) \cup \{\xqs\}. 
		\label{ssq-def}
	\end{equation}
	Thanks to Lemma \ref{lem.FK-palp}(A), we have
	\begin{equation}\label{ssq-bd}
		\#\ssq \leq C(n).
	\end{equation}

	\subsection{Transition jet}

	Recall Section \ref{sect:quad}. For $ S \subset E $ and $ \tau > 0 $, we define the following functions on $ W(S) $:
	\begin{equation}
		\begin{split}
			&\begin{split}
				\mathcal{L} : W(S) &\to \pos \\
				(P^x)_{x \in S} &\mapsto \sum_{x \in S,\abs{\alpha} \leq 1}\abs{\da P(x)} + \sum_{x,y \in S, x \neq y, \abs{\alpha}\leq 1}\frac{\abs{\da(P^x - P^y)(x)}}{\abs{x-y}^\tma}\,,\quad\text{ and }
			\end{split}
			\\
			&\begin{split}
				\mathcal{M}_\tau: W(S,\tau) &\to \pos\\
				(P^x)_{x \in S} &\mapsto \sum_{x \in S, \abs{\alpha} = 1}\frac{\abs{\da P^x}^2}{\tau - P(x)} + \frac{\abs{\da P^x}^2}{\tau + P(x)}\,.
			\end{split}
		\end{split}
		\label{LM-0}
	\end{equation}
	We adopt the conventions that $ \frac{0}{0} = 0 $ and $ \frac{a}{0} = \infty $ for $ a > 0 $. Note that $ \mathcal{L} $ is a norm on $ W(S) $.
	
	Let $ Q \in \Ls $ with $ \Ls $ as in \eqref{Ls-def}. Let $ \xqs $ be as in Lemma \ref{lem.rep}. Let $ \ssq $ be as in \eqref{ssq-def}. We set $ x_0 := \erep(Q) $ with $ \erep $ as in Lemma \ref{lem.FK-CZ}(C). Note that $ x_0 \in \ssq $, thanks to Lemma \ref{lem.sigma-main} and the definition of $ \sigma $. 
	
	\newcommand{\Af}{\mathbb{A}_f}
	Let $ f : E \to \itau $ be given. We consider the following spaces.
	\begin{equation}
		\begin{split}
			\Af^{-\tau} &:= \set{  (P^y)_{y \in \ssq} \in W(\ssq,\tau) : \begin{matrix*}[l]
					P^{\xqs}\equiv-\tau \quad\text{ and }\\
					P^x(x) = f(x) \for x \in \ssq \cap E
			\end{matrix*} },\\
			\Af^{\tau} &:= \set{  (P^y)_{y \in \ssq} \in W(\ssq,\tau) : \begin{matrix*}[l]
					P^{\xqs}\equiv\tau \quad\text{ and }\\
					P^x(x) = f(x) \for x \in \ssq \cap E
			\end{matrix*} }, \text{ and }\\
			\Af^{0} &:= \set{  (P^y)_{y \in \ssq\cap E} \in W(\ssq \cap E,\tau) : \begin{matrix*}[l]
					P^x(x) = f(x) \for x \in \ssq \cap E
			\end{matrix*} }.
		\end{split}
		\label{Af-def}
	\end{equation}
	
	Note that $ \Af^{-\tau} $ and $ \Af^{\tau} $ are affine subspaces of $ W(\ssq) $, and $ \Af^0 $ is an affine subspace of $ W(\ssq\cap E) $. All three depend only on $ \tau $ and $ f|_{\ssq\cap E} $. 
	
	We will be considering the following minimization problems.
	\begin{itemize}
		\item[\textbf{MP}($ -\tau $)] Let $ S = \ssq $ in \eqref{LM-0}. Minimize $ \mathcal{L}+\mathcal{M}_\tau $ over $ \Af^{-\tau} $.
		\item[\textbf{MP}($ \tau $)] Let $ S = \ssq $ in \eqref{LM-0}. Minimize $ \mathcal{L}+\mathcal{M}_\tau $ over $ \Af^{\tau} $.
		\item[\textbf{MP}($ 0 $)] Let $ S = \ssq\cap E $ in \eqref{LM-0}. Minimize $ \mathcal{L}+\mathcal{M}_\tau $ over $ \Af^{0} $.
	\end{itemize}
	
	For $ \star = -\tau, \tau,0 $, we say a Whitney field $ \vp \in \Af^\star $ is an \underline{approximate minimizer} of \textbf{MP}($ \star $) if
	\begin{equation*}
		(\mathcal{L}+\mathcal{M}_\tau)(\vp) \leq C(n)\cdot\inf\set{ (\mathcal{L}+\mathcal{M}_\tau)(\vp') : \vp' \in \Af^\star }.
	\end{equation*}

	\begin{remark}\label{rem.transition-jet-0}
		By Section \ref{sect:quad}, \textbf{MP}($ \star $) can be reformulated as a convex quadratic programming problem with affine constraint and can be solved efficiently using at most $ C(n) $ operations, since the size of $ \ssq $ is controlled. Thus, we can find the approximate minimizers for \textbf{MP}($ \star $) using at most $ C(n) $ operations. Computing $ \ssq $ requires at most $ C\log N $ operations after one-time work using $ CN\log N $ operations and $ CN $ storage, since it involves computing the point $ \xqs $ as in Lemma \ref{lem.rep}.
	\end{remark}

	For future reference, we fix these approximate minimizers:
	\begin{itemize}
		\item[\LA{approximate-minimizer}] For $ \star = -\tau,\tau,0 $, let $ \vp[Q,\star] $ be the approximate minimizer of \textbf{MP}($ \star $) solved via the method in Section \ref{sect:quad}.
	\end{itemize}
	
	Notice that the approximate minimizer for \textbf{MP}(0) contains no information at $ \xqs $. The next lemma takes care of this gap. 
	
	\begin{lemma}\label{lem.MP0-jet}
		Let $ Q \in \Ls $. Let $ \xqs $ be as in Lemma \ref{lem.rep}. Let $ f : E \to \itau $ with $ \norm{f}_{\ctet} \leq M $. Let $ \vp = \vp[Q,0] $ be as in \eqref{approximate-minimizer} above with $ \star = 0 $. Let $ x_0 := \erep(Q) $ with $ \erep(Q) $ as in Lemma \ref{lem.FK-CZ}(C). Let $ \T_{w,\tau} $ be the $ \tau $-constrained Whitney extension operator as in Theorem \ref{thm.WT-tau} associated with the singleton $ S = \set{x_0} $. Then
		\begin{equation*}
			\jet_\xqs \circ \T_{w,\tau} (P^{x_0})  \in \G(\xqs,\ssq\cap E, f,CM).
		\end{equation*}
		Here, $ C $ depends only on $ n $.
	\end{lemma}

	\begin{proof}
		We set $ P^\sharp:\equiv \jet_\xqs \circ \T_{w,\tau} (P^{x_0}) $. We adjoin $ P^\sharp $ to $ \vp $ to form
		\begin{equation*}
			\vp^\sharp := (P^\sharp,\vp) \in W(\ssq).
		\end{equation*}
		By Theorem \ref{thm.WT-tau}(B), it suffices to show that $ \vp^\sharp \in W(\ssq,\tau) $ and $ \norm{\vp^\sharp}_{W(\ssq,\tau)} \leq CM $.
		
		Since $ \vp $ is an approximate minimizer of \textbf{MP}(0) and $ \norm{f}_{\ctet} \leq M $, we have
		\begin{equation}\label{8.3.0}
			\norm{\vp}_{W(\ssq\cap E,\tau)} \leq CM.
		\end{equation}
		Since $ P^{x_0} $ is a component of $ \vp $, 
		\begin{equation*}
			\T_{w,\tau}(P^{x_0}) \in \ctrt \text{ and }
			\norm{\T_{w,\tau}(P^{x_0})} \leq CM.
		\end{equation*}
		Therefore, we have
		\begin{equation}\label{8.3.1}
			P^\sharp \in \K_\tau(\xqs,CM). \quad\text{(Recall Definition  \ref{def.Ktau}.)}
		\end{equation}
		Thus, $ \vp^\sharp \in W(\ssq,\tau) $.
		
		For $ x \in \ssq \cap E $ and $ \abs{\alpha} \leq 1 $, we have
		\begin{equation*}
			\abs{\da(P^x - P^\sharp)(x)} \leq \abs{\da(P^x - P^{x_0})(x)} + \abs{\da(P^{x_0}-\jet_\xqs\circ \T_{w,\tau}(P^{x_0}))(x)}.
		\end{equation*}
		Using \eqref{8.3.0} to estimate the first term and Taylor's theorem to the second, we have
		\begin{equation}
			\abs{\da(P^x - P^\sharp)(x)} \leq CM\brac{\abs{x-x_0}+\abs{x_0 - \xqs}}^\tma \leq C'M\abs{x-\xqs}^\tma.
			\label{8.3.2}
		\end{equation}
		For the last inequality, we use the fact that $ \dist{\xqs}{E} \geq c\dq $, thanks to Lemma \ref{lem.rep}. Applying Taylor's theorem to estimate \eqref{8.3.2}, we have
		\begin{equation}\label{8.3.3}
			\abs{\da(P^x - P^\sharp)(\xqs)} \leq CM\abs{x-\xqs}^\tma.
		\end{equation}
		Combining \eqref{8.3.0}--\eqref{8.3.3}, we see that $ \norm{\vp^\sharp}_{W(\ssq,\tau)} \leq CM $. Lemma \ref{lem.MP0-jet} is proved. 
		
	\end{proof}

	We fix a large parameter $A_T$ exceeding a constant depending only on $n$. For $ Q \in \Ls $ and $ \xqs $ as in Lemma \ref{lem.rep}, we define a map
	\begin{equation}\label{TQ-def}
		\T_{\tau,Q}:\ctet \times \pos \to \P
	\end{equation}
	via the following rules. Let $ (f,M) \in \ctet \times \pos $ be given.
	\begin{itemize}
		\item[\LA{-tau}] Let $ \mathcal{L} $ and $ \mathcal{M}_\tau $ be as in \eqref{LM-0} with $ S = \ssq $, and let $ \vp = \vp[Q,-\tau] $ be as in \eqref{approximate-minimizer}. Suppose $ (\mathcal{L}+\mathcal{M}_\tau)(\vp) \leq A_TM $. Then we set $ \T_{\tau,Q}(f,M):\equiv -\tau $.

		\item[\LA{tau}] Suppose the last condition (\ref{-tau}) fails. Let $ \mathcal{L} $ and $ \mathcal{M}_\tau $ be as in \eqref{LM-0} with $ S = \ssq $, and let $ \vp = \vp[Q,\tau] $ be as in \eqref{approximate-minimizer}. Suppose $ (\mathcal{L}+\mathcal{M}_\tau)(\vp) \leq A_TM $. Then we set $ \T_{\tau,Q}(f,M):\equiv \tau $.

		\item[\LA{0zero}] Suppose both conditions (\ref{-tau}) and (\ref{tau}) above fail. Let $ \vp = \vp[Q,0] $ be as in \eqref{approximate-minimizer}. We set $ \T_{\tau,Q}(f,M):\equiv \jet_\xqs\circ\T_{w,\tau}(P^{\erep(Q)}) $. Here, $ P^{\erep(Q)} $ is the component of $ \vp $ corresponding to the point $ \erep(Q) $, with $ \erep(Q) $ as in Lemma \ref{lem.FK-CZ}(C), and $ \T_{w,\tau} $ is the $ \tau $-constrained Whitney extension operator in Theorem \ref{thm.WT-tau} associated with the singleton $ S = \set{\erep(Q)} $.
	\end{itemize}

	The main lemma of this section is the following.
	
	\begin{lemma}\label{lem.transition-jet}
		Let $ Q \in \Ls $ and $ \xqs $ be as in Lemma \ref{lem.rep}. Let $ k^\sharp_{LIP} = (n+2)^2k^\sharp_{n,\mathrm{old}} $ with $ k^\sharp_{n,\mathrm{old}} $ as in Theorem \ref{thm.fp-sigma}. Let $ \T_{\tau,Q} $ be as in \eqref{TQ-def}. Let $ (f,M) \in \ctet\times\pos $ with $ \norm{f}_\ctet \leq M $. Then
		\begin{equation*}
			\T_{\tau,Q}(f,M) \in \Gk(\xqs,k_{LIP}^\sharp,f,AM)
		\end{equation*}
		with $A = A(n,A_T)$. 
	\end{lemma}

	\begin{proof}
	    We write $A$, $A'$, etc., to denote quantities depending only on $n$ and $A_T$.
	
		\newcommand{\klip}{{k_{LIP}^\sharp}}
		Since $ \norm{f}_\ctet \leq M $, we see that $ \Gk(x,\klip,f,2M) \neq \void $. Therefore, the hypothesis of Lemma \ref{lem.perturb} is satisfied with $ k = \klip $. 
		
		Recall that $ \T_{\tau,Q} $ is defined in terms of a series of rules \eqref{-tau}, \eqref{tau}, \eqref{0zero}. We analyze them in this order.
		
		\paragraph{Rule for \eqref{-tau}.} Suppose $ \T_{\tau,Q}(f,M) $ is defined in terms of \eqref{-tau}. By Theorem \ref{thm.WT-tau}, there exists $ F \in \ctrt  $ with $ \norm{F}_\ctrn \leq AM $, $ F|_{\ssq \cap E} = f $, and $ \jet_{\xqs} \equiv -\tau $. Let $ \erep $ be the map in Lemma \ref{lem.FK-CZ}(C), and recall that $ \erep(Q) \in \ssq \cap 5Q $. By Taylor's theorem, we have
		\begin{equation*}
			\tau + f(\erep(Q)) = \tau + F(\erep(Q)) \leq AM\dq^2.
		\end{equation*}
		Lemma \ref{lem.perturb} then implies $ -\tau \in \Gk(\xqs,\klip,f,AM) $.

		\paragraph{Rule for \eqref{tau}.} Suppose $ \T_{\tau,Q}(f,M) $ is defined in terms of  \eqref{tau}. By Theorem \ref{thm.WT-tau}, there exists $ F \in \ctrt  $ with $ \norm{F}_\ctrn \leq AM $, $ F|_{\ssq \cap E} = f $, and $ \jet_{\xqs} \equiv \tau $. By Taylor's theorem, we have
		\begin{equation*}
			\tau - f(\erep(Q)) = \tau - F(\erep(Q)) \leq AM\dq^2.
		\end{equation*}
		Lemma \ref{lem.perturb} then implies $ \tau \in \Gk(\xqs,\klip,f,AM) $.

		\paragraph{Rule for \eqref{0zero}.} Suppose $ \T_{\tau,Q}(f,M) $ is defined in terms of \eqref{0zero}. Recall that we have chosen $A_T$ to be sufficiently large in \eqref{-tau} and \eqref{tau}. Taylor's theorem then implies, with $ A_{\mathrm{perturb}} $ as in Lemma \ref{lem.perturb},
		\begin{equation*}
			\min\set{\tau - f(x), \tau + f(x)} \geq CA_{\mathrm{perturb}}M\dq^2 \for x \in E \cap 5Q.
		\end{equation*}
		Thus, the hypothesis of Lemma \ref{lem.perturb}(A) is satisfied. 
		
		Since $ \norm{f}_\ctet \leq M $, there exists $ F\in\ctrt $ with $ \norm{F}_\ctrn \leq 2M $, $ F|_E = f $, and 
		\begin{equation*}
			\jet_\xqs F\in \G(\xqs,E,f,2M).
		\end{equation*}

		By Lemma \ref{lem.MP0-jet}, 
		\begin{equation*}
			\T_{\tau,Q}(f,M) \in \G(\xqs,\ssq\cap E, f,CM).
		\end{equation*}
		
		Therefore, by Lemma \ref{lem.sigma-main}, Lemma \ref{lem.SQ}, and the definition of $ \ssq $ in \eqref{ssq-def},  we see that
		\begin{equation*}
			\jet_\xqs F - \T_{\tau,Q}(f,M) \in CM\cdot \sigma(\xqs,\ssq \cap E) \subset C'M\cdot \sk(\xqs,k^\sharp_{n,\mathrm{old}}) \subset C''M\cdot \sk(\xqs,k^\sharp_{LIP}).
		\end{equation*}
		
		By Lemma \ref{lem.sigma-small-2}, we see that
		\begin{equation*}
			\jet_\xqs F - \T_{\tau,Q}(f,M) \in CM\cdot \B(\xqs,\dq) .
		\end{equation*}
		
		For sufficiently large $ A_{\mathrm{perturb}} $, Lemma \ref{lem.perturb}(A) implies
		\begin{equation*}
			\T_{\tau,Q}(f,M) \in \jet_\xqs F + CM\cdot \B(\xqs,\dq)\subset \Gk(\xqs,k^\sharp_{LIP},f,C'M).
		\end{equation*}
		
		Lemma \ref{lem.transition-jet} is proved. 
		
	\end{proof}

	\subsection{Fixing the parameters \texorpdfstring{$A_*$}{A*}}
	
	
	\begin{itemize}
	    \item[\LA{param-1}] In Definitions \ref{def.OK} and \ref{def.CZ}, we fix $A_1, A_2 \gg C(n)$ so that Lemma \ref{lem.diffeo} holds.
	    
	    \item[\LA{param-2}] Let $A_{\mathrm{perturb}}$ and $A_{\mathrm{flat}}$ be as in Lemma \ref{lem.perturb}. We fix $A_{\mathrm{perturb}}$ so that Lemma \ref{lem.perturb}(A) holds. We then fix $A_{\mathrm{flat}} \gg C(n)A_{\mathrm{perturb}}$.
	    
	    \item[\LA{param-3}] Let $A_T$ be the parameter associated with the map \eqref{TQ-def}. We fix $A_T = c(n)A_{\mathrm{flat}}$. 
	\end{itemize}
	
	Henceforth, we treat all the parameters $A_*$, $a_*$ appeared in the previous sections as controlled constants and write $C_*$, $c_*$ instead.

	\subsection{Local interpolation problem with a prescribed jet}
	
	Recall $ \Lsk $ from \eqref{Lsk-def}. Also recall that $ c_G $ is a small dyadic number fixed in Lemma \ref{lem.CZ0}. 
	
	\begin{lemma}\label{lem.lip-operator}
		Let $ Q \in \Lsk $. There exists a map 
		\begin{equation*}
			\E_{\tau,Q}: \ctet \times \pos \to \ct((1+c_G)Q)
		\end{equation*}
		such that the following hold.
		\begin{enumerate}[(A)]
			\item			Given $ (f,M) \in \ctet\times\pos $ with $ \norm{f}_\ctet \leq M $, we have 
			\begin{enumerate}[(1)]
				\item $ -\tau \leq \E_{\tau,Q}(f,M) \leq \tau $ on $ (1+c_G)Q $,
				\item $ \E_{\tau,Q}(f,M)(x) = f(x) $ for $ x \in E \cap (1+c_G)Q $, 
				\item $ \norm{\E_{\tau,Q}(f,M)}_{\ct((1+c_G)Q)} \leq C(n)M $, and
				\item $ \jet_{\xqs}\circ \E_{\tau,Q}(f,M) \equiv \T_{\tau,Q}(f,M) $, with $ \xqs $ as in Lemma \ref{lem.rep} and $ \T_{\tau,Q} $ as in \eqref{TQ-def}.
			\end{enumerate}

			\item For each $ x \in (1+c_G)Q $, there exists a set $ S_Q(x) \subset E $ with $ S_Q(x) \leq D(n) $, such that given $ f,g \in \ctet $ with $ f|_{S_Q(x)} = g|_{S_Q(x)} $, we have 
		\end{enumerate}
		\begin{equation}
			\da\E_{\tau,Q}(f,M)(x) = \da\E_{\tau,Q}(g,M)(x) \text{ for }\abs{\alpha} \leq 2 \text{ and } M\geq 0.
			\label{8.5.0}
		\end{equation}
	\end{lemma}

	\begin{proof}
		\newcommand{\etq}{\E_{\tau,Q}}
		We fix $ k^\sharp_{LIP} = (n+2)^2k^\sharp_{n,\mathrm{old}} $ with $ k^\sharp_{n,\mathrm{old}} $ as in Theorem \ref{thm.fp-sigma}.
		
		The essential ingredients in the construction of the map $ \etq $ is as follows.
		
		\begin{itemize}
			\item Let $ \Phi:\Rn\to\Rn $ be the $ C^2 $-diffeomorphism associated with $ Q $ defined by $ \Phi\circ\rho^{-1} = (\bar{x},x_n-\phi(\bar{x})) $, with $ \phi $ and $ \rho $ as in Lemma \ref{lem.diffeo-jet}. See also Lemma \ref{lem.diffeo}. In particular, $ \Phi $ satisfies the estimate
			\begin{equation}\label{8.5.phi}
				\eqindent
				\abs{\grad^m \Phi(x)} \leq C\dq^{1-m} \for x \in \Rn\text{ and }m = 1,2.
			\end{equation}
			
			\item Let $ \xqs $ and $ c_0 = a_0 $ be as in Lemma \ref{lem.rep}. Let $ \psi \in \ctrn $ be a cutoff function such that
			\begin{equation}
				\eqindent
				0 \leq \psi \leq 1,\,
				\psi \equiv 1 \text{ near } \xqs,\,
				\supp{\psi} \subset B(\xqs,c_0\dq),\text{ and }
				\abs{\da\psi}\leq C(n)\dq^\tma\for\abs{\alpha} \leq 2.
				\label{8.5.1}
			\end{equation}
			
			\item Define an indicator function 
			\begin{equation}
				\eqindent
				\Delta(f,M,Q):= \begin{cases}
					1&\text{ if $ \T_{\tau,Q}(f,M) $ is not the constant polynomial $ \pm \tau $}\\
					0 &\text{ otherwise}
				\end{cases}.
				\label{8.5.2}
			\end{equation}
			
			\item Let $ \bar{\E}_{\tau} $ and $ \bar{\E}_\infty $ be as in \eqref{induction:IH} and \eqref{induction:FK} associated with the set $ \Phi(E\cap (1+c_G)Q) \subset \R^{n-1}\times\set{0} $. We identify $ \R^{n-1}\times\set{0} \cong \R^{n-1}  $.
			
			\item Let $ \mathcal{V} $ be the vertical extension map defined by $ \mathcal{V}(\bar{F})(\bar{x},x_n):= \bar{F}(\bar{x}) $ for an $ (n-1) $-variable function $ \bar{F} $, $ (\bar{x},x_n) \in \R^{n-1}\times\R $.
		\end{itemize}
		
		We begin with conclusion (A). 
		
		We define
		\begin{equation}
			\E_{\tau,Q}(f,M):= \Delta\T_{\tau,Q}(f,M)+(1-\psi)\cdot\widetilde{\E}_{\tau,Q}(f,M),	\label{EQ-def} \end{equation}
		where
		\begin{equation*}
			\widetilde{\E}_{\tau,Q}(f,M):= \overbrace{\brac{
					\mathcal{V}\circ
					\underbrace{\left[
						\brac{
							\Delta\bar{\E}_\infty+(1-\Delta)\bar{\E}_{\tau}(\cdot,C_0M)
						}
						\underbrace{
							\brac{
								\brac{ f - \Delta\T_{\tau,Q}(f,M) }\bigg|_{E \cap (1+c_G)Q} \circ \Phi^{-1}}
						}_{\text{local flattening}}
						\right]}_{\text{$ (n-1) $-dimensional extension}}
			}}^{\text{vertical extension}}
			\circ\Phi.
		\end{equation*}
		In the formula above, $ \Delta = \Delta(f,M,Q) $ and $ C_0 $ is some large controlled constant depending only on $ n $. We also identify $ \R^{n-1}\times\set{0}\cong \R^{n-1} $.
		
		First we note that the map $ \widetilde{\E}_{\tau,Q}(f,M) $ is well defined. Indeed, when $ \Delta = 1 $, the operator in effect is $ \bar{\E}_\infty $, which can be applied to any $ \bar{f} : \R^{n-1}\times\set{0}\supset \Phi(E \cap (1+c_G)Q) \to \R $; when $ \Delta = 0 $, the operator in effect is $ \bar{\E}_{\tau}(\cdot, C_0M) $, and the argument is $ f|_{E\cap (1+c_G)Q}\circ \Phi^{-1} $, which has domain $ \R^{n-1}\times\set{0} $ and range $ \itau $.
		
		We proceed to verify (A1)--(A4) in the following four claims.
		
		\paragraph{Verification of (A1)}

		Suppose $ \T_{\tau,Q}(f,M) \equiv -\tau $. By \eqref{8.5.2}, $ \Delta = 0 $. Formula \eqref{EQ-def} simplifies to
		\begin{equation}
			\E_{\tau,Q}(f,M) = (1-\psi) 
			\cdot 
			\brac{ 
				\mathcal{V} \circ 
				\bar{\E}_{\tau} \brac{
					f|_{E\cap (1+c_G)Q}\circ \Phi^{-1}	
				} 
			}
			\circ \Phi.
			\label{8.5.3}
		\end{equation}
		By the induction hypothesis \eqref{induction:IH}, $ -\tau \leq \bar{\E}_{\tau} \brac{f|_{E\cap (1+c_G)Q}\circ \Phi^{-1}} \leq \tau $. On the other hand, left composition with $ \mathcal{V} $ and right composition with $ \Phi $ do not alter the range, and $ 0 \leq 1-\psi \leq 1 $. Therefore, we have $ -\tau \leq \etq(f,M) \leq \tau  $.
		
		A similar argument shows that when $ \T_{\tau,Q}(f,M) \equiv \tau $, we have $ -\tau \leq \etq(f,M) \leq \tau $.
		
		Now we analyze the more delicate case when $ \T_{\tau,Q}(f,M) $ is not the constant polynomial $ \pm \tau $. In this case, formula \eqref{EQ-def} becomes
		\begin{equation}
			\etq(f,M) = \T_{\tau,Q}(f,M) + (1-\psi)
			\cdot
			\brac{
				\mathcal{V}\circ
				\bar{\E}_\infty
				\brac{
					(f-\T_{\tau,Q}(f,M))\big|_{E\cap (1+c_G)Q} \circ \Phi^{-1}
				}
			} \circ \Phi
			\label{8.5.4}
		\end{equation}
		
		By Lemma \ref{lem.transition-jet}, we have
		\begin{equation}\label{8.5.5}
			\T_{\tau,Q}(f,M)\in\Gk(\xqs,k^\sharp_{LIP},f,CM).
		\end{equation}

		By the assumption $ \norm{f}_\ctet \leq M $, we know there exists $ F \in \ctrt $ with $ F = f $ on $ E $, $ \norm{F}_\ctrn \leq CM $, and
		\begin{equation}\label{8.5.6}
			\jet_{\xqs} F \in \G(\xqs,E,f,CM) \subset \Gk(\xqs,k^\sharp_{LIP},f,CM).
		\end{equation}
		Thanks to Lemma \ref{lem.nearby-jet}, Taylor's theorem, \eqref{8.5.5}, and \eqref{8.5.6}, we see that
		\begin{equation}\label{8.5.7}
			\abs{\da(F - \T_{\tau,Q}(f,M))(x)} \leq CM\dq^\tma \for \abs{\alpha}\leq 2,\, x \in (1+c_G)Q.
		\end{equation}
		Using Lemma \ref{lem.distortion}, \eqref{8.5.phi}, and \eqref{8.5.7}, we have
		\begin{equation}\label{8.5.8}
			\abs{\da\brac{
					(F-\T_{\tau,Q}(f,M))\circ\Phi^{-1}
				}(x)} \leq CM\dq^\tma \for \abs{\alpha}\leq 2\text{ and } x \in \Phi((1+c_G)Q).
		\end{equation}
		Restricting $ (F-\T_{\tau,Q}(f,M))\circ\Phi^{-1} $ to $ \R^{n-1}\times\set{0} \cong \R^{n-1} $, we see that
		\begin{equation}\label{8.5.9}
			\norm{(f-\T_{\tau,Q}(f,M))\circ\Phi^{-1}}_{\ct(\Phi(E \cap (1+c_G)Q))} \leq C_*M.
		\end{equation}
		Here, $ C_* $ is a constant depending only on $ n $, and the trace norm is taken in $ \R^{n-1} $.
		
		Note that the vertical extension map $ \mathcal{V} $ does not increase the $ C^2 $ norm. Therefore, by taking $ C_0 \geq C_* $ in \eqref{EQ-def} and \eqref{8.5.9}, the induction hypothesis \eqref{induction:IH} implies
		\begin{equation}\label{8.5.10}
			\norm{
				G
			}_{\ct(\Phi((1+c_G)Q))
			} \leq CM,
		\end{equation}
		where
		\begin{equation*}
			G:= \mathcal{V}\circ
			\bar{\E}_\infty\brac{ (f-\T_{\tau,Q}(f,M))\circ\Phi^{-1} ,C_0M}.
		\end{equation*}
		In fact, by using \eqref{8.5.8}, \eqref{8.5.9}, and a standard rescaling, we have a stronger estimate
		\begin{equation}\label{8.5.11}
			\abs{\da G(x)} \leq CM\dq^\tma \for \abs{\alpha}\leq 2 \text{ and } x \in \Phi((1+c_G)Q).
		\end{equation}
		Lemma \ref{lem.distortion}, \eqref{8.5.phi}, and \eqref{8.5.11} yields
		\begin{equation}\label{8.5.12}
			\abs{\da(G\circ\Phi)(x)} \leq CM\dq^\tma \for \abs{\alpha}\leq 2  \text{ and } x \in (1+c_G)Q .
		\end{equation}

		Now thanks to \eqref{8.5.12} and the fact that $ 0 \leq \psi \leq 1 $, if $ \tau - \abs{\T_{\tau,Q}(f,M)} \geq AM\dq^2 $ on $ (1+c_G)Q $ for some sufficiently large parameter $ A $, we can conclude that $ -\tau \leq \etq(f,M) \leq \tau $ on $ (1+c_G)Q $. We proceed to examine the value of $ \T_{\tau,Q}(f,M) $ on $ (1+c_G)Q $.
		
		Since we assume that $ \T_{\tau,Q}(f,M) $ is not the constant polynomial $ \pm\tau $, $ \T_{\tau,Q}(f,M) $ must be defined according to \eqref{0zero} (recall \eqref{TQ-def}). In particular, the both the assumptions in ($ -\tau $) and ($ \tau $) fail. Lemma \ref{lem.big-small} then implies that 
		\begin{equation}
			\tau - \abs{f(x)} \geq A_0M\dq^2 \for x \in E \cap (1+c_G)Q.
			\label{8.5.13}
		\end{equation}
		Here, $ A_0 \geq c(n) \cdot (A_T^{1/2}-1) $ with $ A_T $ as in the definition \eqref{TQ-def} of $ \T_{\tau,Q} $.
		
		Recall from Lemma \ref{lem.transition-jet} that $ \T_{\tau,Q}(f,M) \in \Gk(\xqs,k^\sharp_{LIP},f,CM) $. We claim that 
		\begin{equation}
			\tau - \abs{\T_{\tau,Q}(f,M)(\xqs)} \geq c(n)\cdot (\sqrt{A_0}-1)M\dq^2
			\label{8.5.14}
		\end{equation}
		Suppose toward a contradiction, that $ 	\tau - \abs{\T_{\tau,Q}(f,M)(\xqs)} < A_{\mathrm{bad}}M\dq^2 $ for some $ A_{\mathrm{bad}}  $ to be determined. For any $ x \in E \cap (1+c_G)Q $, there exists $ F \in \ctrt $ with $ \norm{F}_{\ctrn} \leq CM $, $ F(x) = f(x) $, and $ \jet_\xqs F \equiv \T_{\tau,Q}(f,M) \in \K_\tau(\xqs,C'M) $, with $ \K_\tau $ as in Definition \ref{def.Ktau}. Thus, \eqref{Ktau-1} and Taylor's theorem imply
		\begin{equation}\label{8.5.15}
			\begin{split}
				\abs{\grad F(x)} &\leq \abs{\grad F(\xqs)} + C\norm{F}_\ctrn\dq \leq C'(\sqrt{A_{\mathrm{bad}} }+1)M\dq^2
				\for x \in (1+c_G)Q, \text{ and }\\
				\tau - \abs{F(x)} &\leq \brac{\tau - \abs{F(\xqs)}} + C\dq\cdot\sup_{y\in(1+c_G)Q}\abs{\grad F} \leq C_{\mathrm{bad}}(\sqrt{A_{\mathrm{bad}}}+1)^2M\dq^2 \for x \in (1+c_G)Q.
			\end{split}
		\end{equation}
		If $ A_{\mathrm{bad}} < C_{\mathrm{bad}}({\sqrt{A_0}-1}) $, with $ A_0 $ as in \eqref{8.5.13} and $ C_{\mathrm{bad}} $ as in \eqref{8.5.15}, we see that \eqref{8.5.15} contradicts \eqref{8.5.13}. Therefore, \eqref{8.5.14} holds.
		
		Thanks to Lemma \ref{lem.dist} and \eqref{8.5.14}, we have
		\begin{equation}\label{8.5.16}
			\dist{\xqs}{\set{\T_{\tau,Q}(f,M) = 0}} \geq c(\sqrt{A_0}-1)\dq.
		\end{equation}
		For sufficiently large $ A_0 $, i.e., sufficiently large $ A_T $ chosen in \eqref{param-2} and \eqref{param-3}, we have
		\begin{equation}\label{8.5.17}
			\tau - \abs{\T_{\tau,Q}(f,M)(x)} \geq CM(\sqrt{A_0}-1)\dq^2
			\for x \in (1+c_G)Q.
		\end{equation}
		Combining \eqref{8.5.12} and \eqref{8.5.17}, we see that $ -\tau \leq \etq(f,M) \leq \tau $ on $ (1+c_G)Q $. (A1) is established.

		\paragraph{Verification of (A2)} Since the support of $ \psi $ is disjoint from $ E $ by \eqref{8.5.1}, and $ \widetilde{\E}_{\tau,Q}(f,M) $ is an extension of $ (f-\T_{\tau,Q}(f,M))\big|_{E\cap (1+c_G)Q} $, conclusion (A2) follows.

		\paragraph{Verification of (A3)} We first deal with the easy case, when $ \T_{\tau,Q}(f,M) $ is not the constant polynomial $ \pm\tau $. Then formula \eqref{EQ-def} becomes \eqref{8.5.4}. By Lemma \ref{lem.transition-jet}, $ \T_{\tau,Q}(f,M) \in \Gk(\xqs,k^\sharp_{LIP},f,CM) $. Thus,
		\begin{equation}\label{8.5.18}
			\norm{\T_{\tau,Q}(f,M)}_{\ct((1+c_G)Q)} \leq CM.
		\end{equation}
		Recall from \eqref{8.5.1} that $ \psi $ satisfies the estimate $ \abs{\da \psi} \leq C\dq^{-\abs{\alpha}} $ for $ \abs{\alpha}\leq 2 $.
		Using Lemma \ref{lem.distortion}, \eqref{8.5.12}, and \eqref{8.5.18} to estimate \eqref{8.5.4}, we can conclude that $ \norm{\etq(f,M)}_{\ct((1+c_G)Q)} \leq CM $.
		
		We now move on to the case where $ \T_{\tau,Q}(f,M) \equiv \pm\tau $. We analyze the case $ \T_{\tau,Q}(f,M)\equiv -\tau $. The case $ \T_{\tau,Q}(f,M)\equiv \tau $ is similar. 
		
		In the present setting, formula \eqref{EQ-def} is simplified to \eqref{8.5.3}, $ \T_{\tau,Q}(f,M) $ in \eqref{TQ-def} is defined using \eqref{-tau}, and by Lemma \ref{lem.perturb} we have $ -\tau \equiv \T_{\tau,Q}(f,M)\in\Gk(\xqs,k^\sharp_{LIP},f,CM) $. Thus,
		\begin{equation}\label{8.5.19}
			f(x)+\tau \leq CM\dq^2\for x \in E\cap (1+c_G)Q.
		\end{equation}
		
		Since $\norm{f}_\ctet \leq M $, there exists $ F \in \ctrt $ with $ F|_E = f $ and $ \norm{F}_\ctrn \leq CM $. In particular, $ \jet_x F \in \K_\tau(x,CM) $ for each $ x \in E\cap (1+c_G)Q $, with $ \K_\tau $ as in Definition \ref{def.Ktau}. Using Taylor's theorem and property \eqref{Ktau-def-2} of $ \K_\tau $ we see that
		\begin{equation}\label{8.5.20}
			\abs{\da F(x)} \leq CM\dq^\tma \for x \in (1+c_G)Q.
		\end{equation}
		Using Lemma \ref{lem.distortion}, \eqref{8.5.phi}, and \eqref{8.5.20}, we have
		\begin{equation}\label{8.5.21}
			\abs{\da (F\circ\Phi^{-1})(x)} \leq CM\dq^\tma \for \abs{\alpha}\leq 2 \text{ and } x \in \Phi((1+c_G)Q).
		\end{equation}
		Restricting $ F\circ\Phi^{-1} $ to $ \R^{n-1}\times\set{0}\cong \R^{n-1} $, we see that
		\begin{equation}\label{8.5.22}
			\norm{f\circ\Phi^{-1}}_{\ct(\Phi(E\cap (1+c_G)Q))} \leq C_*M.
		\end{equation}
		Here, $ C_* $ is a constant depending only on $ n $, and the trace norm is taken in $ \R^{n-1}\times\set{0} \cong \R^{n-1} $.
		
		The vertical extension map does not increase the $ \ct $ norm. By taking $ C_0 \geq C_* $ in \eqref{EQ-def} and \eqref{8.5.22}, the induction hypothesis \eqref{induction:IH} implies
		\begin{equation}\label{8.5.23}
			\norm{H}_{\ct((1+c_G)Q)} \leq CM,
		\end{equation}
		where
		\begin{equation*}
			H:= \mathcal{V}\circ \bar{\E}_\tau(f\circ\Phi^{-1},C_0M).
		\end{equation*}
		In fact, by using \eqref{8.5.21} and \eqref{8.5.22}, together with a standard rescaling, we arrive at the stronger estimate
		\begin{equation}\label{8.5.24}
			\abs{\da H(x)}\leq CM\dq^\tma \for \abs{\alpha}\leq 2\text{ and } x \in \Phi((1+c_G)Q).
		\end{equation}
		Lemma \ref{lem.distortion}, \eqref{8.5.phi}, and \eqref{8.5.24} then imply
		\begin{equation}\label{8.5.25}
			\abs{\da(G\circ\Phi)(x)}\leq CM\dq^\tma \for \abs{\alpha}\leq 2 \text{ and } x \in (1+c_G)Q.
		\end{equation}
		Recall that the cutoff function $ \phi $ satisfies $ \abs{\da\psi}\leq C\dq^{-\abs{\alpha}} $ for $ \abs{\alpha}\leq 2 $. Combining this with \eqref{8.5.25}, we can conclude that $ \norm{\etq(f,M)}_{\ct((1+c_G)Q)} \leq CM $.
		
		We have established (A3).

		\paragraph{Verification of (A4)} Since $ \psi \equiv 1 $ near $ \xqs $ by \eqref{8.5.1}, we have, by Lemma \ref{lem.transition-jet},
		\begin{equation*}
			\jet_{\xqs}\circ\E_{\tau,Q}(f,M) \equiv \T_{\tau,Q}(f,M) \in \Gk(\xqs,k^\sharp_{LIP},f,CM).
		\end{equation*}
		This prove conclusion (A4).
		
		Therefore, Lemma \ref{lem.lip-operator}(A) holds.
		
		\paragraph{Verification of (B)}
		Now we turn to Lemma \ref{lem.lip-operator}(B).
		
		Fix $ x \in (1+c_G)Q $. We begin by defining the set $ S_Q(x) $ to be
		\begin{equation}\label{8.5.26}
			S_Q(x) := \brac{\ssq \cap E } \cup \bar{S}\brac{Proj_{u_Q^\perp}(x-\erep(Q))},
		\end{equation} 
		with $ \ssq $ as in \eqref{ssq-def}, $ \bar{S}(\cdot) $ as in \eqref{induction:IH} (D), and $ u_Q $ as in Lemma \ref{lem.uQ}. Thanks to \eqref{ssq-bd} and \eqref{induction:IH} (D), we have $ \#S_Q(x) \leq D(n) $.
		
		Let $ M \geq 0 $. Let $ f,g \in \ctet $ with $ f = g $ on $ S_Q(x) $.
		
		Since $ f = g $ on $ \ssq \cap E $, we see from the definition of the map $ \T_{\tau,Q}(\cdot,M) $ that 
		\begin{equation}\label{8.5.27}
			\T_{\tau,Q}(f,M) \equiv \T_{\tau,Q}(g,M) \for M \geq 0.
		\end{equation}
		As a consequence, we have
		\begin{equation}\label{8.5.28}
			\Delta_f:= \Delta(f,M,Q) = \Delta(g,M,Q) =: \Delta_g.
		\end{equation}
		The assumption that $ f = g $ on $ \bar{S}\brac{Proj_{u_Q^\perp}(x-\erep(Q))} $ along with \eqref{8.5.27} implies
		\begin{equation}\label{8.5.29}
			(f-\Delta_f\T_{\tau,Q}(f,M))\circ \Phi^{-1} = (g-\Delta_g\T_{\tau,Q}(g,M))\circ \Phi^{-1}
			\text{ on } \bar{S}\brac{Proj_{u_Q^\perp}(x-\erep(Q))}.
		\end{equation}
		It is easy to see that \eqref{8.5.0} follows from substituting \eqref{8.5.27}--\eqref{8.5.29} into \eqref{EQ-def}. This proves Lemma \ref{lem.lip-operator}(B).
		
		Lemma \ref{lem.lip-operator} is proved.

	\end{proof}

	Next, we analyze the algorithmic complexity of Lemma \ref{lem.lip-operator}. Recall that $ \P^+ $ denotes the vector space of polynomials with degree no greater than two and $ \jet_x^+ $ denotes the two-jet at $ x $.

	\begin{lemma}\label{lem.lip-alg}
		Let $ Q \in \Lsk $ with $ \Lsk $ as in \eqref{Lsk-def}. Then there exists a collection of maps 
		\begin{equation*}
			\set{\Xi_{\tau,x,Q} : \tau \in \pos,\,x \in (1+c_G)Q}
		\end{equation*}
		where 
		\begin{equation*}
			\Xi_{\tau,x,Q} : \ctet \times \pos \to \P^+
		\end{equation*}
		for each $ x \in (1+c_G)Q $, such that the following hold.
		\begin{enumerate}[(A)]
			\item There exists a controlled constant $ D(n) $ such that for each $ x \in \Rn $, the map $ \Xi_{\tau,x,Q}(\cdot\,,\cdot) : \ctet\times \pos \to \P^+ $ is of depth $ D $. Moreover, the source of $ \Xi_{\tau,x,Q} $ is independent of $ \tau $.

			\item Suppose we are given $ (f,M) \in \ctet \times \pos $ with $ \norm{f}_{\ctet} \leq M $.
			Then there exists a function $ F_Q \in \ct((1+c_G)Q,\tau) $ such that
			\begin{enumerate}[(1)]
				\item $ \jet^+_x F_Q = \Xi_{\tau,x,Q}(f,M) $ for all $ x $ in the interior of $ (1+c_G)Q $,
				\item $ \norm{F_Q}_{\ct((1+c_G)Q)} \leq C M $,
				\item $ F_Q(x) = f(x) $ for $ x \in E \cap (1+c_G)Q $, and
				\item $ \jet_{\xqs}F_Q \in \Gk(\xqs,k_{LIP}^\sharp,f,CM) $, with $ \xqs $ as in Lemma \ref{lem.rep}.
			\end{enumerate}
			Here, $ C $ depends only on $ n $.

			\item There is an algorithm that takes the given data set $E$, performs one-time work, and then responds to queries.
			
			A query consists of a pair $ (\tau,x,Q) $, and the response to the query is the depth-$ D $ map $ \Xi_{\tau,x,Q} $, given in its efficient representation.
			
			The one-time work takes $ C_1N\log Q $ operations and $ C_2N $ storage. The work to answer a query is $ C_3\log N $. Here, $ C_1, C_2, C_3 $ depend only on $ n $. 
		\end{enumerate}
	\end{lemma}
	
	\begin{proof}
		Let $ \E_{\tau,Q} $ be as in Lemma \ref{lem.lip-operator}, defined by the formula \eqref{EQ-def}. For convenience, we repeat the formula here.
		We set
		\begin{equation}
			\E_{\tau,Q}(f,M):= \Delta\T_{\tau,Q}(f,M)+(1-\psi)\cdot\widetilde{\E}_{\tau,Q}(f,M),\label{EQ-repeat}
		\end{equation}
		where 
		\begin{equation*}
			\widetilde{\E}_{\tau,Q}(f,M):= \overbrace{\brac{
					\mathcal{V}\circ
					\underbrace{\left[
						\brac{
							\Delta\bar{\E}_\infty+(1-\Delta)\bar{\E}_{\tau}(\cdot,C_0M)
						}
						\underbrace{
							\brac{
								\brac{ f - \Delta\T_{\tau,Q}(f,M) }\bigg|_{E \cap (1+c_G)Q} \circ \Phi^{-1}}
						}_{\text{local flattening}}
						\right]}_{\text{$ (n-1) $-dimensional extension}}
			}}^{\text{vertical extension}}
			\circ\Phi\,.
		\end{equation*}
		Recall that $ \jet_x^+ $ denotes the two-jet of a function at $ x $. We would like to define
		\begin{equation}
			\Xi_{\tau,x,Q}:= \jet_x^+\circ \E_{\tau,Q} \for x \in (1+c_G)Q.
		\end{equation}
		
		Lemma \ref{lem.lip-alg}(A) follows from Lemma \ref{lem.lip-operator}(B). Lemma \ref{lem.lip-alg}(B) follows from Lemma \ref{lem.lip-operator}(A).
		
		Now we examine Lemma \ref{lem.lip-alg}(C). Suppose we have performed the necessary one-time work using at most $ CN\log N $ operations and $ CN $ storage. Let $ x \in (1+c_G)Q $ be given. 
		\begin{itemize}
			\item We compute the point 
			\begin{equation*}
				\eqindent
				\bar{y}(x) := Proj_{u_Q^\perp}(x-\erep(Q)).
			\end{equation*}
			Here, 
			\begin{itemize}
				\item $ Proj_{u_Q^\perp} $ is the orthogonal projection onto the hyperplane $ u_Q^\perp $ orthogonal to $ u_Q $;
				\item $ u_Q $ is the vector as in Lemma \ref{lem.uQ};
				\item $ \erep(Q) $ is the map in Lemma \ref{lem.FK-CZ}(C).
			\end{itemize}
			Thanks to Lemma \ref{lem.uQ} and Lemma \ref{lem.projection}, Step 1 requires at most $ C\log N $ operations.
			
			\item Let $ \rho $ be the rotation specified by $ e_n \to u_Q $. We can compute $ \jet_x^+\rho $ using at most $ C $ operations.
			
			\item Let $ u_Q $ and $ Proj_{u_Q^\perp} $ be as in Step 1. We set
			\begin{equation*}
				\bar{E}_Q:= Proj_{u_Q^\perp}(E\cap (1+c_G)Q - \erep(Q)) \subset \R^{n-1}.
			\end{equation*}
			Recall from Lemma \ref{lem.projection} that we can compute all of the unsorted lists $ \bar{E}_Q $ associated with each $ Q \in \Lsk $ using at most $ CN\log N $ operations and $ CN $ storage, which is included in the one-time work. 
			
			\item Let $ \ct(\bar{E}_Q,\tau) $ be the $ (n-1) $-dimensional trace class. Let $ \bar{\Xi}_\tau $ and $ \bar{\Psi} $ be the maps as in \eqref{induction:IH} and \eqref{induction:FK} associated with the set $ \bar{E}_Q $. The one-time work to pre-process $ \bar{E}_Q $ involves $ CN_Q\log N_Q $ operations and $ CN_Q $ storage, with $ N_Q := \#(E\cap (1+c_G)Q) $. The time to answer a query is $ C\log N_Q $. Recall that an answer to a query is the map $ \bar{\Xi}_\tau $ or $ \bar{\Psi} $, given its efficient representation (Definition \ref{def.depth}).
		\end{itemize}
		\begin{remark}\label{rem.no-extra-work}
			Here, we make the crucial remark that the one-time work to pre-process {\em all} the $ \bar{E}_Q $ uses at most
			\begin{equation*}
				\sum_{Q\in \Lsk}CN_Q\log N_Q \leq C'N\log N
			\end{equation*}
			operations and
			\begin{equation*}
				\sum_{Q\in \Lsk}CN_Q \leq C''N
			\end{equation*}
			storage,
			thanks to the bounded intersection property of $ \Lz $ in Lemma \ref{lem.CZ0}. 
		\end{remark}
		
		\begin{itemize}

			\item We can compute the map $ \T_{\tau,Q}(f,M) $ in \eqref{TQ-def} from $ \ssq $ using at most $ C $ operations (Remark \ref{rem.transition-jet-0}). Computing $ \ssq $ requires at most $ C\log N $ operations, thanks to Lemma \ref{lem.FK-CZ}(C) and Lemma \ref{lem.rep}.
			
			\item The jets of $ \Phi $ and $ \Phi^{-1} $ can be computed from the jets of $ \phi $, with $ \phi $ as in Lemma \ref{lem.diffeo-jet}, in $ C\log N_Q $ operations, with $ N_Q = \#(E\cap  (1+c_G)Q) $. 
			
			\item For appropriate choice of cutoff function $ \psi $, we can compute the jets of $ \psi $ using at most $ C $ operations. See Section \ref{sect:pou} below.
		\end{itemize}
		
		Summarizing all of the above, we see that Lemma \ref{lem.lip-alg}(C) holds.
		
	\end{proof}

	\section{Theorems \ref{thm.bd-op} and \ref{thm.bd-alg} and Algorithm \ref{alg.interpolant}}
	\label{sect:main-proof}

	\subsection{Partitions of unity}\label{sect:pou}

	Recall that $ \jet_x^+ $ denotes the two-jet of a function twice-differentiable near $ x \in \Rn $.

	We can construct a partition of unity $ \set{\theta_Q : Q \in \Lz} $ that satisfies the following properties:
	\begin{itemize}
		\item $ 0 \leq \theta_Q \leq 1 $ for each $ Q \in \Lz $.
		\item $ \sum_{Q \in \Lz}\theta_Q \equiv 1 $;
		\item $ \supp{\theta_Q}\subset (1+\frac{c_G}{2})Q $ for each $ Q \in \Lz $;
		\item For each $ Q \in \Lz $, $ \abs{\d^\alpha\theta_Q} \leq C\dq^{2-\abs{\alpha}} $ for $ \abs{\alpha} \leq 2 $;
		\item After one-time work using at most $ CN\log N $ operations and $ CN $ storage, we can answer queries as follows: Given $ x \in \Rn $ and $ Q \in \Lz $, we return $ \jet_x^+\theta_Q $. The time to answer a query is $ C\log N $.
	\end{itemize}

	See Section 28 of \cite{FK09-Data-2} for details.

	\subsection{Proof of Theorem \ref{thm.bd-op}}
	\label{sect:bd-op}

	\newcommand{\eqs}{\mathcal{E}_{\tau,Q}^\sharp}
	
	\begin{proof}
		Let $ \Lz,\Ls,\Lsk, \Le $ be as in Definition \ref{def.CZ}, \eqref{Lsk-def}--\eqref{Le-def}. We have
		\begin{equation*}
			\Lsk \subset \Ls \subset \Lz
			\text{ and }
			\Le = \set{Q \in \Lz \setminus \Ls : \dq \leq A_2^{-1}}.
		\end{equation*}
		
		For $ Q \in \Lz $, we define a map $ \eqs : \ctet \times\pos \to \ct((1+c_G)Q) $ via the following rules.
		
		\begin{itemize}
			\item Suppose $ Q \in \Lsk $. We set $ \eqs:= \E_{\tau,Q} $, with $ \E_{\tau,Q} $ as in Lemma \ref{lem.lip-operator}.
			
			\item Suppose $ Q \in \Ls\setminus\Lsk $. We set $ \eqs:= \T_*^{\xqs}\circ \T_{\tau,Q} $, where $ \T_*^{\xqs} $ is as in Lemma \ref{lem.Ktau}(B) with $ x_0 = \xqs $ and $ \T_{\tau,Q} $ is as in \eqref{TQ-def}.
			
			\item Suppose $ Q \in \Le $. We set $ \eqs:= \T_*^{x_{\mu(Q)}^\sharp}\circ \T_{\tau,\mu(Q)} $, where $ \mu$ is as in Lemma \ref{lem.mu}, $ \T_*^{x_{\mu(Q)}^\sharp} $ is as in Lemma \ref{lem.Ktau}(B) with $ x_0 = x_{\mu(Q)}^\sharp $, and $ \T_{\tau,{\mu(Q)}} $ is as in \eqref{TQ-def}.
			
			\item Suppose $ Q \in \Lz\setminus (\Ls\cup\Le) $. We set $ \eqs:= 0 $. 
		\end{itemize}
		
		Let $ \set{\theta_Q: Q \in \Lz} $ be a partition of unity subordinate to $ \Lz $ as in Section \ref{sect:pou} above. We define $ \E_\tau $ by the formula
		\begin{equation}\label{9.2.1}
			\E_\tau(f,M)(x):= \sum_{Q\in\Lz}\theta_Q(x)\cdot \eqs(f,M)(x).
		\end{equation}
		
		Since $ -\tau \leq \eqs(f,M) \leq \tau $ for each $ Q $, we see that $ -\tau \leq \E_\tau(f,M) \leq \tau $. 
		
		Since $ \eqs(f,M) = f $ on $ E \cap (1+c_G)Q $ for each $ Q \in \Lz $, we see that $ \E_\tau(f,M) = f $ on $ E $.
		
		\newcommand{\touch}{\leftrightarrow}
		
		To estimate the $ \ct $ norm of $ \E_\tau(f,M) $, we need the following lemma.
		
		\begin{lemma}\label{lem.patch-estimate}
			Let $ \eqs $ be defined as above for each $ Q \in \Lz $. Let $ x \in Q \in \Lz $ and $ Q' \in \Lz $ such that $ Q'\touch Q $. Let $ (f,M) \in \ctet\times\pos $ with $ \norm{f}_\ctet \leq M $. We have
			\begin{equation}\label{pe1}
				\abs{\da
					\brac{\eqs(f,M) - \E_{\tau,Q'}^\sharp(f,M)}(x)
				} \leq CM\dq^\tma
				\for \abs{\alpha}\leq 2.
			\end{equation}
		\end{lemma}
		
		We proceed with the proof of Theorem \ref{thm.bd-op} assuming the validity of Lemma \ref{lem.patch-estimate}, postponing the latter's rather tedious proof till the end of the section. 
		
		Now we estimate the $ \ct $ norm of $ \E_\tau(f,M) $. Fix $ x \in \Rn $. Let $ Q(x) \in \Lz $ denote the cube such that $ Q \ni x $. We can write
		\begin{equation}\label{9.2.2}
			\begin{split}	\da \E_\tau(f,M)(x) =
				&\sum_{Q'\touch Q(x)}\theta_{Q'}(x)\cdot \da\E_{Q'}^\sharp (f,M)(x)\\
				&\quad+ \sum_{Q'\touch Q(x), 0 < \beta\leq \alpha}\d^\beta\theta_{Q'}(x)\cdot\d^{\alpha-\beta}\brac{
					\eqs(f,M)- \E_{Q'}^\sharp(f,M)
				}(x).
			\end{split}
		\end{equation}
		
		Now, using the bounded intersection property in Lemma \ref{lem.CZ0} and Lemma \ref{lem.lip-operator} to estimate the first sum in \eq{9.2.2}, and Lemma \ref{lem.patch-estimate} to estimate the second sum, we can conclude that 
		\begin{equation*}
			\norm{\E_\tau(f,M)}_\ctrn \leq CM.
		\end{equation*}
		This proves Theorem \ref{thm.bd-op}(A).
		
		Now we turn to Theorem \ref{thm.bd-op}(B). Fix $ x \in \Rn $. Let $ Q(x) \in \Lz $ be such that $ Q(x)\ni x $. We define
		\begin{equation*}
			S(x) := \brac{\bigcup_{Q'\touch Q(x), Q' \in \Lsk}S_{Q'}(x)} \cup \brac{ \bigcup_{Q'\touch Q(x), Q' \in \Le} S^\sharp(Q') }.
		\end{equation*}
		Here, $ S_{Q'}(x) $ is as in Lemma \ref{lem.lip-operator}(B) and $ S^\sharp(Q') $ is as in \eqref{ssq-def}.
		
		Thanks to Lemma \ref{lem.CZ0}, Lemma \ref{lem.lip-operator}(B), and \eqref{ssq-bd}, we have 
		\begin{equation*}
			\#S(x)\leq D(n).
		\end{equation*}
		
		Given $ f,g \in \ctet $ with $ f = g $ on $ S(x) $, we see from the construction of $ S(x) $ and $ \eqs $ that 
		\begin{equation*}
			\da\eqs(f,M) = \da\eqs(g,M)
			\for \abs{\alpha}\leq 2 \text{ and } M \geq 0.
		\end{equation*}
		From formula \eqref{9.2.1}, we see that
		\begin{equation*}
			\da\E_\tau(f,M) = \da\E_\tau(g,M)
			\for\abs{\alpha}\leq 2 \text{ and } M \geq 0.
		\end{equation*}
		
		Thus, we have established Theorem \ref{thm.bd-op}(B).
		
		The proof of Theorem \ref{thm.bd-op} is complete.
	\end{proof}

	We return to Lemma \ref{lem.patch-estimate}.

	\begin{proof}[Proof of Lemma \ref{lem.patch-estimate}]
		\newcommand{\eqp}{{\E_{\tau,Q'}^\sharp}}
		\newcommand{\xqp}{{x_{Q'}^\sharp}}
		\newcommand{\xmq}{{x_{\mu(Q')}^\sharp}}
		We fix a number $ k^\sharp_{LIP} = k^\sharp_{n,\mathrm{old}} $ with $ k^\sharp_{n,\mathrm{old}} $ as in Theorem \ref{thm.fp-sigma}. Fix $ \alpha $ with $ \abs{\alpha}\leq 2 $ and expand
		\begin{equation}\label{pe2}
			\begin{split}
				\abs{\da
					\brac{\eqs(f,M) - \E_{\tau,Q'}^\sharp(f,M)}(x)
				}
				&\leq \abs{\da\brac{\eqs(f,M) - \jet_{\xqs}\eqs(f,M)}(x)}
				\\
				&\quad\quad + \abs{\da\brac{\eqp(f,M)-\jet_{\xqp}\eqp(f,M)}(x)}\\
				&\quad\quad + \abs{\da\brac{ \jet_{\xqs}\eqs(f,M)-\jet_{\xqp}\eqp(f,M) }(x)}\\
				&=: \eta_1 + \eta_2 + \eta_3.
			\end{split}
		\end{equation}
		
		By Taylor's theorem, 
		\begin{equation}\label{pe3}
			\eta_1, \eta_2 \leq CM\dq^\tma.
		\end{equation}
		
		It remains to show that 
		\begin{equation}\label{eta3}
			\eta_3 \leq CM\dq^\tma.
		\end{equation}
		
		Recall from Lemma \ref{lem.CZ0} that $ C^{-1}\dq\leq \delta_{Q'}\leq C\dq $. We analyze the following cases.
		
		\begin{enumerate}[\textbf{Case} 1.]
			\item Suppose either $ Q $ or $ Q' $ belongs to $ \Lz\setminus\brac{\Ls\cup\Le} $. Then \eqref{eta3} follows from Lemmas \ref{lem.CZ0}, \ref{lem.transition-jet}, \ref{lem.lip-operator}, and Taylor's Theorem. 
			
			For the rest of the analysis, we assume that neither $ Q $ nor $ Q' $ belongs to $ \Lz\setminus\brac{\Ls\cup\Le} $.
			
			\item Suppose both $ Q, Q' \in \Lsk $. Recall from Lemmas \ref{lem.transition-jet} and \ref{lem.lip-operator} that
			\begin{equation*}
				\jet_{\xqs}\eqs(f,M) \in \Gk(\xqs,k^\sharp_{LIP},f,CM)
				\text{ and }
				\jet_\xqp\eqp(f,M)\in\Gk(\xqp,k^\sharp_{LIP},f,CM).
			\end{equation*}
			Then \eqref{eta3} follows from Lemma \ref{lem.nearby-jet}.

			\item Suppose $ Q \in \Lsk $ and $ Q' \in \Ls\setminus\Lsk $. This means that
			\begin{itemize}
				\item $ \eqs = \E_{\tau,Q} $ as in Lemma \ref{lem.lip-operator}, and
				\item $ \eqp = \T_*^{\xqp}\circ\T_{\tau,{Q'}} $, with $ \T_*^{\xqp} $ as in Lemma \ref{lem.Ktau} and $ \T_{\tau,{Q'}} $ as in \eqref{TQ-def}. 
			\end{itemize}
			We expand 
			\begin{equation}\label{pe.3.1}
				\eqindent
				\begin{split}
					\eta_3 &\leq \abs{\da\brac{\jet_{\xqs}\eqs(f,M)-\T_{\tau,{Q'}}(f,M)}(x)} \\
					&\quad\quad+ \abs{\da\brac{ \T_{\tau,{Q'}}(f,M) - \T_*^\xqp\circ \T_{\tau,{Q'}}(f,M)  }(x)} =: \eta_{3,1} +\eta_{3,2}.
				\end{split}
			\end{equation}
			Taylor's theorem implies 
			\begin{equation}\label{pe.3.2}
				\eqindent
				\eta_{3,2} \leq CM\dq^\tma.
			\end{equation}
			On the other hand, we have
			\begin{equation}\label{pe.3.3}
				\eqindent
				\begin{split}
					\eta_{3,1} &= \abs{\da\brac{ \T_{\tau,Q}(f,M) - \T_{\tau,{Q'}}(f,M)  }(x)} \quad \text{(by Lemma \ref{lem.lip-operator}(A4))}\\
					&\leq CM\dq^\tma. \quad \text{(by Lemma \ref{lem.nearby-jet} and Lemma \ref{lem.transition-jet})}
				\end{split}
			\end{equation}
			We see that \eqref{eta3} follows from \eqref{pe.3.1}--\eqref{pe.3.3}.
			
			\item Suppose $ Q \in \Ls\setminus\Lsk $ and $ Q' \in \Lsk $. The analysis is similar to Case 3.

			\item Suppose $ Q \in \Lsk $ and $ Q' \in \Le $. This means that
			\begin{itemize}
				\item $ \eqs = \E_{\tau,Q} $ as in Lemma \ref{lem.lip-operator}, and
				\item $ \eqp = \T_*^{\xmq}\circ\T_{\tau,{\mu(Q')}} $, with $ \T_*^{\xmq} $ as in Lemma \ref{lem.Ktau} and $ \T_{\tau,{\mu(Q')}} $ as in \eqref{TQ-def}. 
			\end{itemize}
			Thanks to Lemma \ref{lem.mu}(B), we have
			\begin{equation}\label{pe5}
				\eqindent
				\abs{\xqs - \xmq},\, \abs{x - \xqs},\,\abs{x - \xmq} \leq C\dq.
			\end{equation}
			We expand
			\begin{equation}\label{pe6}
				\eqindent
				\begin{split}
					\eta_3 &\leq \abs{\da\brac{\jet_{\xqs}\eqs(f,M)-\T_{\tau,{\mu(Q')}}(f,M)}(x)} \\
					&\quad\quad+ \abs{\da\brac{ \T_{\tau,{\mu(Q')}}(f,M) - \T_*^\xqp\circ \T_{\tau,{\mu(Q')}}(f,M)  }(x)} =: \eta_{3,1} +\eta_{3,2}.
				\end{split}
			\end{equation}
			Taylor's theorem and \eqref{pe5} implies 
			\begin{equation}\label{pe7}
				\eqindent
				\eta_{3,2} \leq CM\dq^\tma.
			\end{equation}
			On the other hand, we have
			\begin{equation}\label{pe8}
				\eqindent
				\begin{split}
					\eta_{3,1} &= \abs{\da\brac{ \T_{\tau,Q}(f,M) - \T_{\tau,{\mu(Q')}}(f,M)  }(x)} \quad \text{(by Lemma \ref{lem.lip-operator}(A4))}\\
					&\leq CM\dq^\tma. \quad \text{(by Lemma \ref{lem.nearby-jet} and Lemma \ref{lem.transition-jet})}
				\end{split}
			\end{equation}
			We see that \eqref{eta3} follows from \eqref{pe6}--\eqref{pe8}.
			
			\item Suppose $ Q \in \Le $ and $ Q' \in \Lsk $. The analysis is similar to Case 5.

			\item Suppose $ Q \in \Ls\setminus \Lsk $ and $ Q' \in \Le $. This means that 
			\begin{itemize}
				\item $ \eqs = \T_*^{\xqs} \circ \T_{\tau,{Q}} $, with $ \T_*^\xqs $ as in Lemma \ref{lem.Ktau} and $ \T_{\tau,{Q}} $ as in \eqref{TQ-def}; 
				\item $ \eqp = \T_*^{\xmq}\circ\T_{\tau,{\mu(Q')}} $, with $ \T_*^{\xmq} $ as in Lemma \ref{lem.Ktau} and $ \T_{\tau,{\mu(Q')}} $ as in \eqref{TQ-def}. 
			\end{itemize}
			Notice that $ \jet_{\xqs}\eqs = \T_{\tau,{Q}} $, so that Taylor's theorem, Lemma \ref{lem.nearby-jet}, and \eqref{pe5} imply
			\begin{equation*}
				\eqindent
				\begin{split}
					\eta_3 = \abs{\da\brac{ \T_{\tau,{Q}}(f,M) - \T_*^\xqp\circ \T_{\tau,{\mu(Q')}}(f,M)  }(x)} \leq CM\dq^\tma.
				\end{split}
			\end{equation*}
			This is precisely \eqref{eta3}.

			\item Suppose $ Q \in \Le  $ and $ Q' \in \Ls\setminus \Lsk$. The analysis is similar to Case 7.

			\renewcommand{\xmq}{{x_{\mu(Q)}^\sharp}}
			\newcommand{\xmqp}{{x_{\mu(Q')}^\sharp}}
			\item Suppose both $ Q, Q' \in \Le $. By construction, 
			\begin{itemize}
				\item $ \eqs = \T_*^{\xmq} \circ \T_{\tau,{\mu(Q)}} $, with $ \T_*^\xmq $ as in Lemma \ref{lem.Ktau} and $ \T_{\tau,{\mu(Q)}} $ as in \eqref{TQ-def}; 
				\item $ \eqp = \T_*^{\xmqp}\circ\T_{\tau,{\mu(Q')}} $, with $ \T_*^{\xmqp} $ as in Lemma \ref{lem.Ktau} and $ \T_{\tau,{\mu(Q')}} $ as in \eqref{TQ-def}. 
			\end{itemize}
			Thanks to Lemma \ref{lem.mu}(B) and the assumption that $ Q' \leftrightarrow Q $, we have
			\begin{equation}\label{pe11}
				\eqindent
				\abs{\xmq - \xmqp},\, \abs{x - \xmq},\,\abs{x - \xmqp} \leq C\dq.
			\end{equation}
			We expand
			\begin{equation}\label{pe12}
				\begin{split}
					\eta_3 &\leq \abs{ \da\brac{ \T_*^\xmq \circ \T_{\tau,{\mu(Q)}}(f,M) - \T_{\tau,{\mu(Q)}}(f,M) }(x) }\\
					&\quad\quad + 
					\abs{\da\brac{
							\T_*^\xmqp\circ\T_{\tau,{\mu(Q')}}(f,M) - \T_{\tau,{\mu(Q')}}(f,M)
						}(x)}\\
					&\quad\quad + 
					\abs{\da\brac{
							\T_{\tau,{\mu(Q)}}(f,M) - \T_{\tau,{\mu(Q')}}(f,M)
						}(x)} =: \eta_{3,1}+\eta_{3,2}+\eta_{3,3}.
				\end{split}
			\end{equation}
			It follows from Taylor's theorem, Lemma \ref{lem.Ktau}, and \eqref{pe11} that
			\begin{equation}\label{pe13}
				\eta_{3,1}, \eta_{3,2} \leq CM\dq^\tma.
			\end{equation}
			On the other hand, Lemma \ref{lem.nearby-jet}, Lemma \ref{lem.lip-operator}(A4), and \eqref{pe11} imply
			\begin{equation}\label{pe14}
				\eta_{3,3} \leq CM\dq^\tma.
			\end{equation}
			Therefore, \eqref{eta3} follows from \eqref{pe12}--\eqref{pe14}.
			
		\end{enumerate}
		
		We have analyzed all the possible cases. Therefore, \eqref{eta3} holds. 
		
		Finally, \eqref{pe1} follows from \eqref{pe3} and \eqref{eta3}. The proof of Lemma \ref{lem.patch-estimate} is complete. 
		
	\end{proof}

	\subsection{Proof of Theorem \ref{thm.bd-alg}}
	\label{sect:bd-alg}
	
	\begin{proof}
		We put ourselves in the setting of the proof of Theorem \ref{thm.bd-op} in Section \ref{sect:bd-op}. In particular, recall formula \eqref{9.2.1} and the assignment rules for $ \eqs $.
		
		For each $ x \in \Rn $, we define
		\begin{equation}\label{Xi-tau}
			\Xi_{\tau,x}(f,M) := \jet_x^+\circ \E_{\tau}(f,M) = \sum_{Q\in\Lambda(x)}\jet_x^+\theta_Q\odot_x^+\jet_x^+\circ \eqs(f,M).
		\end{equation}
		In the formula above, $ \jet_x^+ $ denotes the two-jet at $ x $, $ \odot_x^+ $ is the multiplication on the ring of two-jets $ \mathcal{R}_x^+ $, and $ \Lambda(x) $ is as in Lemma \ref{lem.FK-CZ}(A), i.e., $ \Lambda(x) = \set{Q\in\Lz : (1+c_G)Q\ni x} $.
		
		Theorem \ref{thm.bd-alg}(A) follows from Theorem \ref{thm.bd-op}(B). Theorem \ref{thm.bd-alg}(B) follows from Theorem \ref{thm.bd-op}(A). 
		
		We now turn to Theorem \ref{thm.bd-alg}(C). 
		
		Recall from the proof of Theorem \ref{thm.bd-op} in Section \ref{sect:bd-op} that $ \eqs $ can take the following forms.
		\begin{itemize}
			\item $ \eqs = \E_{\tau,Q} $ as in Lemma \ref{lem.lip-operator} for $ Q \in \Lsk $.
			\item $ \eqs = \T_*^\xqs \circ\T_{\tau,{Q}} $ with $ \T_*^\xqs $ as in Lemma \ref{lem.Ktau} and $ \T_{\tau,{Q}}  $ as in \eqref{TQ-def}, for $ Q \in \Ls\setminus \Lsk $.
			\item $ \eqs = \T_*^{x_{\mu(Q)}^\sharp} \circ \T_{\tau,{\mu(Q)}} $, with $ \mu $ as in Lemma \ref{lem.mu}, for $ Q \in \Le $.
			\item $ \eqs \equiv 0 $ for $ Q \in \Lz\setminus\Le $.
		\end{itemize}
		
		Suppose we have performed the necessary one-time work using at most $ CN\log N $ operations and $ CN $ storage. Note that this includes the work and storage involved in Lemma \ref{lem.FK-palp}, Lemma \ref{lem.FK-CZ}, and Remark \ref{rem.no-extra-work} in the proof of Lemma \ref{lem.lip-alg}.
		
		\begin{itemize}
			\item By Lemma \ref{lem.FK-CZ}(A) and Section \ref{sect:pou}, we can compute $ \Lambda(x) $ and $ \set{\jet_x^+\theta_Q : Q \in \Lambda(x)} $ using at most $ C\log N $ operations.
			
			\item 	By Lemma \ref{lem.lip-alg}, we can compute 
			\begin{equation*}
				\set{\jet_x^+\circ \E_{\tau,Q} : Q \in \Lsk \cap \Lambda(x)}
			\end{equation*}
			using at most $ C\log N $ operations. Recall Remark \ref{rem.no-extra-work} in the proof of Lemma \ref{lem.lip-alg}.
			
			\item By Lemma \ref{lem.rep} and Remark \ref{rem.transition-jet-0}, we can compute
			\begin{equation*}
				\set{\jet_x^+\circ \T_*^\xqs \circ \T_{\tau,{Q}}(f,M) : Q \in \Lambda(x)\cap \brac{ \Ls\setminus\Lsk }}
			\end{equation*}
			using at most $ C\log N $ operations. 
			
			\item 	By Lemma \ref{lem.mu}, Lemma \ref{lem.rep}, and Remark \ref{rem.transition-jet-0}, we can compute
			\begin{equation*}
				\set{\jet_x^+\circ \T_*^{x_{\mu(Q)}^\sharp}\circ \T_{\tau,{\mu(Q)}}(f,M)  : Q \in \Le\cap\Lambda(x)  } 
			\end{equation*}
			using at most $ C\log N $ operations. 
		\end{itemize}

		Thus, given $ (f,M) \in \ctet \times\pos $, we can compute $ \Xi_{\tau,x}(f,M) $ in $ C\log N $ operations. Theorem \ref{thm.bd-alg}(C) follows.
		
		This completes the proof of Theorem \ref{thm.bd-alg}.
	\end{proof}

	\section{Proof of Theorem \ref{thm.sfp} and Algorithm \ref{alg.norm}}
	\label{sect:sfp}

	\subsection{Callahan-Kosaraju decomposition}
	
	We will use the data structure introduced by Callahan and Kosaraju\cite{CK95}.

	\begin{lemma}[Callahan-Kosaraju decomposition]\label{lem.wspd}
		Let $ E \subset \Rn $ with $ \#E = N < \infty $. Let $ \kappa > 0 $. We can partition $ E \times E \setminus diagonal(E) $ into subsets $ E_1'\times E_1'', \cdots, E_L'\times E_L'' $ satisfying the following.
		\begin{enumerate}[(A)]
			\item $ L \leq C(\kappa,n)N $.
			\item For each $ \ell = 1, \cdots, L $, we have
			\begin{equation*}
				\diam{E_\ell'}, \diam{E_\ell''} \leq \kappa \cdot\dist{E_\ell'}{E_\ell''}\,.
			\end{equation*}
			\item Moreover, we may pick $ x_\ell' \in E_\ell' $ and $ x_\ell'' \in E_\ell'' $ for each $ \ell = 1, \cdots, L $, such that the $ x_\ell', x_\ell'' $ for $ \ell = 1, \cdots, L $ can all be computed using at most $ C(\kappa,n)N\log N $ operations and $ C(\kappa,n)N $ storage.
		\end{enumerate}
		Here, $ C(\kappa,n) $ is a constant that depends only on $ \kappa $ and $ n $.
	\end{lemma}
	
	With a slight tweak, the argument in the proof of Lemma 3.1 in \cite{F09-Data-3} yields the following. 
	
	\begin{lemma}\label{lem.wf-rep}
		Let $ \tau > 0 $. Let $ E \subset \Rn $ be a finite set. Let $ \kappa_0 > 0 $ be a constant that is sufficiently small. Let $ E_\ell', E_\ell'' $ be as in Lemma \ref{lem.wspd} with $ \kappa = \kappa_0 $. Suppose $ \vec{P} = (P^x)_{x \in E} \in W(E,\tau) $ satisfies the following.
		\begin{enumerate}[(A)]
			\item $ P^x \in \K_\tau(x,M) $ for each $ x \in E $, with $ \K_\tau $ as in Definition \ref{def.Ktau}.
			\item $ \abs{\d^\alpha (P^{x_\ell'} - P^{x_\ell''})(x_\ell'')} \leq M\abs{x_\ell' - x_\ell''}^{2-\abs{\alpha}}
			\text{ for } \abs{\alpha} \leq 1,\,
			\ell = 1, \cdots, L $.
		\end{enumerate}
		Then $ \norm{\vec{P}}_{W(E,\tau)} \leq CM $. 
	\end{lemma}

	\begin{lemma}[Lemma 3.2 of \cite{F09-Data-3}]\label{lem.arise} 
		Let $ E \subset \R^2 $ be a finite set. Let $ E_\ell' $ and $ E_\ell'' $ be as in Lemma \ref{lem.wspd} with $ \ell = 1, \cdots, L $. Then every $ x \in E $ arises as an $ x_\ell' $ for some $ \ell \in \set{1, \cdots, L} $.
		
	\end{lemma}

	\subsection{Proof of Theorem \ref{thm.sfp}}
	
	\begin{proof}[Proof of Theorem \ref{thm.sfp} Assuming Theorem \ref{thm.bd-alg}]
		
		Let $ E \subset \Rn $ be a finite set. Let $ \set{\Xi_{\tau,x}, x \in \Rn} $ be as in Theorem \ref{thm.bd-alg}. For each $ x \in E $, let $ S(x) $ be the source of $ \Xi_{\tau,x} $ (See Definition \ref{def.depth}). Note that $ S(x) $ is independent of $ \tau $, thanks to Theorem \ref{thm.bd-alg}(A).
		
		Let $ \kappa_0 $ be as in Lemma \ref{lem.wf-rep}. Let $ (x_\ell', x_\ell'') \in E\times E $, $ \ell = 1, \cdots, L $, be as in Lemma \ref{lem.wspd} with $ \kappa = \kappa_0 $. 
		
		We set
		\begin{equation}
			S_\ell := \set{x_\ell', x_\ell''} \cup S(x_\ell') \cup S(x_\ell'')
			\enskip,\enskip
			\ell = 1, \cdots, L.
			\label{Sl-def}
		\end{equation}
		
		Conclusion (A) follows from Theorem \ref{thm.bd-alg}(B,C) and Lemma \ref{lem.wspd}. 
		
		Conclusion (B) follows from Theorem \ref{thm.bd-alg}(B,C).
		
		Conclusion (C) follows from Lemma \ref{lem.wspd}(C). 
		
		Now we verify conclusion (D). We modify the argument in \cite{F09-Data-3}.
		
		Fix $ \tau > 0 $ and $ f : E \to \itau $. Set
		\begin{equation}
			M := \max_{\ell = 1, \cdots, L}\norm{f}_{\ct(S_\ell,\tau)}.
			\label{M-fl}
		\end{equation}

		Thanks to \eqref{M-fl}, we see that $ \norm{f}_{\ct(S_\ell,\tau)} \leq M $ for $ \ell = 1, \cdots, L $. Thus, for each $ \ell = 1, \cdots, L $, there exists $ F_\ell \in \ctrt $ such that 
		\begin{equation}
			\norm{F_\ell}_{\ctrn} \leq 2M \text{ and } F_\ell(x) = f(x)\for x \in S_\ell.
			\label{3.6}
		\end{equation}
		Fix such $ F_\ell $. For $ \ell = 1, \cdots, L $, we define 
		\begin{equation}
			f_\ell :E \to \pos 
			\text{ by }
			f_\ell(x) := \begin{cases}
				f(x) \for x \in S_\ell\\
				F_\ell(x) \for x \in E \setminus S_\ell
			\end{cases}.
			\label{3.7}
		\end{equation}
		From \eqref{3.6} and \eqref{3.7}, we see that
		\begin{equation}
			\norm{f_\ell}_{\ctet} \leq 2M \for \ell = 1, \cdots, L.
			\label{f-l-est}
		\end{equation}

		For each $ \ell = 1, \cdots, L $, we define \begin{equation}
			P_\ell' := \jet_{x_\ell'}\brac{\Xi_{\tau,x_\ell'}(f_\ell,2M)}
			\text{ and }
			P_\ell'' := \jet_{x_\ell''}\brac{\Xi_{\tau,x_\ell''}(f_\ell,2M)}.
			\label{P-ell-def}
		\end{equation}
		
		We will show that the assignment \eqref{P-ell-def} unambiguously defines a Whitney field over $ E $.

		\begin{claim}\label{claim.ell}
			Let $ \ell_1, \ell_2 \in \set{1, \cdots, L} $.
			\begin{enumerate}[(a)]
				\item Suppose $ x_{\ell_1}' = x_{\ell_2}' $. Then $ P_{\ell_1}' = P_{\ell_2}' $.
				\item Suppose $ x_{\ell_1}'' = x_{\ell_2}'' $. Then $ P_{\ell_1}'' = P_{\ell_2}'' $.
				\item Suppose $ x_{\ell_1}' = x_{\ell_2}'' $. Then $ P_{\ell_1}' = P_{\ell_2}'' $.
			\end{enumerate}
			
		\end{claim}
		
		\begin{proof}[Proof of Claim \ref{claim.ell}]
			We prove (a). The proofs for (b) and (c) are similar.
			
			Suppose $ x_{\ell_1}' = x_{\ell_2}' =: x_0 $. Let $ S(x_0) $ be the source of $ \Xi_{\tau,x_0} $. By \eqref{Sl-def}, we see that
			\begin{equation*}
				S(x_0) \subset S_{\ell_1} \cap S_{\ell_2}. 
			\end{equation*}
			Therefore, we have
			\begin{equation*}
				f_{\ell_1}(x) = f_{\ell_2}(x) \for x \in S(x_0).
			\end{equation*}
			Thanks to Theorem \ref{thm.bd-alg}(A) and \eqref{f-l-est}, we see that
			\begin{equation*}
				\Xi_{\tau,x_0}(f_{\ell_1},2M) = 	\Xi_{\tau,x_0}(f_{\ell_2},2M).
			\end{equation*}
			By \eqref{P-ell-def}, we see that $ P_{\ell_1} = P_{\ell_2} $. This proves (a).
		\end{proof}

		By Lemma \ref{lem.arise}, there exists a pair of maps:
		\begin{equation}
			\begin{split}
				\text{A surjection } \pi &: \set{1, \cdots, L} \to E
				\text{ such that }
				\pi(\ell) = x_{\ell}'
				\for \ell = 1, \cdots, L\text{, and }\\
				{\text{An injection }} \rho &: E \to \set{1, \cdots, L}
				\text{ such that }x_{\rho(x)}' = x \for x \in E, \text{ i.e., } \pi \circ \rho = id_E. 
			\end{split}
			\label{lx}
		\end{equation}
		The surjection $ \pi $ is determined by the Callahan-Kosaraju decomposition (Lemma \ref{lem.wspd}), but the choice of $ \rho $ is not necessarily unique. 
		
		Thanks to Claim \ref{claim.ell} and the fact that $ E_\ell' \times E_\ell'' \subset E \times E \setminus diagonal(E) $, assignment \eqref{P-ell-def} produces for each $ x \in E $ a uniquely defined polynomial
		\begin{equation}
			P^x = \jet_{x}\brac{\Xi_{\tau,x}(f_{\rho(x)},2M)},
			\label{3.9}
		\end{equation}
		with $ \Xi_x $ as in Theorem \ref{thm.bd-alg} and $ \rho(x) $ as in \eqref{lx}. Note that, as shown in Claim \ref{claim.ell}, the polynomial $ P^x $ in \eqref{3.9} is independent of the choice of $ \rho $ as a right-inverse of $ \pi $ in \eqref{lx}.
		
		Thanks to Theorem \ref{thm.bd-alg}(B) and \eqref{f-l-est}--\eqref{3.9}, for each $ \ell = 1, \cdots, L $, there exists a function $ \tilde{F}_\ell \in \ctrn $ such that
		\begin{equation}\label{F-ell-1}
			\norm{\tilde{F}_\ell}_{\ctrn} \leq CM
			\text{ and }
			-\tau \leq \tilde{F}_\ell \leq \tau \text{ on }\Rn;
		\end{equation}
		\begin{equation}\label{F-ell-2}
			\tilde{F}_\ell = f_{\ell}(x) = f(x)
			\for x \in S_{\ell} ; \text{ and }
		\end{equation}
		\begin{equation}\label{F-ell-3}
			\jet_{x_{\ell}'}\tilde{F}_\ell = P^{x_{\ell}'} = \jet_{x_{\ell}'}\brac{ \Xi_{\tau,x_{\ell}'} (f_{\ell},2M) }, \text{ and }
			\jet_{x_{\ell}''}\tilde{F}_\ell = P^{x_{\ell}''} = \jet_{x_{\ell}''}\brac{ \Xi_{\tau,x_{\ell}''} (f_{\ell},2M) }.
		\end{equation}

		Thanks to \eqref{F-ell-1} and \eqref{F-ell-2}, we have 
		\begin{equation}
			P^{x_\ell'} \in \G(x_\ell',\{x_\ell'\},f,CM)
			\for \ell = 1, \cdots, L.
			\label{3.15}
		\end{equation}
		Thanks to \eqref{F-ell-1} and \eqref{F-ell-3}, we have 
		\begin{equation}
			\abs{\d^\alpha (P^{x_\ell'} - P^{x_\ell''})(x_\ell'')}
			\leq CM\abs{x_\ell' - x_\ell''}^{2-\abs{\alpha}}
			\for \abs{\alpha} \leq 1, \ell = 1, \cdots, L.
			\label{3.16}
		\end{equation}

		Therefore, by Lemma \ref{lem.wf-rep}, \eqref{3.15}, and \eqref{3.16}, the Whitney field $ \vec{P} = (P^x)_{x \in E} $, with $ P^x $ as in \eqref{3.9}, satisfies 
		\begin{equation*}
			\vec{P} \in W(E,\tau),
			P^x(x) = f(x) \for x \in E,
			\text{ and }
			\norm{\vec{P}}_{W(E,\tau)} \leq CM.
		\end{equation*}
		By Whitney's Extension Theorem \ref{thm.WT-tau}(B), there exists a function $ F \in \ctrt $ such that $ \norm{F}_{\ctrn} \leq CM $ and $ \jet_x F = P^x $ for each $ x \in E $. In particular, $ F(x) = P^x(x) = f(x) $ for each $ x \in E $. Thus, $ \norm{f}_{\ctet} \leq CM $. This proves conclusion (D).
		
		Theorem \ref{thm.sfp} is proved. 
		
	\end{proof}

	\subsection{Explanation of Algorithm \ref{alg.norm}}

	Let $ E \subset \Rn $ be given. We compute $ S_1, \cdots, S_L $ from $ E $ as in Theorem \ref{thm.sfp}. This uses one-time work using at most $ CN\log N $ operations and $ CN $ storage. 
	
	For each $ \ell = 1, \cdots, L $, we compute a number $ M_\ell $ that has the order of magnitude as $ \norm{f}_\ct(S_\ell,\tau) $. This can be reformulated as a collection of convex quadratic programming problems as in Section \ref{sect:quad}, and requires at most $ CN $ operations, since $ \#S_\ell \leq C $ and $ L \leq CN $. Finally, $ \norm{f}_\ctet $ has the same order of magnitude as $ \max\set{M_\ell: \ell = 1, \cdots, L} $.

	\bibliographystyle{plain}
	\bibliography{Whitney-bib}

\begin{thebibliography}{10}

\bibitem{BM07}
Edward Bierstone and Pierre~D. Milman.
\newblock $\mathscr{C}^m$-norms on finite sets and $\mathscr{C}^m$-extension
  criteria.
\newblock {\em Duke Mathematical Journal}, 137(1):1--18, 2007.

\bibitem{BMP03}
Edward Bierstone, Pierre~D. Milman, and Wiesław Pawłucki.
\newblock Differentiable functions defined in closed sets. {A} problem of
  {W}hitney.
\newblock {\em Invent. Math.}, 151(2):329--352, 2003.

\bibitem{BMP06}
Edward Bierstone, Pierre~D. Milman, and Wiesław Pawłucki.
\newblock Higher-order tangents and fefferman's paper on {W}hitney's extension
  problem.
\newblock {\em Ann. of Math. (2)}, 164(1):361--370, 2006.

\bibitem{BV-CO}
Stephan Boyd and Lieven Vandenberghe.
\newblock {\em Convex Optimization}.
\newblock Cambridge University Press, 2004.

\bibitem{BS85}
Yuri Brudnyi and Pavel Shvartsman.
\newblock A linear extension operator for a space of smooth functions defined
  on a closed subset in {$\mathbb{R}^n$}.
\newblock {\em Dokl. Akad. Nauk SSSR}, 280(2):268--272, 1985.

\bibitem{BS94-W}
Yuri Brudnyi and Pavel Shvartsman.
\newblock Generalizations of {W}hitney's extension theorem.
\newblock {\em Internat. Math. Res. Notices}, 3(129), 1994.

\bibitem{BS94-Tr}
Yuri Brudnyi and Pavel Shvartsman.
\newblock The traces of differentiable functions to subsets of
  {$\mathbb{R}^n$}.
\newblock In {\em Linear and complex analysis. Problem book 3}, volume 1574 of
  {\em Lecture Notes in Mathematics}, pages 279--281. Springer-Verlag, Berlin,
  1994.

\bibitem{BS97}
Yuri Brudnyi and Pavel Shvartsman.
\newblock The {W}hitney problem of existence of a linear extension operator.
\newblock {\em J. Geom. Anal.}, 7(4):515--574, 1997.

\bibitem{BS98}
Yuri Brudnyi and Pavel Shvartsman.
\newblock The trace of jet space {$J^k\Lambda^\omega$} to an arbitrary closed
  subset of {$\mathbb{R}^n$}.
\newblock {\em Trans. Amer. Math. Soc.}, 350(4):1519--1553, 1998.

\bibitem{BS01}
Yuri Brudnyi and Pavel Shvartsman.
\newblock {W}hitney's extension problem for multivariate
  {${C}^{1,\omega}$}-functions.
\newblock {\em Trans. Amer. Math. Soc.}, 353(6):2487--2512 (electronic), 2001.

\bibitem{CK95}
Paul~B. Callahan and S.~Rao Kosaraju.
\newblock A decomposition of multidimensional point sets with applications to $
  k $-nearest-neighbors and $ n $-body potential fields.
\newblock {\em J. ACM}, 42:67--90, 1995.

\bibitem{range4}
ES~Chan and BH~Ong.
\newblock Range restricted scattered data interpolation using convex
  combination of cubic b{\'e}zier triangles.
\newblock {\em Journal of Computational and Applied Mathematics},
  136(1-2):135--147, 2001.

\bibitem{F05-J}
Charles Fefferman.
\newblock A generalized sharp {W}hitney theorem for jets.
\newblock {\em Rev. Mat. Iberoam.}, 21(2):577--688, 2005.

\bibitem{F05-L}
Charles Fefferman.
\newblock Interpolation and extrapolation of smooth functions by linear
  operators.
\newblock {\em Rev. Mat. Iberoam.}, 21(1):313--348, 2005.

\bibitem{F05-Sh}
Charles Fefferman.
\newblock A sharp form of {W}hitney's extension theorem.
\newblock {\em Ann. of Math. (2)}, 161(1):509--577, 2005.

\bibitem{F06}
Charles Fefferman.
\newblock {W}hitney's extension problem for {$ {C^m} $}.
\newblock {\em Ann. of Math. (2)}, 164(1):313--359, 2006.

\bibitem{F07-L}
Charles Fefferman.
\newblock {${C}^m$} extension by linear operators.
\newblock {\em Ann. of Math. (2)}, 166(2):779--835, 2007.

\bibitem{F09-Data-3}
Charles Fefferman.
\newblock Fitting a {$C^m$}-smooth function to data {III}.
\newblock {\em Ann. of Math. (2)}, 170(1):427--441, 2009.

\bibitem{F09-Int}
Charles Fefferman.
\newblock {W}hitney's extension problems and interpolation of data.
\newblock {\em Bull. Amer. Math. Soc. (N.S.)}, 46(2):207--220, 2009.

\bibitem{FIL16}
Charles Fefferman, Arie Israel, and Garving~K. Luli.
\newblock Finiteness principles for smooth selections.
\newblock {\em Geom. Funct. Anal.}, 26(2):422--477, 2016.

\bibitem{FK09-Data-1}
Charles Fefferman and Bo'az Klartag.
\newblock Fitting a {$C^m$}-smooth function to data. {I}.
\newblock {\em Ann. of Math. (2)}, 169(1):315--346, 2009.

\bibitem{FK09-Data-2}
Charles Fefferman and Bo'az Klartag.
\newblock Fitting a {$C^m$}-smooth function to data. {II}.
\newblock {\em Rev. Mat. Iberoam.}, 25(1):49--273, 2009.

\bibitem{FL04}
Charles Fefferman and Garving~K. Luli.
\newblock The {B}renner-{H}ochster-{K}ollár and {W}hitney problems for
  vector-valued functions and jets.
\newblock {\em Rev. Mat. Iberoam.}, 30(3):875--892, 2014.

\bibitem{FShv18}
Charles Fefferman and Pavel Shvartsman.
\newblock Sharp finiteness principles for {L}ipschitz selections.
\newblock {\em Geom. Funct. Anal.}, 28:1641--1705, 2018.

\bibitem{G58}
Georges Glaeser.
\newblock Étude de quelques algèbres tayloriennes.
\newblock {\em J. Analyse Math.}, 6:1--124, 1958.

\bibitem{range2}
TNT Goodman, BH~Ong, and K~Unsworth.
\newblock Constrained interpolation using rational cubic splines.
\newblock {\em NURBS for curve and surface design}, pages 59--74, 1991.

\bibitem{range1}
Matthias Herrmann, Bernd Mulansky, and Jochen~W. Schmidt.
\newblock Scattered data interpolation subject to piecewise quadratic range
  restrictions.
\newblock {\em J. Comput. and Appl. Math.}, 13:209--223, 1996.

\bibitem{JL20-Ext}
Fushuai Jiang and Garving~K. Luli.
\newblock ${C^2(\mathbb{R}^2)}$ nonnegative interpolation by bounded-depth
  operators.
\newblock {\em Advances in Math.}, 375:107391, 2020.

\bibitem{JL20}
Fushuai Jiang and Garving~K. Luli.
\newblock Nonnegative $ {C}^2(\mathbb{R}^2)$ interpolation.
\newblock {\em Advances in Math.}, 375:107364, 2020.

\bibitem{JL20-Alg}
Fushuai Jiang and Garving~K. Luli.
\newblock Algorithms for nonnegative ${C^2(\mathbb{R}^2)}$ interpolation.
\newblock {\em Advances in Math.}, 385:107756, 2021.

\bibitem{St79}
Elias M.~Stein.
\newblock {\em Singular Integrals and Differentiability Properties of
  Functions}, volume~2 of {\em Monographs in harmonic analysis}.
\newblock Princeton University Press, Princeton, NJ, 1970.

\bibitem{range5}
Bernd Mulansky and Jochen~W. Schmidt.
\newblock {P}owell-{S}abin splines in range restricted interpolation of
  scattered data.
\newblock {\em Computing}, 53(2):137--154, 1994.

\bibitem{range3}
BH~Ong and K~Unsworth.
\newblock On non-parametric constrained interpolation.
\newblock In {\em Mathematical methods in computer aided geometric design II},
  pages 419--430. Elsevier, 1992.

\bibitem{range6}
Gerhard Opfer and Hans~Joachim Oberle.
\newblock The derivation of cubic splines with obstacles by methods of
  optimization and optimal control.
\newblock {\em Numerische Mathematik}, 52(1):17--31, 1987.

\bibitem{Shv82}
Pavel Shvartsman.
\newblock Traces of functions of two variables, which satisfy {Z}ygmund's
  condition.
\newblock In {\em Research on the Theory of Functions of Several Real Variables
  (Russian)}, pages 145--168. 1982.

\bibitem{Shv84}
Pavel Shvartsman.
\newblock {L}ipschitz sections of set-valued mappings and traces of functions
  from the {Z}ygmund class on an arbitrary compactum.
\newblock {\em Dokl. Akad. Nauk SSSR}, 276(3):559--562, 1984.

\bibitem{Shv86}
Pavel Shvartsman.
\newblock {L}ipschitz sections of multivalued mappings.
\newblock In {\em Studies in the Theory of Functions of Several Real Variables
  (Russian)}, pages 121--132. 1986.

\bibitem{Shv87}
Pavel Shvartsman.
\newblock Traces of functions of {Z}ygmund class.
\newblock {\em Sib. Math. J.}, 28:853--863, 1987.

\bibitem{Shv90}
Pavel Shvartsman.
\newblock K-functionals of weighted {L}ipschitz spaces and {L}ipschitz
  selections of multivalued mappings.
\newblock In {\em Interpolation spaces and related topics}, Israel Math. Conf.
  Proc., pages 245--268. 1990.

\bibitem{Shv02}
Pavel Shvartsman.
\newblock {L}ipschitz selections of set-valued mappings and {H}elly's theorem.
\newblock {\em J. Geom. Anal.}, 12(2):289--324, 2002.

\bibitem{Shv04}
Pavel Shvartsman.
\newblock Barycentric selectors and a {S}teiner-type point of a convex body in
  a {B}anach space.
\newblock {\em J. Geom. Anal.}, 210(1):1--42, 2004.

\bibitem{Shv08}
Pavel Shvartsman.
\newblock The {W}hitney extension problem and {L}ipschitz selections of
  set-valued mappings in jet-spaces.
\newblock {\em Trans. Amer. Math. Soc.}, 360(10):5529--5550, 2008.

\bibitem{Shv01}
Pavel Shvartsman.
\newblock On {L}ipschitz selections of affine-set valued mappings.
\newblock {\em Geom. Funct. Anal.}, 11(4):2001, 840-868.

\bibitem{W34-1}
Hassler {W}hitney.
\newblock Analytic extensions of differentiable functions defined in closed
  sets.
\newblock {\em Trans. Amer. Math. Soc.}, 36(1):63--89, 1934.

\bibitem{W34-2}
Hassler {W}hitney.
\newblock Differentiable functions defined in closed sets. {I}.
\newblock {\em Trans. Amer. Math. Soc.}, 36(2):369--387, 1934.

\bibitem{W34-3}
Hassler {W}hitney.
\newblock Functions differentiable on the boundaries of regions.
\newblock {\em Ann. of Math. (2)}, 35(3):482--485, 1934.

\bibitem{Z98}
Nahum Zobin.
\newblock {W}hitney's problem on extendability of functions and an intrinsic
  metric.
\newblock {\em Advances in Math.}, 133(1):96--132, 1998.

\bibitem{Z99}
Nahum Zobin.
\newblock Extension of smooth functions from finitely connected planar domains.
\newblock {\em J. Geom. Anal.}, 9(3):489--509, 1999.

\end{thebibliography}

\end{document}